\documentclass[12pt,a4paper, oneside,reqno]{article}

\usepackage[utf8]{inputenc}
\usepackage[english]{babel}

\usepackage[left=1in,right=1in,top=1in,bottom=1in]{geometry}

\usepackage{amsfonts,amssymb,amsmath,amsthm}
\usepackage{verbatim}

\usepackage[symbol]{footmisc} 
\usepackage{longtable}
\usepackage{tasks}

\usepackage{enumerate}
\usepackage{cite}
\usepackage{hyperref}
\usepackage{tkz-graph}

\usepackage{fancyhdr} 
\theoremstyle{plain}
\newtheorem{theorem}{Theorem}
\newtheorem{lemma}[theorem]{Lemma}
\newtheorem{proposition}[theorem]{Proposition}
\newtheorem{remark}[theorem]{Remark}

\newtheorem{definition}[theorem]{Definition}
\newtheorem{corollary}[theorem]{Corollary}

\usetikzlibrary{arrows}

\usepackage{fouriernc} 

\begin{document}


\bigskip

\noindent{\Large
Unital $3$-dimensional structurable algebras: \\ classification, properties and $\rm{AK}$-construction}
 \footnote{
The    work is supported by 
FCT   2023.08031.CEECIND, 2023.08952.CEECIND, and
UID/00212/2025.}

 \bigskip

\begin{center}

 {\bf
Kobiljon Abdurasulov\footnote{CMA-UBI, University of  Beira Interior, Covilh\~{a}, Portugal;  \ 
Romanovsky Institute of Mathematics, Academy of Sciences of Uzbekistan, Tashkent, Uzbekistan; \ abdurasulov0505@mail.ru},
Maqpal Eraliyeva\footnote{Romanovsky Institute of Mathematics, Academy of Sciences of Uzbekistan, Tashkent, Uzbekistan; \ eraliyevamaqpal@gmail.com} 
\&      
Ivan Kaygorodov\footnote{CMA-UBI, University of  Beira Interior, Covilh\~{a}, Portugal; \    kaygorodov.ivan@gmail.com}  }

\end{center}

\noindent {\bf Abstract:}
{\it  This paper is devoted to the classification and studying properties of complex unital $3$-dimensional structurable algebras. 
We provide a complete list of non-isomorphic classes, identifying five algebras  for type $(2, 1)$ and two algebras  for type $(1, 2).$ For each obtained algebra, we describe the derivation algebra, the automorphism group,  the lattice of subalgebras and ideals, and functional identities of degree $2$. 
Furthermore, we investigate the Allison-Kantor   construction for the classified algebras. We determine the structure of the resulting $\mathbb{Z}$-graded Lie algebras, providing their dimensions and Levi decompositions. }

 \medskip 


\medskip


	 \medskip

 
\tableofcontents

\newpage
\section*{Introduction}
Jordan algebras represent one of the most prominent classes of non-associative algebras.\footnote{Due to the many connections between associative and Lie algebras, we consider both as remaining outside the "non-associative world".} The first Jordan algebras appeared in a classical paper by Jordan, von Neumann, and Wigner in 1934 \cite{JVW}, and they have remained a subject of active investigation ever since.
 The most notable generalizations of Jordan algebras include: 
 non-commutative Jordan algebras, 
 terminal algebras, 
 commutative $\mathfrak{CD}$-algebras, 
 conservative and structurable algebras, among others. 
 The present paper is dedicated to the study of small-dimensional structurable algebras. 
 
 \medskip 
 
 A class of unital algebras with an involution, satisfying a specific quasi-identity depending on the fixed unit element and the involution, was introduced by Allison in 1978 \cite{A78}. 
 If this involution is the identity mapping, the definition of structurable algebras reduces to that of Jordan algebras. 
 The class of structurable algebras includes alternative algebras with an involution 
 $\big($for example, octonions, which are currently under active investigation \cite{L26,LR,CV25}$\big).$ 
 In his paper \cite{A78}, Allison provided important examples of simple structurable $\big($non-associative and non-Jordan$\big)$ algebras. 
 Arguably, the most significant example of non-trivial structurable algebras is constructed via the tensor product of two composition algebras. This construction is the subject of extensive research: Morandi, Pérez-Izquierdo, and Pumplün described derivations and automorphisms of tensor products of octonions in 2001 \cite{MPP01}; 
 Mondoc studied Kantor triple systems defined on tensor products of composition algebras \cite{M05,M06}; 
 Elduque and Okubo studied normal symmetric triality algebras, whose defining conditions reflect the properties of the tensor products of two symmetric composition algebras, in 2005/2007 \cite{O05,EO07}; 
 Blachar, Rowen, and Vishne studied identities of the tensor product of two octonions in 2023 \cite{BRV23}, among others. 
 Since its introduction, this class has been intensively studied.The class of structurable algebras is closely related to the class of conservative algebras introduced by Kantor in 1972 \cite{K72} and to various classes of ternary algebras. 
 A systematic study of structurable algebras began with the cited paper by Allison. 
 Isotopes of structurable algebras were studied by Allison and Hein in 1981 \cite{AH81}, and later by Allison and Faulkner in 1995 \cite{AF95}. 
 A Cayley-Dickson process for structurable algebras was investigated by Allison and Faulkner in 1984 \cite{AF84}. 
 Some fundamental results in the structure theory of structurable algebras were obtained by Schafer: he proved the Wedderburn principal theorem for structurable algebras in 1985 \cite{S85}; subsequently, he introduced structurable bimodules 
 $\big($using Eilenberg's general definition of a bimodule$\big)$ and proved a generalization of the first Whitehead lemma for Jordan algebras to the structurable case $\big($another generalization of the first Whitehead lemma for Jordan algebras was given in the context of $\delta$-derivations \cite{ZZ25}$\big),$ as well as the Malcev-Harish-Chandra theorem for Jordan algebras in \cite{S85-2}. 
 The third of his principal results concerning structurable algebras states that the radical of a finite-dimensional structurable algebra over a field of characteristic zero is nilpotent \cite{S86}. 
 Smirnov provided the complete classification of finite-dimensional simple and semisimple structurable algebras $\big($in the case of characteristic different from $2,3,5\big)$ in 1990 \cite{S90,S89}. 
 The case of characteristic equal to $5$ was considered by Stavrova in 2022 \cite{S22}. 
 Later, Pozhidaev, Shestakov, and Faulkner, in a series of papers (2010), classified all finite-dimensional simple structurable superalgebras over an algebraically closed field of characteristic zero \cite{F10,F102,PS}. 
 Structurable bimodules over algebras generated by skew-symmetric elements were studied by Smirnov in 1996 \cite{S96}. A systematic study of structurable algebras with skew-rank $1$ was initiated by Allison in 1990 \cite{A90} (see also \cite{DM19}). 
 Later, some interesting examples of structurable algebras of skew-rank $1$ were constructed in two papers by Pumplün in 2010/2011 \cite{P10,P11}. Cuypers and Meulewaeter proved the existence of a one-to-one correspondence (up to isomorphism) between simple structurable algebras of skew-dimension one (up to isotopy) and finite-dimensional Lie algebras generated by extremal elements that are not symplectic in 2021 \cite{CM21}. 
 Gradings on simple structurable algebras were actively studied by Elduque, Kochetov, Aranda-Orna, and Córdova-Martínez in 2014 and 2020 \cite{AEK14,AC20}. Kaygorodov and Okhapkina described all $\delta$-derivations of semisimple finite-dimensional structurable algebras from the Smirnov classification in 2014 \cite{KO}. 
 Elduque, Kamiya, and Okubo established a fundamental relation between left unital Kantor triple systems and structurable algebras in 2014 \cite{EKO14}. Structurable algebras play an important role in the theory of Moufang sets \cite{BD13,DBS19,DM24}. In particular, Boelaert and De Medts discovered a deep connection between quadrangular algebras and structurable algebras \cite{BD13}. A connection between structurable algebras and groups of type $E_6$ and $E_7$ was discussed in papers by Garibaldi and Alsaody in 2001/2024 \cite{G01,A24}. Let us also mention that in a series of papers by Kamiya, Mondoc, and Okubo, the theory of certain algebras generalizing structurable algebras $\big($such as pre-structurable, $\delta$-structurable, $(\alpha,\beta,\gamma)$-structurable, etc.$\big)$ was also studied \cite{KM13,KMO14,KO14}.\medskip 
 
 The present paper continues the study of structurable algebras and shifts the focus of investigation to small-dimensional algebras. 
 First, we obtain our principal result: the classification of $3$-dimensional structurable algebras $\big($Theorems \ref{Thcl1} and \ref{Thcl2}$\big)$. 
 The classification problem for non-associative algebras satisfying certain properties is one of the most central problems in non-associative algebra theory $\big($see \cite{akt, FKS25, FKS25-2, k23, MS, l24} and references therein$\big)$. 
 Second, we study the properties of the obtained algebras: we describe derivations (Section \ref{der}), automorphisms (Section \ref{aut}), and subalgebras (Section \ref{sub}). 
 The study of derivations and automorphisms is a classical problem in the theory of non-associative algebras; furthermore, the description of subalgebras provides a principal tool for future investigations of Rota-Baxter operators and other Rota-type operators on the obtained algebras. 
 Section \ref{ident} is dedicated to the description of involutive identities of unital $3$-dimensional structurable algebras. 
 The final Section presents an investigation of the Allison-Kantor construction for the classified algebras. We determine the structure of the resulting $\mathbb{Z}$-graded Lie algebras, providing their dimensions and Levi decompositions.

 \medskip

  \noindent  
{\bf Notations.}
In general, we are working with a complex field, but some results are correct for other fields.
We do not provide some well-known definitions
and refer the readers to consult previously published papers. 
For a set of vectors $S,$ we denote by  $\big\langle S \big\rangle$ the vector space generated by $S.$ 
The operator of left multiplication $L_x$ is defined by $L_x (y)=x\cdot y.$ 
We also   always  assume that the algebras  under our consideration are nontrivial  
$\big($i.e., they have nonzero multiplications$\big)$
and have an involution $\big($i.e., a linear map $^-$, such that $\overline{\overline{x}}=x$
and $\overline{x\cdot y}=\overline{y}\cdot \overline{x}\big).$
For a  set of some elements $\mathfrak{Set}$, we denote by $\overline{\mathfrak{Set}}$ the subset of elements that are invariant under the action of the involution $^-$
\ $\big($for example, 
$\overline{\mathfrak{Der}}$ is the set of derivations invariant under the involution, 
$\overline{\rm{Aut}}$ is the set of automorphisms  invariant under the involution, etc.$\big).$
All $n$-dimensional algebras with a basis $\{e_1, \ldots, e_n\}$ will be unital with unit $e_1,$
i.e., $e_1\cdot  e_i = e_i\cdot  e_1 = e_i,\     1 \leq i\leq n.$ 
Such multiplications $\big($and also zero multiplications$\big)$ will be omitted.
In the case of commutative and anticommutative algebras, the full multiplication table can be recovered due to  (anti)commutativity.

\newpage

\section{Classification and  properties} 
\subsection{Preliminaries}

\begin{definition}
    
 Let $\mathcal{A}$ be an algebra.  
A map $^- : \mathcal{A}\to \mathcal{A}$ is called an involution if, 
for all $x, y \in \mathcal{A}$, the following conditions hold:
\begin{center}
$\overline{\overline{x}} = x,  \quad\overline{x   y} = \overline{y}  \   \overline{x}.$
\end{center}
\end{definition}
For each   complex $n$-dimensional unital algebra $\mathcal{A}$ with involution $^-$, 
we have
\begin{center}
$\mathcal{A} = \mathcal{H} \oplus \mathcal{S},$
where 
$\mathcal{H} \ =\  \big\{ a \in \mathcal{A} \ | \ \overline{a} = a \big\}$ and 
$\mathcal{S} \ = \ \big\{ a \in \mathcal{A} \ | \ \overline{a} = -a\big\}$.
\end{center}If $\rm{dim}(\mathcal{H})=k$, then we call $\mathcal{A}$ an algebra of type $(k,n-k)$.

For $x,y\in\mathcal{A}$, define 
$V_{x,y}\in \rm{End}({\mathcal{A}})$ and $T_x=V_{x,e_1}$ as
\begin{longtable}{rcl}
$V_{x,y}(z) $&$=$&$ (x\overline{y})z + (z\overline{y})x-(z\overline{x})y,$\\
$T_x(z)$&$=$&$xz+zx-z\overline{x}.$
\end{longtable}
In  particular, if $a\in \mathcal{H}$, then $T_a=L_a.$  

\begin{definition}[see \cite{A78}] \label{identityunit}    
A unital algebra $\mathcal{A}$ with an involution $^-$ is structurable if it satisfies the following special identity
\begin{equation}\label{structurableunit}
T_{z} V_{x,y} - V_{x,y}T_{z}  \ =\  V_{{T_{z}(x),y}}-V_{{x,T_{\overline{z}}(y)}}.   
\end{equation}
   \end{definition}

Now, applying identity \eqref{structurableunit} to arbitrary elements $x, y, z$ and $t$, we obtain the following:
$$T_{z}V_{x,y}(t)-V_{x,y}T_z(t)=V_{{T_{z}(x),y}}(t)-V_{{x,T_{\overline{z}}(y)}}(t),$$
where 
\begin{longtable}{lclcl}
$T_{z}V_{x,y}(t)$&$=$&$z\big((x\overline{y})t\big)$&$+$&$z\big((t\overline{y})x\big)-z\big((t\overline{x})y\big)+\big((x\overline{y})t\big)z\ +$\\
&&&$+$&$\big((t\overline{y})x\big)z-\big((t\overline{x})y\big)z-\big((x\overline{y})t\big)\overline{z}-\big((t\overline{y})x\big)\overline{z}+\big((t\overline{x})y\big)\overline{z},$\\

$V_{x,y}T_z(t)$&$=$&$(x\overline{y})(zt)$&$+$&$(x\overline{y})(tz)-(x\overline{y})(t\overline{z})+\big((zt)\overline{y}\big)x\ +$\\
&&&$+$&$\big((tz)\overline{y}\big)x-\big((t\overline{z})\overline{y}\big)x-\big((zt)\overline{x}\big)y-\big((tz)\overline{x}\big)y+\big((t\overline{z})\overline{x}\big)y,$\\

$V_{{T_{z}(x),y}}(t)$&$=$&$\big((zx)\overline{y}\big)t$&$+$&$\big((xz)\overline y\big)t-\big((x\overline z)\overline y\big)t+(t\overline y)(zx)\ +$\\
&&&$+$&$(t\overline y)(xz)-(t\overline y)(x\overline z)-\big(t(\overline x\ \overline z)\big)y-\big(t(\overline z \ \overline x)\big)y+\big(t(z\overline x)\big)y,$\\

$V_{{x,T_{\overline{z}}(y)}}(t)$&$=$&$\big(x(\overline y z)\big)t$&$+$&$\big(x(z\overline y)\big)t-\big(x(\overline z\ \overline y)\big)t+\big(t(\overline yz)\big)x\ +$\\
&&&$+$&$\big(t(z\overline y)\big)x-\big(t(\overline z\ \overline y)\big)x-(t\overline x)(\overline zy)-(t \overline x)(y \overline z)+(t\overline x)(yz).$\\
\end{longtable}
Thus, identity \eqref{structurableunit} takes the following form:

${\rm ID}(x,y,z,t)\ :=$

$\Big(z\big((x \overline y)t\big)+z\big((t \overline y)x\big)-z\big((t\overline x)y\big)+\big((x\overline y)t\big)z+\big((t\overline y)x\big)z-\big((t\overline x)y\big)z-$
\begin{equation}\label{id_princ}
    \big((x\overline y)t\big)\overline z-\big((t\overline y)x\big)\overline z+\big((t\overline x)y\big)\overline z-(x\overline y)(zt)-(x\overline y)(tz)+(x\overline y)(t\overline z)-
\end{equation}

\begin{flushright}$\big((zt)\overline y\big)x-\big((tz)\overline y\big)x+\big((t\overline z)\overline y\big)x+\big((zt)\overline x\big)y+\big((tz)\overline x\big)y-\big((t\overline z)\overline x\big)y\ \Big)-$
\end{flushright}

$\Big(    \big((zx)\overline y\big)t+\big((xz)\overline y\big)t-\big((x\overline z)\overline y\big)t+(t\overline y)(zx)+(t\overline y)(xz)-(t\overline y)(x\overline z) -
$

\begin{center}
$\big(t(\overline x\ \overline z)\big)y-\big(t(\overline z\  \overline x)\big)y+\big(t(z\overline x)\big)y-\big(x(\overline y z)\big)t-\big(x(z\overline y)\big)t+\big(x(\overline z\ \overline y)\big)t -$
\end{center}

\begin{flushright}$\big(t(\overline yz)\big)x-\big(t(z\overline y)\big)x+\big(t(\overline{z}\ \overline y)\big)x+(t\overline x)(\overline zy)+(t \overline x)(y \overline z)-(t\overline x)(yz) \Big) \ = \ 0.$
\end{flushright}

\begin{definition}
    Let $(\mathcal{A},\cdot, ^-)$ and $(\mathcal{B},\star, \wideparen{ \ } \ )$ be structurable algebras.
$\varphi : \mathcal{A} \to \mathcal{B}$
is called an isomorphism of structurable algebras if it satisfies
\begin{center}
 $\varphi(x\cdot y)=\varphi(x) \star \varphi(y)$ \ and \  
 $\varphi(\overline{x})= \wideparen{\varphi(x)}$.
\end{center}
In this case,$(\mathcal{A},\cdot, ^-)$ and $(\mathcal{B},\star, \wideparen{ \ } \ )$ are said to be isomorphic, and we write $\mathcal{A} \cong \mathcal{B}.$
\end{definition}

\subsection{Classification of $3$-dimensional structurable algebras}

\begin{definition}
An $n$-dimensional algebra with a multiplication table  
$e_1e_i=e_ie_1=e_i$ for $1\leq i\leq n,$ 
and the involution \ $^-$ \   given as 
$\overline{e_{i_1}} = e_{i_1}$ for $1\leq i_1\leq k$ and 
$\overline{e_{i_2}} = -e_{i_2}$ for $k+1\leq i_2\leq n$ and 
is called a
 universal unital algebra of type $(k,n-k).$
\end{definition}
Obviously, each  universal unital algebra of type $(k,n-k)$ is structurable.

\begin{remark}
    Let $\mathcal{A}$ be a complex $3$-dimensional unital structurable algebra of type $(3,0)$. 
Then, it is a Jordan algebra with the identical involution, and according to {\rm  \cite{G}}, it is isomorphic  
to the universal unital algebra ${\rm J}_1$ of type $(3,0)$, 
or it is isomorphic to one of the following algebras:
    \begin{longtable}{lllllllllll}
${\rm J}_2$ & $:$&$ e_2e_2$&$=$&$e_2$&$e_3 e_3$&$=$&$e_3$ \\
${\rm J}_3$ &$:$&$ e_2e_2$&$=$&$e_2$&$  e_2 e_3$&$=$&$e_3$&$  e_3 e_2$&$=$&$e_3$ \\

${\rm J}_4$ &$:$&$e_2e_2$&$=$&$e_3$\\
${\rm J}_5$ &$:$&$e_2e_2$&$=$&$e_2$\\

${\rm J}_6$ &$:$&$e_2e_2$&$=$&$e_1$&$ e_3e_3$&$=$&$e_1$\\
\end{longtable}
\end{remark}

Below, we consider two nontrivial situations of structurable algebras with $\mathcal{S}\neq0,$
i.e., the involution will be non-identical. 
The following Lemma is trivial and well-known, but it will be very useful in our considerations.

\begin{lemma}\label{automor} Let $\mathcal{A}$ be a unital algebra and let 
$\varphi \in {\rm Aut} (\mathcal{A}),$ then $\varphi(e_1)=e_1.$
\end{lemma}

\subsubsection{Type $(2,1)$}
Here  we consider  $3$-dimensional unital structurable algebras of type $(2,1)$. 
Let $\{e_1,e_2,e_3\}$ be a basis such that $e_1,e_2 \in \mathcal{H}$ and $e_3\in \mathcal{S}.$

\begin{theorem}\label{Thcl1}
Let  $\mathcal{A}$ be a  complex unital structurable algebra $\mathcal{A}$ of type $(2,1),$ then it is isomorphic to the universal unital algebra ${\rm A}_1$ of type $(2,1),$ or it is isomorphic to an algebra
with   an involution\  $^-$, given by $\overline{e_1}=e_1,$ 
$\overline{e_2}=e_2,$ 
$\overline{e_3}=-e_3,$ and the multiplication table given by one of the following:
    \begin{longtable}{lllllllllll}

${\rm A}_2:$&$e_3 e_3=e_2$  && \\ 
    
${\rm A}_3:$&$ e_2  e_2=e_2$\\ 
    
${\rm A}_4:$&$e_2  e_2=e_2 $&$ e_3  e_3=-e_1+e_2$\\

${\rm A}_5:$&$e_2  e_3=e_2$&$ e_3  e_2=-e_2$&$e_3  e_3=e_1$\\

\end{longtable}

\end{theorem}

\begin{proof}
It is easy to see that
\begin{equation}\label{(2,1)}
    \overline{e_2  e_2}=\overline{e_2} \  \overline{e_2}=e_2  e_2, \ 
    \overline{e_3  e_3}=\overline{ e_3}\   \overline{ e_3}= (-e_3)  (-e_3)=e_3 e_3, \ \overline{e_2 e_3}=\overline{ e_3 }\  \overline{e_2}=-e_3  e_2.\end{equation}

Using the relations \eqref{(2,1)}, we obtain the following table of multiplication in the algebra:
\begin{longtable}{lcrlcr}
$e_2  e_2$&$=$&$  \alpha_1 e_1+\beta_1e_2,$&$e_2 e_3 $&$= $&$ \alpha_2 e_1+\beta_2e_2+\gamma e_3,$\\ 
$e_3  e_3$&$=$&$\alpha_3 e_1+\beta_3e_2,$ &$
e_3  e_2$&$  =$&$ -\alpha_2e_1-\beta_2e_2+\gamma e_3.$
\end{longtable}


\noindent 
By choosing the new basis elements 
$e_1':=e_1, \ e_2':=e_2-\gamma e_1, \ e_3':=e_3,$ we can assume $\gamma\ =\ 0.$

Applying the identity \eqref{id_princ}  to some useful elements, we have 

\begin{longtable}{lcl}
  ${\rm ID}(e_3,e_1,e_2,e_2)$&$=$&$-2 \alpha_1 e_3,$ i.e.   $\alpha_1=0;$\\
  ${\rm ID}(e_1,e_2,e_3,e_3)$&$=$&$6 \alpha_2 \beta_2 e_1+8 \alpha_2 e_3+6\big(\beta_2^2-\alpha_3-\beta_1 \beta_3\big) e_2,$ i.e.   $\alpha_2=0,$  $\alpha_3=\beta_2^2-\beta_1\beta_3;$\\
  ${\rm ID}(e_1,e_2,e_3,e_2)$&$=$&$ -4\beta_1\beta_2 e_2,$ i.e.   $\beta_1\beta_2=0;$\\
  ${\rm ID}(e_3,e_3,e_3,e_1)$&$=$&$ 8\beta_2\beta_3 e_2,$ i.e.   $\beta_2\beta_3=0.$\\

\end{longtable}

\noindent
Now we have the following parametric family of algebras:
\begin{longtable}{lcllcl}
$e_2 e_2 $&$=$&$ \beta_1e_2,$  & $e_2 e_3 $&$=$&$ \beta_2e_2,$  \\
$e_3 e_3 $&$=$&$\big(\beta_2^2-\beta_1\beta_3\big) e_1+\beta_3e_2,$ &$e_3 e_2  $&$=$&$-\beta_2e_2.$
\end{longtable}

\noindent 
By choosing the new basis elements 
$e_1:=e_1,$ $e_2:=Ae_2,$ $e_3:=Ce_3,$ we can assume

\begin{center}
$e_2 e_2 \ = \ A\beta_1 e_2,$ \  \ 
$e_2 e_3 \ =\  C\beta_2e_2 \ = \ - e_3e_2,$ \ \ 
$e_3 e_3\ =\ C^2\big(\beta_2^2-\beta_1\beta_3\big) e_1+{C^2}{A^{-1}}\beta_3e_2.$
\end{center}

The condition $\beta_1\beta_2=\beta_2\beta_3=0$ leads us to the following cases.

\begin{enumerate}[(1)]
\item If $\beta_1=\beta_2=\beta_3=0,$ then we have the algebra $\rm A_1.$ 
\item If $\beta_1=\beta_2=0$ and $\beta_3\neq 0,$ then by choosing $A=\beta_3$ and $C=1,$ we obtain   $\rm A_2.$

\item If $\beta_1\neq 0$ and $\beta_2=\beta_3= 0,$ then by choosing $A= {\beta^{-1}_1},$ we obtain   $\rm A_3.$ 
\item If $\beta_1\neq0,$ $\beta_2=0,$ and $\beta_3\neq 0,$ then by choosing $A={\beta^{-1}_1}$ and $ C=\frac 1{\sqrt{\beta_1\beta_3}}$, we obtain   $\rm A_4.$ 

 \item If $\beta_1=\beta_3=0$ and  $\beta_2\ne 0$  then by choosing  $C={\beta^{-1}_2},$ we obtain   $\rm A_5.$ 

 \end{enumerate}
\end{proof}

\subsubsection{Type $(1,2)$}
Here  we consider  $3$-dimensional unital structurable algebras of type $(1,2)$. 
Let $\{e_1,e_2,e_3\}$ be a basis such that $e_1 \in \mathcal{H}$ and $e_2,e_3\in \mathcal{S}.$

\begin{theorem}\label{Thcl2}
Let  $\mathcal{A}$ be a  complex unital structurable algebra $\mathcal{A}$ of type $(1,2),$ then it is isomorphic to the universal unital algebra ${\rm S}_1$ of type $(1,2),$ or it is isomorphic to an algebra
with   an involution\  $^-\ ,$ given by $\overline{e_1}=e_1,$ 
$\overline{e_2}=-e_2,$ 
$\overline{e_3}=-e_3,$ and the multiplication table given by one of the following:

\begin{longtable}{llllllllllllllll}
 ${\rm S}_2$&$: $&$ e_2 e_3=e_2$&$e_3 e_2=  -e_2$&$e_3 e_3=e_1$ 
    \end{longtable}
\end{theorem}

\begin{proof}
It is easy to see that  
$\overline{e_2  e_3}=\overline{e_3} \  \overline{e_2}=e_3 e_2,$ \ 
$\overline{e_2  e_2}=e_2e_2$ and 
$\overline{e_3  e_3}= e_3e_3,$
hence
\begin{longtable}{lcllcl}
$e_2 e_2 $&$=$&$\alpha_1 e_1$ & $e_2 e_3  $&$=$&$  \alpha_2 e_1+\beta e_2+\gamma e_3$\\ 
$ e_3e_3$&$=$&$\alpha_3 e_1$&$ e_3 e_2 $&$=$&$ \alpha_2 e_1-\beta e_2-\gamma e_3$
\end{longtable}

Applying the identity \eqref{id_princ}  to some useful elements, we have 

\begin{longtable}{lcl}
  ${\rm ID}(e_1,e_2,e_3,e_2)$&$=$&$6 (\alpha_2+\beta \gamma) e_2-6 (\alpha_1- \gamma^2) e_3,$ i.e.   $\alpha_1=\gamma^2$ and $\alpha_2=-\beta\gamma;$\\

  ${\rm ID}(e_1,e_2,e_3,e_3)$&$=$&$-6 (\alpha_3-\beta^2) e_2,$ i.e.   $\alpha_3=\beta^2.$

  \end{longtable}


Now we have the following parametric family of algebras:
\begin{longtable}{lcllcl}
$e_2 e_2$&$= $&$\gamma^2e_1$&$e_2 e_3$&$ =$&$ -\beta \gamma e_1+\beta e_2+\gamma e_3$\\ 
$e_3 e_3$&$=$&$\beta^2e_1$ & $ e_3 e_2 $&$=$&$ -\beta \gamma e_1-\beta e_2-\gamma e_3$
\end{longtable}

In this case, the conditions $(\beta,\gamma)=(0,0)$ and $(\beta,\gamma)\neq (0,0)$ are invariant.
If $(\beta,\gamma)=(0,0)$ then we have the algebra $\rm S_1$. 

If $(\beta,\gamma)\neq (0,0)$ then, without loss of generality, 
we may assume $\beta\neq0$. By taking the new basis elements
$e_1'=e_1,$ 
$e_2'=e_2+\gamma\beta^{-1}e_3,$ and 
$e_3'={\beta^{-1}}e_3$, we have  the algebra $\rm S_2$.
\end{proof}

\subsection{Properties  of unital $3$-dimensional structurable algebras} 

\subsubsection{Derivations}\label{der}

\begin{definition}
Let $\mathcal{A}$ be a structurable algebra. A linear map $D : \mathcal{A} \to \mathcal{A}$
is called a derivation if
$D(xy)\ =\ D(x)y+xD(y).$
If a derivation of $\mathcal A$ satisfies  $D(x)\ =\ \overline{D(\overline{x})},$ then it is called a derivation commuting with the involution 
$\big($or $\overline{derivation}\big)$.
The set of all derivations of $\mathcal{A}$ and the set of all $\overline{derivations}$ are  denoted by 
${\mathfrak{Der}}(\mathcal{A})$
and $\overline{{\mathfrak{Der}}}(\mathcal{A}).$
\end{definition}


\begin{proposition}
Derivations and $\overline{derivations}$ of complex unital $3$-dimensional  structurable algebras are given below:

\begin{longtable}{|lcl|}
  \hline 
  
${\mathfrak{Der}}({\rm A}_1)$&$=$&$\big\langle d_1,d_2,d_3,d_4 \ | \ 
d_1(e_2)=e_2; \ 
d_2(e_3)=e_3; \ 
d_3(e_2)=e_3; \ 
d_4(e_3)=e_2\big\rangle$\\

$\overline{{\mathfrak{Der}}}({\rm A}_1)$&$=$&$\big\langle d_1,d_2 \ | \ 
d_1(e_2)=e_2; \ 
d_2(e_3)=e_3\big\rangle$\\
\hline 

${\mathfrak{Der}}({\rm A}_2)$&$=$&$\big\langle d_1, d_2 \ | \  
d_1(e_2)=2e_2, \ 
d_1(e_3)=e_3; \ 
d_2(e_3)=e_2 \big\rangle$\\

$\overline{{\mathfrak{Der}}}({\rm A}_2)$&$=$&$\big\langle d_1 \ | \  
d_1(e_2)=2e_2, \ d_1(e_3)=e_3\big\rangle$\\

\hline 
  
${\mathfrak{Der}}({\rm A}_3)$&$=$&$\big\langle d_1 \ | \   d_1(e_3)=e_3\big\rangle$\\

$\overline{\mathfrak{Der}}({\rm A}_3)$&$=$&$\big\langle d_1 \ | \   d_1(e_3)=e_3\big\rangle$\\

\hline 
  
${\mathfrak{Der}}({\rm A}_4)$&$=$&$\big\langle 0\big\rangle$\\
$\overline{\mathfrak{Der}}({\rm A}_4)$&$=$&$\big\langle 0\big\rangle$\\

\hline 
  
${\mathfrak{Der}}({\rm A}_5)$&$=$&$\big\langle d_1, d_2 \ | \  
d_1(e_2)=e_2; \ 
d_2(e_3)=e_2 \big\rangle$\\

$\overline{\mathfrak{Der}}({\rm A}_5)$&$=$&$\big\langle d_1 \ | \  d_1(e_2)=e_2\big\rangle$\\

\hline 
  
${\mathfrak{Der}}({\rm S}_1)$&$=$&$\big\langle d_1,d_2,d_3,d_4 \ | \ 
d_1(e_2)=e_2; \ 
d_2(e_2)=e_3; \ 
d_3(e_3)=e_2; \  
d_4(e_3)=e_3\big\rangle$\\

$\overline{\mathfrak{Der}}({\rm S}_1)$&$=$&$\big\langle d_1,d_2,d_3,d_4 \ | \ 
d_1(e_2)=e_2; \ 
d_2(e_2)=e_3; \ 
d_3(e_3)=e_2; \  
d_4(e_3)=e_3\big\rangle$\\

\hline 

${\mathfrak{Der}}({\rm S}_2)$&$=$&$\big\langle d_1,d_2 \ | \ 
d_1(e_2)=e_2; \ 
d_2(e_3)=e_2\big\rangle$\\

$\overline{\mathfrak{Der}}({\rm S}_2)$&$=$&$\big\langle d_1,d_2 \ | \ 
d_1(e_2)=e_2; \ 
d_2(e_3)=e_2\big\rangle$\\
\hline 
  
\end{longtable}

\end{proposition}

\subsubsection{Automorphisms}\label{aut}
\begin{definition}
Let $\mathcal{A}$ be a structurable algebra. A linear map $\varphi : \mathcal{A} \to \mathcal{A}$
is called an automorphism   if
$\varphi(xy)\ =\ \varphi(x)\varphi(y).$
If an automorphism  of $\mathcal A$ satisfies  $\varphi(x)\ =\ \overline{\varphi(\overline{x})},$ then it is called an automorphism   commuting with the involution $\big($or $\overline{automorphism}\big)$.
The set of all automorphisms of $\mathcal{A}$ and the set of all $\overline{automorphisms}$ are  denoted by 
${\rm{Aut}}(\mathcal{A})$
and $\overline{{\rm{Aut}}}(\mathcal{A}).$
\end{definition}

\begin{proposition}\label{Aut} 
Automorphisms and $\overline{automorphisms}$ of complex unital $3$-dimensional  structurable algebras are given below
$\big($it has to be mentioned that all matrices given below are non-degenerate$\big)$:

\begin{longtable}{|lcc|lcc|}
  \hline 

$\mathrm{Aut}({\rm A}_1)$&$=$&$\begin{pmatrix} 
1 & 0& 0 \\ 
0 & \alpha & \gamma  \\ 
0 & \delta & \beta  \end{pmatrix}  $&$  \overline{\mathrm{Aut}}({\rm A}_1)$&$=$&$\begin{pmatrix} 
1 & 0& 0 \\ 
0 & \alpha & 0 \\ 
0 & 0 & \beta  \end{pmatrix}$\\

  \hline 

$\mathrm{Aut}({\rm A}_2)$&$=$&$\begin{pmatrix} 
1 & 0& 0 \\ 
0 & \alpha^2 & \beta \\ 
0 & 0 & \alpha  \end{pmatrix}$&$\overline{\mathrm{Aut}}({\rm A}_2)$&$=$&$\begin{pmatrix}
1 & 0& 0 \\ 0 & \alpha^2 & 0 \\ 
0 & 0 & \alpha
\end{pmatrix}$\\

  \hline 

$\mathrm{Aut}({\rm A}_3)$&$=$&$\begin{pmatrix} 1 & 0& 0 \\ 
0 & \alpha & 0 \\ 
0& 0 & \beta  \end{pmatrix}$&$\overline{\mathrm{Aut}}({\rm A}_3)$&$=$&$\begin{pmatrix}
1 & 0& 0 \\ 
0 & \alpha  & 0 \\ 
0& 0 & \beta  \end{pmatrix}$\\ 
  \hline

$\mathrm{Aut}({\rm A}_4)$&$=$&$\begin{pmatrix} 1 & 0& 0 \\ 0 & 1 & 0 \\ 0 & 0 & \pm1 \end{pmatrix}$&$\overline{\mathrm{Aut}}({\rm A}_4)$&$=$&$\begin{pmatrix} 1 & 0& 0 \\ 0 & 1 & 0 \\ 0 & 0 & \pm1 \end{pmatrix}$\\
  \hline

$\mathrm{Aut}({\rm A}_5)$&$=$&$\begin{pmatrix} 
1 & 0& 0 \\ 
0 & \alpha & \beta \\
0 & 0 & 1\end{pmatrix}$&$\overline{\mathrm{Aut}}({\rm A}_5)$&$=$&$\begin{pmatrix} 
1 & 0& 0 \\ 
0 & \alpha & 0 \\ 
0 & 0 & 1  \end{pmatrix}$\\
  \hline

$\mathrm{Aut}({\rm S}_1)$&$=$&$\begin{pmatrix} 1 & 0& 0 \\ 
0 & \alpha & \gamma  \\ 
0 & \delta & \beta \end{pmatrix}  $&$  \overline{\mathrm{Aut}}({\rm S}_1)$&$=$&$\begin{pmatrix} 1 & 0& 0 \\ 
0 & \alpha & \gamma  \\ 
0 & \delta & \beta \end{pmatrix}$\\
  \hline

$\mathrm{Aut}({\rm S}_2)$&$=$&$\begin{pmatrix} 
1 & 0& 0 \\ 
0 & \alpha & \beta \\
0 & 0 & 1
\end{pmatrix}$&$\overline{\mathrm{Aut}}({\rm S}_2)$&$=$&$\begin{pmatrix} 
1 & 0& 0 \\ 
0 & \alpha & \beta \\
0 & 0 & 1\end{pmatrix}$\\
  \hline

\end{longtable}
\end{proposition}

\subsubsection{Subalgebras of  unital $3$-dimensional structurable algebras}\label{sub}

\medskip \noindent{\it Subalgebras of  ${\rm A}_1.$}

\begin{theorem}\label{subA1}
Let $\mathfrak{s}$ be a one-dimensional subalgebra of $\rm{A}_1,$ then $\mathfrak{s}$ is one of the following subalgebras:
$\big\langle e_1 \big\rangle,$ 
$\big\langle e_2 \big\rangle,$ or 
$\big\langle \alpha  e_2 + e_3 \big\rangle_{\alpha \in \mathbb C}.$
Up to isomorphisms, all one-dimensional subalgebras of $\rm{A}_1$ are equivalent to  $\big\langle e_1 \big\rangle $ or $ \big\langle e_2 \big\rangle.$ 

\end{theorem}
\begin{proof}
Let $\mathfrak{s}=\big\langle \alpha e_1 + \beta e_2 + \gamma e_3 \big\rangle .$
Consider the product:
\[
\big(\alpha e_1 + \beta e_2 + \gamma e_3\big)
\big(\alpha e_1 + \beta e_2 + \gamma e_3\big)
\ =\  k\big(\alpha e_1 + \beta e_2 + \gamma e_3\big),
\]
which leads to
\[
\alpha^2 e_1 + 2\alpha\beta e_2 + 2\alpha\gamma e_3
\ =\  k\big(\alpha e_1 + \beta e_2 + \gamma e_3\big),
\]
i.e., $\alpha^2 \ =\  k\alpha,$ \ $2\alpha\beta \ =\  k\beta,$ \ and $2\alpha\gamma \ =\  k\gamma.$ 
Let us consider the following cases:

\begin{enumerate}
\item[(1)] 
If $k = 0$, then $\alpha = 0$, and we obtain two types of subalgebras
$\big\langle e_2 \big\rangle$ and
$\big\langle \alpha  e_2 + e_3 \big\rangle_{\alpha \in \mathbb C}.$

\item[(2)] If $k \neq 0$ and $\alpha = 0$, then $\beta = 0$ and $\gamma = 0$, which is a contradiction.

\item[(3)] If $k \neq 0$ and $\alpha \neq 0$, then $\alpha = k$ and $\beta = 0$, $\gamma = 0$, which yields the subalgebra
$\langle e_1 \rangle .$
\end{enumerate}
Taking automorphisms of $\rm{A}_1$ from Proposition \ref{Aut}, 
we have the second part of our statement.
\end{proof}
\begin{corollary}

Let $\mathfrak{s}$ be a one-dimensional $\overline{\mbox{subalgebra}}$ of $\rm{A}_1,$ then $\mathfrak{s}$ is one of the following $\overline{\mbox{subalgebras}}$:
$\big\langle e_1 \big\rangle,$ 
$\big\langle e_2 \big\rangle,$ or 
$\big\langle   e_3 \big\rangle.$
Up to isomorphisms, all one-dimensional $\overline{\mbox{subalgebras}}$ of $\rm{A}_1$ are equivalent to  $\big\langle e_1 \big\rangle,$
$\big\langle e_2 \big\rangle$ or
$\big\langle e_3 \big\rangle.$
\end{corollary}

\begin{corollary} One-dimensional ideals of $\rm{A}_1$ are
$\big\langle e_2 \big\rangle$ and  $\big\langle \alpha e_2+e_3 \big\rangle_{\alpha \in \mathbb C};$ 
one-dimensional   $\overline{\mbox{ideals}}$ of $\rm{A}_1$ are   $\big\langle   e_2 \big\rangle$ and  $\big\langle e_3 \big\rangle.$
\end{corollary}

\begin{theorem}\label{sub2A1}
Let $\mathfrak{s}$ be a two-dimensional subalgebra of $\rm{A}_1,$ then $\mathfrak{s}$ is one of the following subalgebras:
$\big\langle e_2, e_3\big\rangle,$\ 
$\big\langle e_1, e_2\big\rangle$, or 
$\big\langle e_1, \alpha e_2 + e_3 \big\rangle_{\alpha \in \mathbb C}.$
Up to isomorphisms, all two-dimensional subalgebras of $\rm{A}_1$ are equivalent to  
$\langle e_1  ,e_2\rangle$ or 
$\langle e_2 ,e_3\rangle.$

\end{theorem}
\begin{proof}
Obviously, that $\big\langle e_2, e_3\big\rangle$ is a subalgebra with zero multiplication.  
Let 
\begin{center}
$\mathfrak{s}=\big\langle e_1 + \alpha_2 e_2+\alpha_3 e_3,\  \beta_2e_2+ \beta_3 e_3 \big\rangle.$
\end{center}
Consider the product:
\begin{center}
$\big(e_1 + \alpha_2 e_2+\alpha_3 e_3\big)\big(e_1+ \alpha_2 e_2+\alpha_3 e_3\big)
=  e_1 + 2\big(\alpha_2 e_2+\alpha_3 e_3\big)
=k_1\big(e_1 + \alpha_2 e_2+\alpha_3 e_3)+k_2(\beta_2e_2+ \beta_3 e_3\big),$
\end{center}
i.e., $k_1=1,$ $\alpha_2=k_2\beta_2,$ and  $\alpha_3=k_2\beta_3.$
We have that $\mathfrak{s}=\big\langle e_1,\  \beta_2e_2+ \beta_3 e_3 \big\rangle.$
 This gives  the first part of our statement.
Taking automorphisms of $\rm{A}_1$ from Proposition \ref{Aut}, 
we have the second part of our statement.
\end{proof}
\begin{corollary}

Let $\mathfrak{s}$ be a two-dimensional $\overline{\mbox{subalgebra}}$ of $\rm{A}_1,$ then $\mathfrak{s}$ is one of the following $\overline{\mbox{subalgebras}}$:
$\big\langle e_1,e_2 \big\rangle,$ 
$\big\langle e_1,e_3 \big\rangle,$ or 
$\big\langle e_2,e_3\big\rangle.$  
Up to isomorphisms, all two-dimensional $\overline{\mbox{subalgebras}}$ of $\rm{A}_1$ are equivalent to  $\big\langle e_1,e_2 \big\rangle,$ 
$\big\langle e_1,e_3 \big\rangle$ or 
$\big\langle e_2,e_3\big\rangle.$  
\end{corollary}

\begin{corollary} 
There is only one two-dimensional ideal of $\rm{A}_1,$ 
it is an $\overline{\mbox{ideal}}$ and equal to  $\big\langle e_2 ,e_3\big\rangle.$
\end{corollary}

\medskip \noindent{\it Subalgebras of  ${\rm A}_2.$} 

\begin{theorem}
Let $\mathfrak{s}$ be a one-dimensional subalgebra of $\rm{A}_2,$ then $\mathfrak{s}$ is one of the following subalgebras:
$\big\langle e_1 \big\rangle$ or $\big\langle e_2 \big\rangle.$ 
Up to isomorphisms, all one-dimensional subalgebras of $\rm{A}_2$ are equivalent to  $\big\langle e_1 \big\rangle$ or
$\big\langle e_2 \big\rangle.$

\end{theorem}
\begin{proof}
Let $\mathfrak{s}=\big\langle \alpha e_1 + \beta e_2 + \gamma e_3 \big\rangle .$
Consider the product:
\[
\big(\alpha e_1 + \beta e_2 + \gamma e_3\big)
\big(\alpha e_1 + \beta e_2 + \gamma e_3\big)
\ =\  k\big(\alpha e_1 + \beta e_2 + \gamma e_3\big),
\]
which leads to
\[
\alpha^2 e_1 + \gamma^2e_2+2\alpha\beta e_2 + 2\alpha\gamma e_3
\ =\  k\big(\alpha e_1 + \beta e_2 + \gamma e_3\big),
\]
i.e., $\alpha^2 \ =\  k\alpha,$ \ $\gamma^2+2\alpha\beta \ =\  k\beta,$ \ and $2\alpha\gamma \ =\  k\gamma.$ 
Let us consider the following cases:

\begin{enumerate}
\item[(1)] 
If $k = 0$, then $\alpha = 0,\  \gamma=0$, and we obtain the subalgebra
$\big\langle e_2 \big\rangle$. 
\item[(2)] If $k \neq 0$ and $\alpha = 0$, then $\beta = 0$ and $\gamma = 0$, which is a contradiction.

\item[(3)] If $k \neq 0$ and $\alpha \neq 0$, then $\alpha = k$ and $\beta = 0$, $\gamma = 0$, which yields the subalgebra
$\big\langle e_1 \big\rangle .$
\end{enumerate}
Taking automorphisms of $\rm{A}_2$ from Proposition \ref{Aut}, 
we have the second part of our statement.
\end{proof}
\begin{corollary}

Let $\mathfrak{s}$ be a one-dimensional $\overline{\mbox{subalgebra}}$ of $\rm{A}_2,$ then $\mathfrak{s}$ is one of the following $\overline{\mbox{subalgebras}}$:
$\big\langle e_1 \big\rangle$ or  
$\big\langle e_2 \big\rangle.$ 
Up to isomorphisms, all one-dimensional $\overline{\mbox{subalgebras}}$ of $\rm{A}_2$ are equivalent to  $\big\langle e_1 \big\rangle$ or
$\big\langle e_2 \big\rangle.$ 
\end{corollary}
\begin{corollary}
There is only one one-dimensional ideal of $\rm{A}_2,$ 
it is an $\overline{\mbox{ideal}}$ and equal to    $\big\langle   e_2 \big\rangle.$
\end{corollary}

\begin{theorem}
Let $\mathfrak{s}$ be a two-dimensional subalgebra of $\rm{A}_2,$ then $\mathfrak{s}$ is one of the following subalgebras:
$\big\langle e_1 ,e_2\big\rangle$ or 
$\big\langle e_2,e_3 \big\rangle.$
Up to isomorphisms, all two-dimensional subalgebras of $\rm{A}_2$ are equivalent to  
$\big\langle e_1,e_2\big\rangle$ or 
$\big\langle e_2,e_3\big\rangle.$

\end{theorem}
\begin{proof}

Obviously, that $\big\langle e_2, e_3\big\rangle$ is a subalgebra of  $\rm{A}_2$.  
Let 
\begin{center}
$\mathfrak{s}=\big\langle e_1 + \alpha_2 e_2+\alpha_3 e_3,\  \beta_2e_2+ \beta_3 e_3 \big\rangle.$
\end{center}
Consider the product:
\begin{center}
$\big(e_1 + \alpha_2 e_2+\alpha_3 e_3 \big)\big(e_1 + \alpha_2 e_2+\alpha_3 e_3 \big)=
e_1 + (2\alpha_2 +\alpha_3^2)e_2+2\alpha_3 e_3 = 
k_1(e_1 + \alpha_2 e_2+\alpha_3 e_3)+k_2(\beta_2e_2+ \beta_3 e_3),$
\end{center}
i.e., $k_1=1,$ $\alpha_2+\alpha_3^2=k_2\beta_2,$ and $\alpha_3=k_2\beta_3;$
and $\mathfrak{s}=\big\langle e_1 - \alpha_3^2 e_2,\  \beta_2e_2+ \beta_3 e_3 \big\rangle.$
If $\beta_3\neq0,$ then 
$\big(\beta_2e_2+ \beta_3 e_3 \big)\big(\beta_2e_2+ \beta_3 e_3 \big)=\beta_3^2 e_2 \in \mathfrak{s},$
i.e., $e_2 \in \mathfrak{s}$ and $e_3 \in \mathfrak{s},$ which gives a contradiction.
Hence, $\beta_3=0$ and 
$\big\langle e_1,   e_2 \big\rangle$ gives one more two-dimensional subalgebra of $\rm A_2$
and the first part of our statement.

Taking automorphisms of $\rm{A}_2$ from Proposition \ref{Aut}, 
we have the second part of our statement.
\end{proof}
\begin{corollary}
Let $\mathfrak{s}$ be a two-dimensional $\overline{\mbox{subalgebra}}$ of $\rm{A}_2,$ then $\mathfrak{s}$ is one of the following 
$\big\langle e_1 ,e_2\big\rangle$ or
$\big\langle e_2 ,e_3\big\rangle.$  
Up to isomorphisms, all two-dimensional $\overline{\mbox{subalgebras}}$ of $\rm{A}_2$ are equivalent to  $\big\langle e_1 ,e_2\big\rangle$ or
$\big\langle e_2 ,e_3\big\rangle.$  
\end{corollary}

\begin{corollary} There is only one two-dimensional ideal of $\rm{A}_1,$ 
it is an $\overline{\mbox{ideal}}$ and equal to  $\big\langle e_2, \ e_3\big\rangle.$
\end{corollary}

\medskip \noindent{\it Subalgebras of  ${\rm A}_3.$ }

\begin{theorem}\label{A3subd1}
Let $\mathfrak{s}$ be a one-dimensional subalgebra of $\rm{A}_3,$ then $\mathfrak{s}$ is one of the following subalgebras:
$\big\langle e_1 \big\rangle,$ 
$\big\langle e_2 \big\rangle,$  
$\big\langle e_3 \big\rangle,$ or 
$\big\langle e_1-e_2 \big\rangle.$
Up to isomorphisms, all two-dimensional subalgebras of $\rm{A}_3$ are equivalent to  
$\big\langle e_1 \big\rangle,$ 
$\big\langle e_2 \big\rangle,$  
$\big\langle e_3 \big\rangle$ or 
$\big\langle e_1-e_2 \big\rangle.$
\end{theorem}
\begin{proof}
Let $\mathfrak{s}=\big\langle \alpha e_1 + \beta e_2 + \gamma e_3 \big\rangle .$
Consider the product:
\[
\big(\alpha e_1 + \beta e_2 + \gamma e_3\big)
\big(\alpha e_1 + \beta e_2 + \gamma e_3\big)
\ =\  k\big(\alpha e_1 + \beta e_2 + \gamma e_3\big),
\]
which leads to
\[
\alpha^2 e_1 + \beta^2e_2+2\alpha\beta e_2 + 2\alpha\gamma e_3
\ =\  k\big(\alpha e_1 + \beta e_2 + \gamma e_3\big),
\]
i.e., $\alpha^2 \ =\  k\alpha,$ \ $\beta^2+2\alpha\beta \ =\  k\beta,$ \ and $2\alpha\gamma \ =\  k\gamma.$ 
Let us consider the following cases:

\begin{enumerate}
\item[(1)] 
If $k = 0$, then $\alpha = 0$ and$\beta=0.$
We obtain the subalgebra
$\big\langle e_3 \big\rangle$. 

\item[(2)] If $k \neq 0$ and $\alpha = 0$, then $\gamma = 0$, and  we obtain the subalgebra
$\big\langle e_2 \big\rangle$. 

\item[(3)] If $k \neq 0$ and $\alpha \neq 0$, then $\alpha = k$,\  $\gamma = 0$. If $\beta=0,$ which yields the subalgebra
$\big\langle e_1 \big\rangle .$

\item[(4)] If $k \neq 0$ and $\alpha \neq 0$, then $\alpha = k$,\  $\gamma = 0$. If $\beta\neq0,$ which yields the subalgebra
$\big\langle e_1-e_2 \big\rangle .$
 
\end{enumerate}
Taking automorphisms of $\rm{A}_3$ from Proposition \ref{Aut},  we have the second part of our statement.
\end{proof}
\begin{corollary}

Let $\mathfrak{s}$ be a one-dimensional $\overline{\mbox{subalgebra}}$ of $\rm{A}_3,$ then $\mathfrak{s}$ is one of the following $\overline{\mbox{subalgebras}}$:
$\big\langle e_1 \big\rangle,$ $\big\langle e_2 \big\rangle,$ 
$\big\langle e_3 \big\rangle,$ or  
$\big\langle e_1-e_2 \big\rangle.$ 
Up to isomorphisms, all one-dimensional $\overline{\mbox{subalgebras}}$ of $\rm{A}_3$ are equivalent to  $\big\langle e_1 \big\rangle,$ $\big\langle e_2 \big\rangle,$ 
$\big\langle e_3 \big\rangle$ or  
$\big\langle e_1-e_2 \big\rangle.$ 
\end{corollary}

\begin{corollary} 
Each one-dimensional ideal  of $\rm{A}_3$ is an    $\overline{\mbox{ideal}}$ and 
it is one of the following: $\big\langle e_2 \big\rangle$  or $\big\langle  e_3 \big\rangle.$ 
\end{corollary}

\begin{theorem}
Let $\mathfrak{s}$ be a two-dimensional subalgebra of $\rm{A}_3,$ then $\mathfrak{s}$ is one of the following subalgebras:
$\big\langle e_1,e_3 \big\rangle,$ \  
$\big\langle e_1-e_2,e_3 \big\rangle,$ \ 
$\big\langle e_1, e_2 \big\rangle,$ or
$\big\langle e_2, e_3\big\rangle.$
Up to isomorphisms, all two-dimensional subalgebras of $\rm{A}_3$ are equivalent to  
$\big\langle e_1,e_3 \big\rangle,$ \  
$\big\langle e_1-e_2,e_3 \big\rangle,$ \ 
$\big\langle e_1, e_2 \big\rangle$ or
$\big\langle e_2, e_3 \big\rangle.$

\end{theorem}
\begin{proof}
It is obviously that $\big\langle e_2, e_3\big\rangle$ is a subalgebra of ${\rm A}_3.$
Let $\mathfrak{s}=\langle e_1 + \alpha_2 e_2 +\alpha_3 e_3,  \beta_2e_2+\beta_3e_3 \rangle.$
Consider the product:
\begin{center}
    $\big(e_1 + \alpha_2 e_2 +\alpha_3 e_3\big)\big(e_1 + \alpha_2 e_2 +\alpha_3 e_3\big)= 
    e_1+ (2\alpha_2 + \alpha_2^2) e_2 + 2\alpha_3e_3
    =k_1\big(e_1 + \alpha_2 e_2+\alpha_3 e_3)+k_2(\beta_2e_2+ \beta_3 e_3\big),$
\end{center}
i.e., $k_1=1,$ $\alpha_2+\alpha_2^2=k_2\beta_2,$ and  $\alpha_3=k_2\beta_3;$
and $\mathfrak{s}=\big\langle e_1 - \alpha_2^2 e_2,\  \beta_2e_2+ \beta_3 e_3 \big\rangle.$
If $\beta_2\neq0,$ then 
$\big(\beta_2e_2+ \beta_3 e_3 \big)\big(\beta_2e_2+ \beta_3 e_3 \big)=\beta_2^2 e_2 \in \mathfrak{s},$
i.e., $e_2 \in \mathfrak{s}$ and $ \beta_3=0,$ which gives a subalgebra $\big\langle e_1,   e_2 \big\rangle$.
If $\beta_2=0,$ then  $\beta_3 \neq 0$ and $e_3 \in \mathfrak{s}.$
Taking the classification of one-dimentional subalgebras of ${\rm A}_3$ 
$\big($see Theorem \ref{A3subd1}$\big),$ we have two cases $\alpha_2=0$ or $\alpha_2=-1.$
The last observation completes the proof of the first part of the statement. 

Taking automorphisms of $\rm{A}_3$ from Proposition \ref{Aut},  we have the second part of our statement.
\end{proof}
\begin{corollary}

Let $\mathfrak{s}$ be a two-dimensional $\overline{\mbox{subalgebra}}$ of $\rm{A}_3,$ then $\mathfrak{s}$ is one of the following $\big\langle e_1 \big\rangle,$ $\big\langle e_2 \big\rangle,$ 
$\big\langle e_3 \big\rangle,$ or  
$\big\langle e_1-e_2 \big\rangle.$ 
Up to isomorphisms, all two-dimensional $\overline{\mbox{subalgebras}}$ of $\rm{A}_3$ are equivalent to  $\big\langle e_1 \big\rangle,$ $\big\langle e_2 \big\rangle,$ 
$\big\langle e_3 \big\rangle$ or  
$\big\langle e_1-e_2 \big\rangle.$ 
\end{corollary}

\begin{corollary}  There is only one two-dimensional ideal of $\rm{A}_3,$ 
it is an $\overline{\mbox{ideal}}$ and equal to  $\big\langle e_2, \ e_3\big\rangle.$ 
\end{corollary}

\medskip \noindent{\it Subalgebras of  ${\rm A}_4.$ }

\begin{theorem}
Let $\mathfrak{s}$ be a one-dimensional subalgebra of $\rm{A}_4,$ then $\mathfrak{s}$ is one of the following subalgebras:
\begin{center}
    $
\big\langle e_1 \big\rangle,$ \ 
$\big\langle e_2 \big\rangle,$ \ 
$\big\langle e_1-e_2 \big\rangle,$ \ 
$\big\langle e_1+e_2 \pm {\mathfrak i}e_3 \big\rangle$ \ or
$\big\langle e_1-e_2\pm{\mathfrak i}e_3 \big\rangle.$
\end{center}
Up to isomorphisms, all one-dimensional subalgebras of $\rm{A}_4$ are equivalent to  
\begin{center}
$\big\langle e_1 \big\rangle,$ \ 
$\big\langle e_2 \big\rangle,$ \ 
$\big\langle e_1-e_2 \big\rangle$ \  or 
$\big\langle e_1 \pm e_2+{\mathfrak i}e_3 \big\rangle  .$ 
\end{center}
\end{theorem}
\begin{proof}
Let $\mathfrak{s}=\big\langle \alpha e_1 + \beta e_2 + \gamma e_3 \big\rangle .$
Consider the product:
\[
\big(\alpha e_1 + \beta e_2 + \gamma e_3\big)
\big(\alpha e_1 + \beta e_2 + \gamma e_3\big)
\ =\  k\big(\alpha e_1 + \beta e_2 + \gamma e_3\big),
\]
which leads to
\[
\alpha^2 e_1 + \beta^2e_2+\gamma^2(-e_1+e_2)+2\alpha\beta e_2 + 2\alpha\gamma e_3
\ =\  k\big(\alpha e_1 + \beta e_2 + \gamma e_3\big),
\]
i.e., $\alpha^2-\gamma^2 \ =\  k\alpha,$ \ $\beta^2+\gamma^2+2\alpha\beta \ =\  k\beta,$ \ and $2\alpha\gamma \ =\  k\gamma.$ 
Let us consider the following cases:

\begin{enumerate}
\item[(1)] 
If $k = 0,$ then  $\alpha=\gamma=\beta=0,$ which is a contradiction.
 
\item[(2)] If $k\neq 0,\gamma = 0$ and  $\alpha = 0$, and  we obtain the subalgebra $\big\langle e_2 \big\rangle.$
\item[(3)] If $k\neq 0, \gamma = 0$ and $\alpha \neq 0,$ then $\alpha=k$ and  $\beta^2=- \alpha\beta,$ we have those subalgebras 
$\big\langle e_1 \big\rangle,$ and $\big\langle e_1-e_2 \big\rangle.$
\item[(4)] If $k \neq 0$ and $\gamma \neq 0$, 
then $\alpha = \frac{k}{2},$ which gives 
$\beta^2 = \frac{k^2}{4}$ and 
 $\gamma^2=-\frac{k^2}{4}$. 
We construct the following subalgebras:
$\big\langle e_1+e_2\pm{\mathfrak i}e_3 \big\rangle$ and 
$\big\langle e_1-e_2\pm{\mathfrak i}e_3 \big\rangle.$

\end{enumerate}
\noindent Taking automorphisms of $\rm{A}_4$ from Proposition \ref{Aut}, 
we have the second part of our statement.\end{proof}
\begin{corollary}

Let $\mathfrak{s}$ be a one-dimensional $\overline{\mbox{subalgebra}}$ of $\rm{A}_4,$ then $\mathfrak{s}$ is one of the following $\overline{\mbox{subalgebras}}$:
$\big\langle e_1 \big\rangle,$ 
$\big\langle e_2 \big\rangle,$ or 
$\big\langle e_1-e_2 \big\rangle.$ 
Up to isomorphisms, all one-dimensional $\overline{\mbox{subalgebras}}$ of $\rm{A}_4$ are equivalent to  $\big\langle e_1 \big\rangle,$ $\big\langle e_2 \big\rangle$ 
or $\big\langle e_1-e_2 \big\rangle.$ 

\end{corollary}
\begin{corollary} 
There is only one one-dimensional ideal of $\rm{A}_4,$ 
it is an $\overline{\mbox{ideal}}$ and equal to    $\big\langle   e_2 \big\rangle.$
\end{corollary}

\begin{theorem}
Let $\mathfrak{s}$ be a two-dimensional subalgebra of $\rm{A}_4,$ then $\mathfrak{s}$ is one of the following subalgebras:
$\big\langle e_1, e_2\big\rangle,$
$\big\langle e_1\pm {\mathfrak i} e_3, e_2\big\rangle,$
$\big\langle e_1-e_2,  e_3\big\rangle$ or 
$\big\langle e_1, \pm {\mathfrak i}e_2+ e_3\big\rangle.$
Up to isomorphisms, all two-dimensional subalgebras of $\rm{A}_4$ are equivalent to  
\begin{center}
$\big\langle e_1, e_2\big\rangle,$
$\big\langle e_1+ {\mathfrak i} e_3, e_2\big\rangle,$
$\big\langle e_1-e_2,  e_3\big\rangle$ or 
$\big\langle e_1,   {\mathfrak i}e_2+ e_3\big\rangle.$
\end{center}

\end{theorem}
\begin{proof}
It is easy to see that $\big\langle e_2, e_3\big\rangle$ is not a subalgebra of ${\rm A}_4.$
Hence, 
\begin{center}
$\mathfrak{s}=\big\langle e_1 + \alpha_2 e_2+\alpha_3 e_3, \  \beta_2 e_2+\beta_3 e_3 \big\rangle.$
\end{center}
Consider the product:
\begin{equation}
  \label{eqA4}  (\beta_2 e_2+\beta_3 e_3)(\beta_2 e_2+\beta_3 e_3)= -\beta_3^2 e_1+ (\beta_2^2+\beta_3^2)e_2 \in 
\mathfrak s,  
\end{equation}
which arrives us to the following two cases.
\begin{enumerate}
    
\item[(1)] If $\beta_3=0,$ then $e_2\in \mathfrak{s}$ and $\mathfrak{s}= \big\langle e_1+\alpha_3 e_3, e_2\big\rangle.$ It follows that
\begin{center}
$(1-\alpha_3^2)e_1+\alpha_3^2e_2+2\alpha_3 e_3 =
\big(e_1 +  \alpha_3 e_3\big) \big(e_1 +  \alpha_3 e_3\big)=
k_1(e_1 +  \alpha_3 e_3)+k_2 e_2,$ i.e.,
 \end{center}
$k_1=1-\alpha_3^2, \  \alpha_3^2 = k_2 , \  2\alpha_3= k_1 \alpha_3.$
\begin{enumerate}
    \item 
If $\alpha_3=0,$ we have the subalgebra $\big\langle e_1, e_2\big\rangle.$

\item If $\alpha_3\neq0,$ we have that $k_1=2$ and $\alpha^2_3=-1,$
i.e., we have two subalgebras 
$\big\langle e_1\pm {\mathfrak i} e_3, e_2\big\rangle.$

\end{enumerate}

\item[(2)] If $\beta_3 \neq 0,$ then  can suppose that $\alpha_3=0$ and $\beta_3=1,$
i.e., $\mathfrak{s}=\big\langle e_1 + \alpha_2 e_2, \  \beta_2 e_2+  e_3 \big\rangle.$
Hence,

\begin{center}
$ e_1+(2\alpha_2+\alpha_2^2)e_2 =
\big(e_1 + \alpha_2 e_2 \big) \big(e_1 + \alpha_2 e_2 \big)=
k_1(e_1 + \alpha_2 e_2)+k_2(\beta_2 e_2+  e_3),$ i.e.,
 \end{center}

$k_1=1, \ 
2\alpha_2+\alpha_2^2   = k_1\alpha_2+k_2\beta_2, \ 0=  k_2.$
The last implies $\alpha_2+\alpha_2^2=0$, and we have the following two subcases.

\begin{enumerate}
    \item If $\alpha_2=0,$ then due to \eqref{eqA4}, we have $\beta_2= \pm {\mathfrak i}.$

    \item If $\alpha_2=-1,$ then due to \eqref{eqA4}, we have $\beta_2=0.$
    \end{enumerate}

Summarizing, the case $(2)$ gives  three subalgebras: 
$\big\langle e_1-e_2,  e_3\big\rangle$ and 
$\big\langle e_1, \pm {\mathfrak i}e_2+ e_3\big\rangle.$
\end{enumerate}
Taking automorphisms of $\rm{A}_4$ from Proposition \ref{Aut}, we have the second part of our statement.
\end{proof}
\begin{corollary}

Let $\mathfrak{s}$ be a two-dimensional $\overline{\mbox{subalgebra}}$ of $\rm{A}_4,$ then $\mathfrak{s}$ is one of the following 
$\big\langle e_1, e_2\big\rangle$ or
$\big\langle e_1-e_2,  e_3\big\rangle.$  
Up to isomorphisms, all two-dimensional $\overline{\mbox{subalgebras}}$ of $\rm{A}_4$ are equivalent to  
$\big\langle e_1, e_2\big\rangle$ or 
$\big\langle e_1-e_2,  e_3\big\rangle.$  
\end{corollary}
\begin{corollary} Two-dimensional ideals of $\rm{A}_4$ are the following
$\big\langle e_1\pm {\mathfrak i} e_3, e_2\big\rangle$ or
$\big\langle e_1-e_2,  e_3\big\rangle.$ 
There is only one two-dimensional $\overline{\mbox{ideal}}$ equal to 
$\big\langle e_1-e_2,  e_3\big\rangle.$ 
\end{corollary}

\medskip \noindent{\it Subalgebras of  ${\rm A}_5.$}

\begin{theorem}\label{subA5}
Let $\mathfrak{s}$ be a one-dimensional subalgebra of $\rm{A}_5,$ then $\mathfrak{s}$ is one of the following subalgebras:
$\big\langle e_1 \big\rangle,$ \ 
$\big\langle e_2 \big\rangle,$ \ 
or 
$\big\langle e_1+\alpha e_2 \pm e_3 \big\rangle_{\alpha\in \mathbb C}.$
Up to isomorphisms, all one-dimensional subalgebras of $\rm{A}_5$ are equivalent to  
$\big\langle e_1 \big\rangle,$ \ 
$\big\langle e_2 \big\rangle$ \ 
or 
$\big\langle e_1  \pm e_3 \big\rangle.$
\end{theorem}
\begin{proof}
Let $\mathfrak{s}=\big\langle \alpha e_1 + \beta e_2 + \gamma e_3 \big\rangle .$
Consider the product:
\[
\big(\alpha e_1 + \beta e_2 + \gamma e_3\big)
\big(\alpha e_1 + \beta e_2 + \gamma e_3\big)
\ =\  k\big(\alpha e_1 + \beta e_2 + \gamma e_3\big),
\]
which leads to
\[
(\alpha^2 + \gamma^2)e_1+2\alpha\beta e_2 + 2\alpha\gamma e_3
\ =\  k\big(\alpha e_1 + \beta e_2 + \gamma e_3\big),
\]
i.e., $\alpha^2+\gamma^2 \ =\  k\alpha,$ \ $2\alpha\beta \ =\  k\beta,$ \ and $2\alpha\gamma \ =\  k\gamma.$ 
Let us consider the following cases:

\begin{enumerate}
\item[(1)] 
If $k = 0$, then $\alpha =  \gamma=0$, and  we get the subalgebra $\big\langle e_2 \big\rangle.$
  
\item[(2)] If $k \neq 0$ and $\gamma = 0$, 
then $\alpha = 0$ or $\alpha = k.$
The first opportunity gives $\beta=0$, i.e., it is a contradiction. 
The second  opportunity gives $\beta=0$, i.e., we have the subalgebra $\big\langle e_1 \big\rangle.$

\item[(3)] If $k\neq 0$ and $\gamma \neq 0$, then 
$\alpha = \frac{k}{2}$ and $\gamma=\pm \frac{k}{2};$ 
which gives    
$\big\langle  e_1+2\beta e_2 \pm e_3 \big\rangle_{\beta\in\mathbb C}.$

\end{enumerate}
This gives  the first part of our statement.
Taking automorphisms of $\rm{A}_5$ from Proposition \ref{Aut}, 
we have the second part of our statement.
\end{proof}
\begin{corollary}

Let $\mathfrak{s}$ be a one-dimensional $\overline{\mbox{subalgebra}}$ of $\rm{A}_5,$ then $\mathfrak{s}$ is one of the following $\overline{\mbox{subalgebras}}$:
$\big\langle e_1 \rangle$ or $\langle e_2 \big\rangle.$
Up to isomorphisms, all one-dimensional $\overline{\mbox{subalgebras}}$ of $\rm{A}_5$ are equivalent to  $\big\langle e_1 \big\rangle$ or $\big\langle e_2 \big\rangle.$ 

\end{corollary}

\begin{corollary}
There is only one one-dimensional ideal of $\rm{A}_5,$ 
it is an $\overline{\mbox{ideal}}$ and equal to    $\big\langle   e_2 \big\rangle.$
\end{corollary}

\begin{theorem}\label{sub2A5}
Let $\mathfrak{s}$ be a two-dimensional subalgebra of $\rm{A}_5,$ then $\mathfrak{s}$ is one of the following subalgebras:
$\big\langle e_1,e_2 \big\rangle,$ 
$\big\langle e_1\pm e_3,e_2 \big\rangle,$ 
or $
\big\langle e_1,\alpha e_2+e_3 \big\rangle_{\alpha \in \mathbb C}.$
Up to isomorphisms, all two-dimensional subalgebras of $\rm{A}_5$ are equivalent to  
$\big\langle e_1,e_2 \big\rangle,$ 
$\big\langle e_1\pm e_3,e_2 \big\rangle$ 
or $
\big\langle e_1, e_3 \big\rangle.$

\end{theorem}
\begin{proof}
It is easy to see that $\big\langle e_2, e_3\big\rangle$ is not a subalgebra. Let 
\begin{center}$\mathfrak{s}=\big\langle e_1 + \alpha_2 e_2+\alpha_3 e_3,  \beta_2e_2+\beta_3e_3 \big\rangle .$
\end{center} 
Consider the product:
\begin{center}
$(\beta_2e_2+\beta_3e_3)(\beta_2e_2+\beta_3e_3)=
\beta_3^2 e_1 \in \mathfrak s,$
\end{center}
which gives us two cases:
\begin{enumerate}
    \item If $\beta_3\neq0,$ then we have the following family of subalgebras $\big\langle e_1, \alpha e_2+e_3\big\rangle_{\alpha \in \mathbb C}.$
\item If $\beta_3=0,$ then 
$\mathfrak{s}=\big\langle e_1 + \alpha_3 e_3,  e_2\big\rangle$ and 
$(e_1 + \alpha_3 e_3)(e_1 + \alpha_3 e_3) = 
(1+\alpha_3^2)e_1 + 2\alpha_3 e_3 \in \mathfrak s.$
That gives three opportunities for $\alpha_3:$ $0$ or $\pm 1.$\end{enumerate}

\noindent This gives  the first part of our statement.
Taking automorphisms of $\rm{A}_5$ from Proposition \ref{Aut}, 
we have the second part of our statement.
\end{proof}
\begin{corollary}

Let $\mathfrak{s}$ be a two-dimensional $\overline{\mbox{subalgebra}}$ of $\rm{A}_5,$ then $\mathfrak{s}$ is one of the following 
$\big\langle e_1,e_2 \big\rangle$ or 
$\big\langle e_1,e_3 \big\rangle.$ 
Up to isomorphisms, all two-dimensional $\overline{\mbox{subalgebras}}$ of $\rm{A}_5$ are equivalent to  $\big\langle e_1,e_2 \big\rangle$ or 
$\big\langle e_1,e_3 \big\rangle.$ 
\end{corollary}

\begin{corollary} Two-dimensional ideals  of $\rm{A}_5$ are the following
$\big\langle e_1\pm e_3,e_2 \big\rangle,$ 
There is no   one two-dimensional
$\overline{\mbox{ideals}}$ of $\rm{A}_5.$
\end{corollary}
\medskip \noindent{\it Subalgebras of  ${\rm S}_1.$ }

\begin{theorem}
Let $\mathfrak{s}$ be a one-dimensional subalgebra of $\rm{S}_1,$ then $\mathfrak{s}$ is one of the following subalgebras:
$\big\langle e_1 \big\rangle,$ 
$\big\langle e_2 \big\rangle,$ or 
$\big\langle \alpha  e_2 + e_3 \big\rangle_{\alpha \in \mathbb C}.$
Up to isomorphisms, all one-dimensional subalgebras of $\rm{S}_1$ are equivalent to  $\big\langle e_1 \big\rangle $ or $ \big\langle e_2 \big\rangle.$ 

\end{theorem}
\begin{proof}
The multiplication table  of ${\rm S}_1$ coincides with the multiplication table of 
${\rm A}_1,$ hence our proof follows from Theorem \ref{subA1}. 
\end{proof}
\begin{corollary}

Let $\mathfrak{s}$ be a one-dimensional $\overline{\mbox{subalgebra}}$ of $\rm{S}_1,$ then $\mathfrak{s}$ is one of the following $\overline{\mbox{subalgebras}}$:
$\big\langle e_1 \big\rangle,$ 
$\big\langle e_2 \big\rangle,$ or 
$\big\langle \alpha  e_2 + e_3 \big\rangle_{\alpha \in \mathbb C}.$
Up to isomorphisms, all one-dimensional $\overline{\mbox{subalgebras}}$ of $\rm{S}_1$ are equivalent to  $\big\langle e_1 \big\rangle$ or $\big\langle e_2 \big\rangle.$ 

\end{corollary}
\begin{corollary} 
Each one-dimensional ideal of $\rm{S}_1$ is an    $\overline{\mbox{ideal}}$ and it is one of the following: \begin{center}
    $\big\langle e_2 \big\rangle$ or $\big\langle \alpha e_2+e_3 \big\rangle_{\alpha \in \mathbb C}$.
\end{center}\end{corollary}

\begin{theorem}
Let $\mathfrak{s}$ be a two-dimensional subalgebra of $\rm{S}_1,$ then $\mathfrak{s}$ is one of the following subalgebras:
$\big\langle e_1, e_2 \big\rangle,$ \
$\big\langle e_2, e_3 \big\rangle,$ \
or 
$\big\langle e_1, \alpha   e_2 + e_3 \big\rangle_{\alpha \in \mathbb C}.$
Up to isomorphisms, all two-dimensional subalgebras of $\rm{S}_1$ are equivalent to  
$\big\langle e_1 ,e_2\big\rangle$ or 
$\big\langle e_2 ,e_3\big\rangle.$

\end{theorem}

\begin{proof}
The multiplication table of  ${\rm S}_1$ coincides with the multiplication table of ${\rm A}_1,$ hence our proof follows from Theorem \ref{sub2A1}. 
\end{proof}

\begin{corollary}

Let $\mathfrak{s}$ be a two-dimensional $\overline{\mbox{subalgebra}}$ of $\rm{S}_1,$ then $\mathfrak{s}$ is one of the following $\overline{\mbox{subalgebras}}$:
$\big\langle e_1, e_2 \big\rangle,$ \
$\big\langle e_2, e_3 \big\rangle,$ \
or 
$\big\langle e_1, \alpha   e_2 + e_3 \big\rangle_{\alpha \in \mathbb C}.$
Up to isomorphisms, all two-dimensional $\overline{\mbox{subalgebras}}$ of $\rm{S}_1$ are equivalent to  $\big\langle e_1,e_2 \big\rangle$  or 
$\big\langle     e_2,e_3 \big\rangle.$
\end{corollary}

\begin{corollary} 
There is only one two-dimensional ideal of $\rm{S}_1,$ 
it is an $\overline{\mbox{ideal}}$ and equal to  $\big\langle e_2 ,e_3\big\rangle.$
\end{corollary}
\medskip \noindent{\it Subalgebras of  ${\rm S}_2.$ }

\begin{theorem}\label{49 th} 
Let $\mathfrak{s}$ be a one-dimensional subalgebra of $\rm{S}_2,$ then $\mathfrak{s}$ is one of the following subalgebras:
$\big\langle e_1 \big\rangle,$ \ 
$\big\langle e_2 \big\rangle,$ \ or 
$\big\langle e_1+\alpha e_2 \pm e_3 \big\rangle_{\alpha\in \mathbb C}.$
Up to isomorphisms, all one-dimensional subalgebras of  $\rm{S}_2$ are equivalent to  
$\big\langle e_1 \big\rangle,$ \ 
$\big\langle e_2 \big\rangle,$ \ or 
$\big\langle e_1\pm e_3 \big\rangle.$  
\end{theorem}
\begin{proof}
The multiplication table of  ${\rm S}_2$ coincides with the multiplication table of ${\rm A}_5,$ hence our proof follows from Theorem \ref{subA5}. 
\end{proof}
\begin{corollary}

Let $\mathfrak{s}$ be a one-dimensional $\overline{\mbox{subalgebra}}$ of $\rm{S}_2,$ then $\mathfrak{s}$ is one of the following $\overline{\mbox{subalgebras}}$:
$\big\langle e_1 \big\rangle$ 
or 
$\big\langle e_2 \big\rangle.$
Up to isomorphisms, all one-dimensional 
$\overline{\mbox{subalgebras}}$ of $\rm{S}_2$ are equivalent to  $\big\langle e_1 \big\rangle$ or $\big\langle e_2 \big\rangle.$ 

\end{corollary}

\begin{corollary}
There is only one-dimensional ideal of $\rm{S}_2,$ 
it is an $\overline{\mbox{ideal}}$ and equal to  $\big\langle e_2 \big\rangle.$
\end{corollary}

\begin{theorem}
Let $\mathfrak{s}$ be a two-dimensional subalgebra of $\rm{S}_2,$ then $\mathfrak{s}$ is one of the following subalgebras:
$\big\langle e_1,e_2 \big\rangle,$ 
$\big\langle e_1\pm e_3,e_2 \big\rangle,$ 
or $
\big\langle e_1,\alpha e_2+e_3 \big\rangle_{\alpha \in \mathbb C}.$
Up to isomorphisms, all two-dimensional subalgebras of $\rm{S}_2$ are equivalent to  
$\big\langle e_1,e_2 \big\rangle,$ 
$\big\langle e_1\pm e_3,e_2 \big\rangle$ 
or $
\big\langle e_1, e_3 \big\rangle.$
\end{theorem}

\begin{proof}
The multiplication table of  ${\rm S}_2$ coincides with the multiplication table of ${\rm A}_5,$ hence our proof follows from Theorem \ref{sub2A5}.
\end{proof}

\begin{corollary}

Let $\mathfrak{s}$ be a two-dimensional $\overline{\mbox{subalgebra}}$ of $\rm{S}_2,$ then $\mathfrak{s}$ is one of the following 
$\big\langle e_1,e_2 \big\rangle,$ 
or $
\big\langle e_1,\alpha e_2+e_3 \big\rangle_{\alpha \in \mathbb C}.$
Up to isomorphisms, all two-dimensional $\overline{\mbox{subalgebras}}$ of $\rm{S}_2$ are equivalent to  
$\big\langle e_1,e_2 \big\rangle$ or 
$\big\langle e_1,e_3 \big\rangle.$
\end{corollary}

\begin{corollary}  Two-dimensional ideals  of $\rm{S}_2$ are the following
$\big\langle e_1\pm e_3,e_2 \big\rangle,$ 
There is no   one two-dimensional
$\overline{\mbox{ideals}}$ of $\rm{S}_2.$
\end{corollary}

\subsubsection{Identities  of  unital $3$-dimensional structurable algebras}\label{ident}
Let us remember that each structurable algebra with the identical involution is a Jordan algebra, i.e., all of them are commutative, and they satisfy the same identity of degree $2.$ 
On the other hand, our algebras $\rm A_5$ and $\rm S_2$ are not commutative.
The present subsection aims to study the question of the existence of a common functional identity of degree $2$, i.e., an identity that includes an involution and is valued in all unital $3$-dimensional structurable algebras with nontrivial involution.
Let us consider the general form of functional identities of degree $2$ below:
\begin{center}
$f(x,y) = \alpha_1 xy+ \alpha_2  yx +
\beta_1 \overline{x}y+ \beta_2 x \overline{y} + \beta_3 \overline{y} x +\beta_4 y \overline{x}+ 
\gamma_1 \overline {xy} + \gamma_2\overline{yx}.$
\end{center}
In the commutative case, our identity $f(x,y)$ can be reduced to 
\begin{center}
$f_c(x,y) = \alpha xy+ \beta \overline{x}y+ \gamma  x \overline{y} +\delta  \overline {xy}.$
\end{center}
All our algebras, except for  $\rm A_5$ and $\rm S_2$ are commutative, 
hence, for all our commutative algebras, we will consider functional identities reduced $f_c(x,y)$ by the commutative law. 

\begin{proposition}
    Let $\mathfrak f(x,y)$ be a functional identity of degree $2$ of $\rm A_1,$ 
    then $\mathfrak f(x,y)$ is constructed from the commutative identity and 
    $\mathfrak F(x,y)=  xy - \overline{x}y- x \overline{y} +   \overline {xy}.$
  \end{proposition}

\begin{proof}
We consider the following relations:
\begin{longtable}{lllllll}
    
$0$&$=$&$\mathfrak f(e_1,e_1)$&$=$&$\alpha e_1+ \beta e_1+ \gamma  e_1 +\delta  e_1,$ & then &  $\alpha+\beta+\gamma+\delta=0;$\\

$0$&$=$&$\mathfrak f(e_1,e_3)$&$=$&$\alpha e_3+ \beta e_3 - \gamma  e_3 - \delta  e_3,$ & then &  $\alpha+\beta-\gamma-\delta=0;$\\ 

$0$&$=$&$\mathfrak f(e_3,e_1)$&$=$&$\alpha e_3- \beta e_3 + \gamma  e_3 - \delta  e_3,$ & then &  $\alpha-\beta+\gamma-\delta=0;$\\ 
    \end{longtable}
\noindent that gives $\alpha=-\beta=-\gamma=\delta$ and $\rm A_1$ satisfies the identity 
 $\mathfrak F(x,y):=  xy - \overline{x}y- x \overline{y} +   \overline {xy}.$
\end{proof}

\begin{proposition}
    Let $\mathfrak f(x,y)$ be a functional identity of degree $2$ of $\rm A_2,$ 
    then $\mathfrak f(x,y)$ is constructed  from the commutative identity. 
  \end{proposition}

\begin{proof}
We consider the following relations:
\begin{longtable}{lllllll}
    
$0$&$=$&$\mathfrak f(e_1,e_1)$&$=$&$\alpha e_1+ \beta e_1+ \gamma  e_1 +\delta  e_1,$ & then &  $\alpha+\beta+\gamma+\delta=0;$\\

$0$&$=$&$\mathfrak f(e_1,e_3)$&$=$&$\alpha e_3+ \beta e_3 - \gamma  e_3 - \delta  e_3,$ & then &  $\alpha+\beta-\gamma-\delta=0;$\\ 

$0$&$=$&$\mathfrak f(e_3,e_1)$&$=$&$\alpha e_3- \beta e_3 + \gamma  e_3 - \delta  e_3,$ & then &  $\alpha-\beta+\gamma-\delta=0;$\\

$0$&$=$&$\mathfrak f(e_3,e_3)$&$=$&$\alpha e_2- \beta e_2 - \gamma  e_2 +\delta  e_2,$ & then &  $\alpha-\beta-\gamma+\delta=0;$  
    \end{longtable}
\noindent that gives $\alpha=\beta=\gamma=\delta=0$ and $\rm A_2$ does not satisfy non-commutative functional identities of degree $2.$  
\end{proof}

\begin{proposition}
    Let $\mathfrak f(x,y)$ be a functional identity of degree $2$ of $\rm A_3,$ 
    then $\mathfrak f(x,y)$ is constructed from the commutative identity and 
    $\mathfrak F(x,y):=  xy - \overline{x}y- x \overline{y} +   \overline {xy}.$
  \end{proposition}

\begin{proof}
We consider the following relations:
\begin{longtable}{lllllll}
    
$0$&$=$&$\mathfrak f(e_1,e_1)$&$=$&$\alpha e_1+ \beta e_1+ \gamma  e_1 +\delta  e_1,$ & then &  $\alpha+\beta+\gamma+\delta=0;$\\

$0$&$=$&$\mathfrak f(e_1,e_3)$&$=$&$\alpha e_3+ \beta e_3 - \gamma  e_3 - \delta  e_3,$ & then &  $\alpha+\beta-\gamma-\delta=0;$\\ 

$0$&$=$&$\mathfrak f(e_3,e_1)$&$=$&$\alpha e_3- \beta e_3 + \gamma  e_3 - \delta  e_3,$ & then &  $\alpha-\beta+\gamma-\delta=0;$\\ 
    
    \end{longtable}
\noindent that gives $\alpha=-\beta=-\gamma=\delta$ and $\rm A_3$ satisfies the identity 
 $\mathfrak F(x,y):=  xy - \overline{x}y- x \overline{y} +   \overline {xy}.$
\end{proof}

\begin{proposition}
    Let $\mathfrak f(x,y)$ be a functional identity of degree $2$ of $\rm A_4,$ 
    then $\mathfrak f(x,y)$ is constructed from the commutative identity.
  \end{proposition}

\begin{proof}
We consider the following relations:
\begin{longtable}{lllllll}
    
$0$&$=$&$\mathfrak f(e_1,e_1)$&$=$&$\alpha e_1+ \beta e_1+ \gamma  e_1 +\delta  e_1,$ & then &  $\alpha+\beta+\gamma+\delta=0;$\\

$0$&$=$&$\mathfrak f(e_1,e_3)$&$=$&$\alpha e_3+ \beta e_3 - \gamma  e_3 - \delta  e_3,$ & then &  $\alpha+\beta-\gamma-\delta=0;$\\ 

$0$&$=$&$\mathfrak f(e_3,e_1)$&$=$&$\alpha e_3- \beta e_3 + \gamma  e_3 - \delta  e_3,$ & then &  $\alpha-\beta+\gamma-\delta=0;$\\

$0$&$=$&$\mathfrak f(e_3,e_3)$&$=$&$ \big(\alpha  - \beta   - \gamma    +\delta\big)  (e_2-e_1),$ & then &  $\alpha-\beta-\gamma+\delta=0;$\\ 
    
    \end{longtable}\noindent
that gives $\alpha=\beta=\gamma=\delta=0$ and $\rm A_4$ satisfies the identity 
 $\mathfrak f_1(x,y):=  0.$
\end{proof}

\begin{proposition}
    Let $\mathfrak f(x,y)$ be a functional identity of degree $2$ of $\rm A_5,$ 
    then $\mathfrak f(x,y)$ is constructed from 
$\mathfrak f_1(x,y) = xy-yx-\overline{xy}+\overline{yx}$ and 
$\mathfrak f_2(x,y) =  \overline{x}y+x\overline{y}-\overline{y}x-y\overline{x}.$

  \end{proposition}
\begin{proof}
    First, we consider the following relations:
\begin{longtable}{llllr}
$0$&$=$&$\mathfrak f(e_1,e_1)$&$=$&$\big(\alpha_1+\alpha_2+\beta_1+\beta_2+\beta_3+\beta_4+\gamma_1+\gamma_2\big)e_1;$ \\

$0$&$=$&$\mathfrak f(e_1,e_3)$&$=$&$\big(\alpha_1+\alpha_2+\beta_1-\beta_2-\beta_3+\beta_4-\gamma_1-\gamma_2\big)e_3;$ \\

$0$&$=$&$\mathfrak f(e_2,e_3)$&$=$&$\big(\alpha_1-\alpha_2+\beta_1-\beta_2+\beta_3-\beta_4+\gamma_1-\gamma_2\big)e_2;$  \\

$0$&$=$&$\mathfrak f(e_3,e_1)$&$=$&$\big(\alpha_1+\alpha_2-\beta_1+\beta_2+\beta_3-\beta_4-\gamma_1-\gamma_2\big)e_3;$\\ 

$0$&$=$&$\mathfrak f(e_3,e_2)$&$=$&$\big(-\alpha_1+\alpha_2+\beta_1-\beta_2+\beta_3-\beta_4-\gamma_1+\gamma_2\big)e_2;$ \\ 

$0$&$=$&$\mathfrak f(e_3,e_3)$&$=$&$\big(\alpha_1+\alpha_2-\beta_1-\beta_2-\beta_3-\beta_4+\gamma_1+\gamma_2\big)e_1.$  \\ 
    \end{longtable}
\noindent The present relations give a system of linear equations, which has the following solution:
\begin{center}
    $\alpha_1= \xi, \ 
    \alpha_2=-\xi,\ 
    \beta_1=\nu,\ 
    \beta_2=\nu,\ 
    \beta_3=-\nu,\ 
    \beta_4=-\nu,\ 
    \gamma_1=-\xi,\ 
    \gamma_2=\xi,$
\end{center}
i.e., $\mathfrak f(x,y)= \xi \mathfrak f_1(x,y) + \nu \mathfrak f_2(x,y),$ 
where $\mathfrak f_1(x,y)$ and $\mathfrak f_2(x,y)$ are functional identities of $\rm A_5,$ given by
\begin{center}
$\mathfrak f_1(x,y):= xy-yx-\overline{xy}+\overline{yx}$ and 
$\mathfrak f_2(x,y):=  \overline{x}y+x\overline{y}-\overline{y}x-y\overline{x}.$
\end{center}

\end{proof}

\begin{proposition}
    Let $\mathfrak f(x,y)$ be a functional identity of degree $2$ of $\rm S_1,$ 
    then $\mathfrak f(x,y)$ is constructed from the commutative identity and 
    $\mathfrak F(x,y) =  xy - \overline{x}y- x \overline{y} +   \overline {xy}.$
  \end{proposition}

\begin{proof}
We consider the following relations:
\begin{longtable}{lllllll}
    
$0$&$=$&$\mathfrak f(e_1,e_1)$&$=$&$\alpha e_1+ \beta e_1+ \gamma  e_1 +\delta  e_1,$ & then &  $\alpha+\beta+\gamma+\delta=0;$\\

$0$&$=$&$\mathfrak f(e_1,e_3)$&$=$&$\alpha e_3+ \beta e_3 - \gamma  e_3 - \delta  e_3,$ & then &  $\alpha+\beta-\gamma-\delta=0;$\\ 

$0$&$=$&$\mathfrak f(e_3,e_1)$&$=$&$\alpha e_3- \beta e_3 + \gamma  e_3 - \delta  e_3,$ & then &  $\alpha-\beta+\gamma-\delta=0;$\\ 
    
    \end{longtable}
\noindent that gives $\alpha=-\beta=-\gamma=\delta$ and $\rm S_1$ satisfies the identity 
 $\mathfrak F(x,y):=  xy - \overline{x}y- x \overline{y} +   \overline {xy}.$
\end{proof}

\begin{proposition}
    Let $\mathfrak f(x,y)$ be a functional identity of degree $2$ of $\rm S_2,$ 
    then $\mathfrak f(x,y)$ is constructed from  general identity and 
\begin{center}
$\mathfrak g_1(x,y)= xy-yx +\overline{x}y-y\overline{x};$  \ 
$\mathfrak g_2(x,y)= xy-yx +x\overline{y}-\overline{y}x;$ and 
$\mathfrak g_3(x,y)= xy-yx +\overline{xy}-\overline{yx}.$
\end{center}

  \end{proposition}
\begin{proof}
    First, we consider the following relations:
\begin{longtable}{llllll}
    
$0$&$=$&$\mathfrak f(e_1,e_1)$&$=$&$(\alpha_1+\alpha_2+\beta_1+\beta_2+\beta_3+\beta_4+\gamma_1+\gamma_2)e_1;$ \\

$0$&$=$&$\mathfrak f(e_1,e_2)$&$=$&$(\alpha_1+\alpha_2+\beta_1-\beta_2-\beta_3+\beta_4-\gamma_1-\gamma_2)e_2;$\\


$0$&$=$&$\mathfrak f(e_2,e_1)$&$=$&$(\alpha_1+\alpha_2-\beta_1+\beta_2+\beta_3-\beta_4-\gamma_1-\gamma_2)e_2;$\\

$0$&$=$&$\mathfrak f(e_2,e_3)$&$=$&$(\alpha_1-\alpha_2-\beta_1-\beta_2+\beta_3+\beta_4-\gamma_1+\gamma_2)e_2;$  \\ 

    $0$&$=$&$\mathfrak f(e_3,e_3)$&$=$&$(\alpha_1+\alpha_2-\beta_1-\beta_2-\beta_3-\beta_4+\gamma_1+\gamma_2)e_1.$  \\ 
    \end{longtable}

From the above equations, we obtain the following solution:
\begin{center}
$\alpha_1 = \nu + \mu + \xi,$\ 
$\alpha_2 = -(\nu + \mu + \xi),$\
$\beta_1= \nu,$\ 
$\beta_2=\mu,$\ 
$\beta_3 = -\mu,$\ 
$\beta_4 = -\nu,$\ 
$\gamma_1= \xi,$\ 
$\gamma_2 = -\xi,$
\end{center}
i.e., $\mathfrak f(x,y)= \nu \mathfrak g_1(x,y) + \mu \mathfrak g_2(x,y)+ \xi \mathfrak g_3(x,y),$ 
where $\mathfrak g_1(x,y),$ $\mathfrak g_2(x,y),$ and $\mathfrak g_3(x,y)$ are functional identities of $\rm S_2,$ given by
\begin{center}
$\mathfrak g_1(x,y):= xy-yx +\overline{x}y-y\overline{x};$  \ 
$\mathfrak g_2(x,y):= xy-yx +x\overline{y}-\overline{y}x;$ and  
$\mathfrak g_3(x,y):= xy-yx +\overline{xy}-\overline{yx}.$
\end{center}

\end{proof}

\begin{corollary}
    Each unital $3$-dimensional structurable algebra, including the case of type $(3,0),$
    satisfies the following two functional identities:
\begin{center}   $\mathfrak f_1(x,y) = xy-yx-\overline{xy}+\overline{yx}$ and 
$\mathfrak f_2(x,y) =  
\overline{x}y-y\overline{x}+x\overline{y}-\overline{y}x.$
\end{center}
\end{corollary}

\subsubsection{Allison---Hein construction}

\begin{proposition}[Allison---Hein construction  \cite{AH81}]
Let $\mathcal A$ be a unital structurable algebra, 
then the multiplication $\star$ defined as 
$x\star y :=   T_x(y),$ 
gives a structure of a conservative
algebra of order $2$ with unit element $e_1$.
\end{proposition}

Direct application of the construction introduced by  Allison and Hein
to unital $3$-dimensional structurable algebras of types $(2,1)$ and $(1,2)$ gives the following statement. 

\begin{proposition}
The multiplication tables of conservative algebras constructed from 
unital $3$-dimensional structurable algebras with a non-identical involution via the Allison---Hein construction are given below 
 
\begin{longtable}{|c|lcllcllcl|}
    \hline 
$\mathcal{C}(\rm{A}_1)$&$
e_1 \star e_1 $&$=$&$ e_1$&$
e_1 \star e_2 $&$=$&$ e_2$&$ 
e_1 \star e_3 $&$=$&$ e_3$ \\
& $e_2 \star e_1 $&$=$&$ e_2$&$
e_2\star e_2 $&$=$&$ 0$&$
e_2 \star e_3 $&$=$&$ 0$ \\
& $e_3 \star e_1 $&$=$&$ 3e_3$&$
e_3 \star e_2 $&$=$&$ 0$&$ 
e_3 \star e_3 $&$=$&$ 0$\\
\hline 

$\mathcal{C}(\rm{A}_2)$&$
e_1 \star e_1 $&$=$&$ e_1$&$
e_1 \star e_2 $&$=$&$ e_2$&$
e_1 \star e_3 $&$=$&$ e_3$ \\
&$e_2 \star e_1 $&$=$&$ e_2$&$
e_2 \star e_2 $&$=$&$ 0$&$
e_2 \star e_3 $&$=$&$ 0$ \\
&$e_3 \star e_1 $&$=$&$ 3e_3$&$
e_3 \star e_2 $&$=$&$ 0$&$
e_3 \star e_3 $&$=$&$ 3e_2$\\
\hline 

$\mathcal{C}(\rm{A}_3)$&$
e_1 \star e_1 $&$=$&$ e_1$&$
e_1 \star e_2 $&$=$&$ e_2$&$
e_1 \star e_3 $&$=$&$ e_3$ \\
&$e_2 \star e_1 $&$=$&$ e_2$&$
e_2 \star e_2 $&$=$&$ e_2$&$
e_2 \star e_3 $&$=$&$ 0$ \\
&$e_3 \star e_1 $&$=$&$ 3e_3$&$
e_3 \star e_2 $&$=$&$ 0$&$
e_3 \star e_3 $&$=$&$ 0$\\
\hline 

$\mathcal{C}(\rm{A}_4)$&$
 e_1 \star e_1 $&$=$&$ e_1$&$
e_1 \star e_2 $&$=$&$ e_2$&$
e_1 \star e_3 $&$=$&$ e_3$ \\
&$e_2 \star e_1 $&$=$&$ e_2$&$
e_2 \star e_2 $&$=$&$ e_2$&$ 
e_2 \star e_3 $&$=$&$0$ \\
&$e_3 \star e_1 $&$=$&$ 3e_3$&$
e_3 \star e_2 $&$=$&$ 0$&$
e_3 \star e_3 $&$=$&$ -3e_1 + 3e_2$\\
\hline 
$\mathcal{C}(\rm{A}_5)$&$
e_1\star e_1 $&$=$&$ e_1$&$
e_1 \star e_2 $&$=$&$ e_2$&$
e_1 \star e_3 $&$=$&$ e_3$ \\
&$e_2 \star e_1 $&$=$&$ e_2$&$
e_2 \star e_2 $&$=$&$ 0$&$
e_2 \star e_3 $&$=$&$ e_2$ \\
&$e_3 \star e_1 $&$=$&$ 3e_3$&$
e_3 \star e_2 $&$=$&$ e_2$&$
e_3 \star e_3 $&$=$&$ 3e_1$\\
\hline 
 
\hline 

$\mathcal{C}(\rm{S}_1)$&$
e_1 \star e_1 $&$=$&$ e_1$&$ 
e_1 \star e_2 $&$=$&$ e_2$&$ 
e_1 \star e_3 $&$=$&$ e_3$ \\
&$e_2 \star e_1 $&$=$&$ 3e_2$&$ 
e_2 \star e_2 $&$=$&$ 0$&$ 
e_2 \star e_3 $&$=$&$ 0$ \\

&$e_3 \star e_1 $&$=$&$ 3e_3$&$ 
e_3 \star e_2 $&$=$&$ 0$&$ 
e_3 \star e_3 $&$=$&$ 0$\\
\hline 

$\mathcal{C}(\rm{S}_2)$&$
e_1 \star e_1 $&$=$&$ e_1$&$
e_1 \star e_2 $&$=$&$ e_2$&$
e_1 \star e_3 $&$=$&$ e_3$ \\
&$e_2 \star e_1 $&$=$&$ 3e_2$&$ 
e_2 \star e_2 $&$=$&$ 0$&$ 
e_2 \star e_3 $&$=$&$ -e_2$ \\
&$e_3 \star e_1 $&$=$&$ 3e_3$&$ 
e_3 \star e_2 $&$=$&$ e_2$&$
e_3 \star e_3 $&$=$&$ 3e_1$\\
\hline 
\end{longtable}
\end{proposition}


\begin{proposition}
Derivations of complex  $3$-dimensional  conservative algebras, constructed from unital structurable algebras via the Allison---Hein construction, are given below:

\begin{longtable}{|lcl|}
  \hline 
  
${\mathfrak{Der}}(\mathcal{C}(\rm{A}_1))$&$=$&$\big\langle d_1,d_2 \ | \ 
d_1(e_2)=e_2; \ 
d_2(e_3)=e_3 \ 
\big\rangle$\\

\hline

${\mathfrak{Der}}(\mathcal{C}(\rm{A}_2))$&$=$&$\big\langle d_1 \ | \  
d_1(e_2)=2e_2;\ 
d_1(e_3)=e_3 \ 
\big\rangle$\\

\hline 
  
${\mathfrak{Der}}(\mathcal{C}(\rm{A}_3))$&$=$&$\big\langle d_1 \ | \   d_1(e_2)=e_2 \big\rangle$\\

\hline 
  
${\mathfrak{Der}}(\mathcal{C}(\rm{A}_4))$&$=$&$\big\langle 0\big\rangle$\\

\hline 
  
${\mathfrak{Der}}(\mathcal{C}(\rm{A}_5))$&$=$&$\big\langle d_1 \ | \  
d_1(e_2)=e_2 \ 
\big\rangle$\\

\hline 
  
${\mathfrak{Der}}(\mathcal{C}(\rm{S}_1))$&$=$&$\big\langle d_1,d_2,d_3,d_4 \ | \ 
d_1(e_2)=e_2; \
d_2(e_2)=e_3; \ 
d_3(e_3)=e_2; \ 
d_4(e_3)=e_3\big\rangle$\\

\hline 

${\mathfrak{Der}}(\mathcal{C}(\rm{S}_2))$&$=$&$\big\langle d_1,d_2 \ | \ 
d_1(e_2)=e_2;\ d_2(e_3)=e_2 \ 
\big\rangle$\\

\hline 
  
\end{longtable}
\end{proposition}

\newpage
\section{Allison---Kantor
construction}

\subsection{The Allison---Kantor structure algebra}\label{AKalg}
Let $\overline{\mathfrak{Der}}(\mathcal{A}) $ be the set of derivations of $\mathcal{A}$ that commute with $^-$. It is easy to note that if $x \in \mathcal{A}$
and $T_x(1)=0$ then $x=0$. Therefore, $T_{\mathcal{A}}\cap \overline{\mathfrak{Der}}(\mathcal{A})=0$ and we may define 
\begin{longtable}{rcl}
${\rm Sk}(\mathcal{A})$&$=$&$T_{\mathcal{S}}\oplus \overline{\mathfrak{Der}}(\mathcal{A}),$\\
${\rm Str}(\mathcal{A})$&$=$&$T_{\mathcal{A}}\oplus \overline{\mathfrak{Der}}(\mathcal{A}).$
\end{longtable}
Denote by $L_x$ and $R_x$ the operators of left and right multiplication on $x,$ i.e., $L_x(y)=xy,\ R_x(y)=yx.$ If $E\in {\rm End} (\mathcal{A})$ then we define
$$E^\delta=E+R_{\overline{E(e_1)}},\ \ \ 
E^\varepsilon=E-T_{E(e_1)+\overline{E(e_1)}},\ \  \    
\overline{E}(x)=\overline{E(\overline x)}.$$
If $D\in \overline{\mathfrak{Der}}(\mathcal{A})  $ then $D^\delta=D^\varepsilon=\overline{D}=D$ and 
$$T_x^\delta(y)=xy+y\overline{ x};\  \  T^\varepsilon_x=-T_{\overline{ x}};\ \  \overline{T_x}(y)=\overline{ x}y -xy+y\overline{ x}.$$
Consequently, if 
$E\in {\rm Str}(\mathcal{A}) $ then $\overline{E^\delta}=E^\delta$ and so $E^\delta $ stabilizes $\mathcal{S}$ and $\mathcal{H}.$

\medskip 

Put $N=\big\{(x,s): \ x\in \mathcal{A}, s\in \mathcal{S}\big\}.$ 
Then, since the map $E \mapsto E^\delta |_s$ is a Lie algebra homomorphism, $N$ is a 
${\rm Str}(\mathcal{A})$ -module under the action 
$E(x,s)\ =\ \big(E(x),E^\delta(s)\big).$ Put 
$$\mathcal{F}(\mathcal{A})=\wideparen  N \oplus {\rm Str}(\mathcal{A})\oplus N,$$
where $\wideparen{N}$ is   isomorphic of  $N$. 
Define anticommutative operation $[\cdot,\cdot]$ on 
$\mathcal{F}(\mathcal{A}),$
such that 
for $x,y\in \mathcal{A}, s,r\in \mathcal{S}, \ E\in {\rm Str}(\mathcal{A}): $
\begin{longtable}{rcl}
$[{E},  (x,s)] $&$=$&$ ({E}(x), {E}^\delta(s)),$\\ 
$[{E}, \wideparen {(x,s)} ] $&$ =$&$
\wideparen{({E}^\varepsilon(x), {E}^{\varepsilon\delta}(s))},$\\
$[(x,s),(y,r)]$&$ = $&$(0, x\overline y - y\overline x),$\\
$[\wideparen{(x,s)}, \wideparen{(y,r)}]$&$=$&$ \wideparen{(0, x\overline {y} - y\overline {x})},$\\
$[(x,s), \wideparen{(y,r)}] $&$=$&$(s y, 0)- \wideparen{(r x, 0)}  + V_{x,y} + L_s L_r.$
\end{longtable} 

It is known that $\mathcal{F}(\mathcal{A})$ is a  $\mathbb{Z}$-graded Lie algebra:
\begin{center}
$\mathcal{F}(\mathcal{A})\ =\ \mathcal{F}_{-2}\oplus\mathcal{F}_{-1}\oplus\mathcal{F}_{0}\oplus\mathcal{F}_{1}\oplus\mathcal{F}_{2},$\\
$\mathcal{F}_{-2}\ =\ \wideparen{(0,\mathcal{S)}}, \ \ 
\mathcal{F}_{-1}\ =\ \wideparen{(\mathcal{A},0)},\ \ 
\mathcal{F}_{0}\ =\ \rm{Instr} (\mathcal{A}), \ \  
\mathcal{F}_1\ =\ (\mathcal{A},0), \ \ 
\mathcal{F}_{2}\ =\ (0,\mathcal{S}).$\end{center}

For each $3$-dimensional structurable algebra $\mathcal{A}$ of type $(2,1),$
the $\rm{AK}$-algebra  $\mathcal{F}(\mathcal{A})$ has dimension $11.$
Let us now  fix the following notations for 
$3$-dimensional structurable algebra $\mathcal{A}$ of type $(2,1):$
\begin{longtable}{lclcl}
$\mathcal{F}_2$&$=$&$(0,\mathcal{S})$&$=$&$\big\langle \varepsilon_4:=(0,e_3)\big\rangle;$\\

$\mathcal{F}_1$&$=$&$(\mathcal{A},0)$&$=$&$\big\langle \varepsilon_1:=(e_1,0), \ \varepsilon_2:=(e_2,0), \ \varepsilon_3:=(e_3,0)\big\rangle;$\\ 

$\mathcal{F}_0$&$=$&$\rm{Instr}(\mathcal{A})$&$=$&$\big\langle 
\varepsilon_5:=T_{e_1}, \ 
\varepsilon_6:=T_{e_2}, \ 
\varepsilon_7:=T_{e_3}\big\rangle;$\\

$\mathcal{F}_{-1}$&$=$&$\wideparen{(\mathcal{A},0)}$&$=$&$\big\langle 
\varepsilon_8:=\wideparen{(e_1,0)}, \ 
\varepsilon_9:=\wideparen{(e_2,0)}, \ 
\varepsilon_{10}:=\wideparen{(e_3,0)}\big\rangle;$\\

$\mathcal{F}_{-2}$&$=$&$\wideparen{(0,\mathcal{S})}$&$=$&$\big\langle \varepsilon_{11}:=\wideparen{(0,e_3)} \big\rangle.$
\end{longtable}

 The following proposition can be obtained by some direct calculations, and it will be useful in our contruction.
 
\begin{proposition}
    \label{(ss)}
Elements $T_{\rm A}$ for a unital $3$-dimensional structurable algebra ${\rm A}$ are given below.

\begin{longtable}{lcllcllcl}

$T_{{\rm A}_1}$&$=$&$\begin{pmatrix} \alpha &0 & 0 \\ \beta & \alpha & 0 \\ 3\gamma & 0 & \alpha \end{pmatrix},$ & 

 $T_{{\rm A}_2}$&$=$&$\begin{pmatrix} \alpha & 0& 0 \\ \beta & \alpha & 3\gamma\\ 3\gamma & 0 & \alpha \end{pmatrix},$ & 

 $T_{{\rm A}_3}$&$=$&$\begin{pmatrix} \alpha & 0& 0 \\ \beta & \alpha +\beta & 0\\ 3\gamma & 0 & \alpha \end{pmatrix},$ 
\end{longtable}
\begin{longtable}{lcllcl}
$T_{{\rm A}_4}$&$=$&$\begin{pmatrix} \alpha & 0 & -3\gamma \\ \beta & \alpha+\beta & 3\gamma \\ 3\gamma & 0 & \alpha \end{pmatrix},$ &

$T_{{\rm A}_5}$&$=$&$\begin{pmatrix} \alpha & 0& 3\gamma \\ \beta& \alpha+\gamma & \beta  \\ 3\gamma & 0 & \alpha \end{pmatrix},$ \\ 
\\

 $T_{{\rm S}_1}$&$=$&$\begin{pmatrix} \alpha & 0& 0\\ 3\beta  & \alpha & 0 \\ 3\gamma & 0 & \alpha \end{pmatrix},$ &

$T_{{\rm S}_2}$&$=$&$\begin{pmatrix} \alpha & 0 & 3\gamma \\ 3\beta & \alpha+\gamma &  -\beta \\ 3\gamma &0 & \alpha \end{pmatrix}$.
\end{longtable}

\end{proposition}
\subsection{$\rm{AK}$ construction for ${\rm A}_1$}

\begin{theorem}
 
$\mathcal{F}({\rm A}_1)$  is an $11$-dimensional Lie algebra  with the  product defined by

\begin{longtable}{lcrlcrlcr}
$[\varepsilon_1,\varepsilon_3]$&$=$&$-2 \varepsilon_4,$&$
[\varepsilon_1,\varepsilon_5]$&$=$&$-\varepsilon_1, $&$ 
[\varepsilon_1,\varepsilon_6]$&$=$&$-\varepsilon_2,$ \\ 
$[\varepsilon_1,\varepsilon_7]$&$=$&$-3\varepsilon_3,$&$ 
[\varepsilon_1,\varepsilon_8]$&$=$&$\varepsilon_5,$&$ 
[\varepsilon_1,\varepsilon_9]$&$=$&$\varepsilon_6,$\\ 
$[\varepsilon_1,\varepsilon_{10}]$&$=$&$-\varepsilon_7,$&$
[\varepsilon_1,\varepsilon_{11}]$&$=$&$-\varepsilon_{10},$&$
[\varepsilon_2,\varepsilon_5]$&$=$&$-\varepsilon_2,$\\
$[\varepsilon_2,\varepsilon_8]$&$=$&$\varepsilon_6, $&$ 
[\varepsilon_3,\varepsilon_5]$&$=$&$ -\varepsilon_3,$&$
[\varepsilon_3,\varepsilon_8]$&$=$&$\varepsilon_7,$\\ 

$[\varepsilon_4,\varepsilon_5]$&$=$&$-2\varepsilon_4,$&$
[\varepsilon_4,\varepsilon_8]$&$=$&$\varepsilon_3,$&$ 
[\varepsilon_5,\varepsilon_8]$&$=$&$-\varepsilon_8,$\\
$[\varepsilon_5,\varepsilon_9]$&$=$&$-\varepsilon_9,$&$
[\varepsilon_5,\varepsilon_{10}]$&$=$&$-\varepsilon_{10},$&$
[\varepsilon_5,\varepsilon_{11}]$&$=$&$-2\varepsilon_{11},$\\ 
$[\varepsilon_6,\varepsilon_8]$&$=$&$-\varepsilon_9,$&$
[\varepsilon_7,\varepsilon_8]$&$=$&$3\varepsilon_{10},$&$
[\varepsilon_8,\varepsilon_{10}]$&$=$&$-2\varepsilon_{11}.$
\end{longtable}
\end{theorem}

\begin{proof}
Based on the $\rm{AK}$-algorithm given in Subsection \ref{AKalg}, we define  
$$\mathcal{F}_0=\left\langle
T_{e_1}:=\begin{pmatrix} 1 & 0 & 0 \\ 0 & 1& 0 \\ 0 & 0 & 1 \end{pmatrix}, \ 
T_{e_2}:=\begin{pmatrix}  0& 0 & 0 \\ 1& 0 & 0 \\ 0 & 0 & 0 \end{pmatrix},\ 
T_{e_3}:=\begin{pmatrix} 0 & 0 &0\\ 0 & 0 & 0 \\ 3& 0 & 0 \end{pmatrix}  \right\rangle.$$
   
It is easy to see that $[\mathcal F_0, \mathcal F_{0}]=0.$

 Below, we provide a description  of 
 $T_{e_i}^\varepsilon,$ 
 $T_{e_i}^\delta,$ 
 $T_{e_i}^{\varepsilon \delta},$ and 
 $V_{e_i,e_j}$:

\begin{longtable}{|lcl|lclcl|}
\hline
   $T_{e_1}(e_{1})$&$=$&$e_1$&   $T_{e_1}^\varepsilon$&$=$&$T_{e_1}-T_{T_{e_1}(e_1)+\overline{T_{e_1}(e_1)}}$&$=$ &$-T_{e_1}$\\

$\overline{T_{e_1}({e_1})}$&$=$&$e_1$&$T_{e_1}^\delta$&$=$&$T_{e_1}+R_{\overline{T_{e_1}({e_1})}}$&$=$&$2T_{e_1}$\\

$R_{e_1}$&$=$&$T_{e_1}$&$T_{e_1}^{\varepsilon\delta}$&$=$&$T_{e_1}^\varepsilon+R_{\overline{T_{e_1}^\varepsilon({e_1})}}$&$=$&$-2T_{e_1}$\\
\hline 

 \hline
    
   $T_{e_2}({e_1})$&$=$&${e_2} $&$ T_{e_2}^\varepsilon$&$=$&$T_{e_2}-T_{T_{e_2}({e_1})+\overline{T_{e_2}({e_1})}}$&$=$&$-T_{e_2}$\\

$\overline{T_{e_2}({e_1})}$&$=$&${e_2}$&   
$T_{e_2}^\delta $&$=$&$T_{e_2}+R_{\overline{T_{e_2}({e_1})}}$&$=$&$2T_{e_2}$\\
   
$R_{e_2}$&$=$&$T_{e_2}$&$T_{e_2}^{\varepsilon\delta}$&$=$&$T_{e_2}^\varepsilon+R_{\overline{T_{e_2}^\varepsilon({e_1})}}$&$=$&$-2T_{e_2}$\\

   \hline
    
\hline
$T_{e_3}({e_1})$&$=$&$3{e_3}$&$    T_{e_3}^\varepsilon$&$=$&$T_{e_3}-T_{T_{e_3}({e_1})+\overline{T_{e_3}({e_1})}}$&$=$&$T_{e_3}$\\

$\overline{T_{e_3}({e_1}})$&$=$&$-3{e_3}$&$ T_{e_3}^\delta$&$=$&$T_{e_3}+R_{\overline{T_{e_3}({e_1})}}$&$=$&$0$\\
   
$R_{e_3}$&$=$&$\frac{1}{3}T_{e_3}$&$T_{e_3}^{\varepsilon\delta}$&$=$&$T_{e_3}^\varepsilon+R_{\overline{T_{e_3}^\varepsilon ({e_1})}}$&$=$&$0$\\
\hline\end{longtable}

\begin{longtable}{lcllcllcl}

$V_{{e_1},{e_1}}$&$=$&$T_{e_1},$&$ V_{{e_1},{e_2}}$&$=$&$T_{e_2},$&$ V_{{e_1},{e_3}}$&$=$&$-T_{e_3},$\\ 
$V_{{e_2},{e_1}}$&$=$&$T_{e_2},$&$ V_{{e_2},{e_2}}$&$=$&$0, $&$ V_{{e_2},{e_3}}$&$=$&$0,$\\
$V_{{e_3},{e_1}}$&$=$&$T_{e_3},$&$ V_{{e_3},{e_2}}$&$=$&$0,$&$ V_{{e_3},{e_3}}$&$=$&$0.$
\end{longtable}

We are now in a position to determine the multiplication table of $\mathcal F({\rm A}_1).$

\begin{enumerate}[(I)]
    \item As the first step, we define  $[\mathcal F_0, \mathcal F_{1}+\mathcal F_{2}],$ i.e., $[E,(x,s)]=(E(x), E^\delta(s)).$

\begin{longtable}{lclclcl}
$[\varepsilon_5,\varepsilon_1]$&$=$&$  [T_{e_1}, ({e_1},0)]$&$=$&$ \big(T_{e_1}({e_1}), T_{e_1}^\delta(0) \big)$&$=$&$({e_1}, 0) \  =\ \varepsilon_1;$\\

$[\varepsilon_5,\varepsilon_2]$&$=$&$[T_{e_1}, ({e_2},0)]$&$=$&$(T_{e_1}({e_2}), T_{e_1}^\delta(0)$&$=$&$({e_2},\ 0)\ =\ \varepsilon_2;$\\

$[\varepsilon_5,\varepsilon_3]$&$=$&$[T_{e_1}, ({e_3},0)]$&$=$&$(T_{e_1}({e_3}), T_{e_1}^\delta (0))$&$=$&$({e_3},\ 0)\ =\  \varepsilon_3;$\\
$[\varepsilon_5,\varepsilon_4]$&$=$&$[T_{e_1}, (0,{e_3})]$&$=$&$
(T_{e_1}(0), T_{e_1}^\delta (e_3))$&$=$&$(0,\  2e_3) \ =\ 2\varepsilon_4;$\\

$[\varepsilon_6,\varepsilon_1]$&$=$&$[T_{e_2}, ({e_1},0)]$&$=$&$
(T_{e_2}({e_1}), T_{e_2}^\delta (0))$&$=$&$(e_2,\ 0)\ =\ \varepsilon_2;$\\

$[\varepsilon_6,\varepsilon_2]$&$=$&$[T_{e_2}, ({e_2},0)]$&$=$&$
(T_{e_2}({e_2}), T_{e_2}^\delta (0))$&$=$&$(0,\ 0) \ =\ 0;$\\

$[\varepsilon_6,\varepsilon_3]$&$=$&$[T_{e_2}, ({e_3},0)]$&$=$&$
(T_{e_2}({e_3}), T_{e_2}^\delta (0))$&$=$&$(0,\ 0)\  =\ 0;$\\
$[\varepsilon_6,\varepsilon_4]$&$=$&$[T_{e_2}, (0,{e_3})]$&$=$&$
(T_{e_2}(0),T_{e_2}^\delta ({e_3}))$&$=$&$(0,\ 0) \ =\ 0;$\\

$[\varepsilon_7,\varepsilon_1]$&$=$&$[T_{e_3}, ({e_1},0)]$&$=$&$
(T_{e_3}({e_1}), T_{e_3}^\delta (0))$&$=$&$(3{e_3},\ 0) \ =\ 3\varepsilon_3;$\\

$[\varepsilon_7,\varepsilon_2]$&$=$&$[T_{e_3}, ({e_2},0)]$&$=$&$
(T_{e_3}({e_2}), T_{e_3}^\delta(0))$&$=$&$(0,\ 0) \ = \ 0;$\\

$[\varepsilon_7,\varepsilon_3]$&$=$&$[T_{e_3}, ({e_3},0)]$&$=$&$
(T_{e_3}({e_3}), T_{e_3}^\delta (0))$&$=$&$(0,\ 0)\ =\ 0;$\\

$[\varepsilon_7,\varepsilon_4]$&$=$&$[T_{e_3}, (0,{e_3})]$&$=$&$
(T_{e_3}(0), T_{e_3}^\delta ({e_3}))$&$=$&$(0,\ 0)\ =\ 0.$
\end{longtable}

  \item As the second step, we define   $[\mathcal F_0, \mathcal F_{-2}+\mathcal F_{-1}],$ i.e., 
  $[{E}, \wideparen {(x,s)} ] =
\wideparen{({E}^\varepsilon(x), {E}^{\varepsilon\delta}(s))}.$

\begin{longtable}{lclclcl}
$[\varepsilon_5,\varepsilon_8]$&$=$&$[T_{e_1}, \wideparen{({e_1},0)}] $&$=$&$ \wideparen{(T_{e_1}^\varepsilon({e_1}),\  0)}$&$=$&$
\wideparen{(-T_{e_1}({e_1}), \ 0)}\ =\ \wideparen{(-{e_1}, \ 0)}
\ =\ -\varepsilon_8;$ \\

$[\varepsilon_5,\varepsilon_9]$&$=$&$[T_{e_1}, \wideparen{({e_2},0)}] $&$=$&$ \wideparen{(T_{e_1}^\varepsilon({e_2}),\  0)}$&$=$&$
\wideparen{(-T_{e_1}({e_2}), \ 0)}\ =\ \wideparen{(-{e_2}, \ 0)}
\ =\ -\varepsilon_9;$ \\

$[\varepsilon_5,\varepsilon_{10}]$&$=$&$[T_{e_1}, \wideparen{({e_3},0)}] $&$=$&$ \wideparen{(T_{e_1}^\varepsilon({e_3}),\  0)}$&$=$&$
\wideparen{(-T_{e_1}({e_3}), \ 0)}\ =\ \wideparen{(-{e_3}, \ 0)}
\  =\ -\varepsilon_{10};$ \\

$[\varepsilon_5,\varepsilon_{11}]$&$=$&$[T_{e_1}, \wideparen{(0,{e_3})} ]$&$=$&$
\wideparen{(0,\  T_{e_1}^{\varepsilon\delta} ({e_3}))} $&$=$&$
\wideparen{(0,\ -2{e_3})}\ =\ -2\varepsilon_{11};$\\

$[\varepsilon_6,\varepsilon_8]$&$=$&$[T_{e_2}, \wideparen{({e_1},0)}] $&$=$&$ \wideparen{(T_{e_2}^\varepsilon({e_1}),\  0)}$&$=$&$
\wideparen{(-{e_2}, \ 0)}
\ =\ -\varepsilon_9;$\\

$[\varepsilon_6,\varepsilon_9]$&$=$&$[T_{e_2}, \wideparen{({e_2},0)}] $&$=$&$ \wideparen{(T_{e_2}^\varepsilon({e_2}),\  0)}$&$=$&$
\wideparen{(0, \ 0)}
\ =\ 0;$ \\

$[\varepsilon_6,\varepsilon_{10}]$&$=$&$[T_{e_2},\wideparen{({e_3},0)}] $&$=$&$ \wideparen{(T_{e_2}^\varepsilon({e_3}),\  0)}$&$=$&$
\wideparen{(0, \ 0)}
\ =\ 0;$\\

$[\varepsilon_6,\varepsilon_{11}]$&$=$&$[T_{e_2}, \wideparen{(0,{e_3})} ]$&$=$&$
\wideparen{(0,\  T_{e_2}^{\varepsilon\delta} ({e_3}))}$&$=$&$
\wideparen{(0,\ 0)} \ =\ 0;$\\
$[\varepsilon_7,\varepsilon_8]$&$=$&$[T_{e_3}, \wideparen{({e_1},0)}] $&$=$&$ \wideparen{(T_{e_3}^\varepsilon({e_1}),\  0)}$&$=$&$
\wideparen{(3{e_3}, \ 0)}
\ =\ 3\varepsilon_{10};$\\

$[\varepsilon_7,\varepsilon_9]$&$=$&$[T_{e_3}, \wideparen{({e_2},0)}] $&$=$&$ \wideparen{(T_{e_3}^\varepsilon({e_2}),\  0)}$&$=$&$
\wideparen{(0, \ 0)} \ =\ 0;$\\

$[\varepsilon_7,\varepsilon_{10}]$&$=$&$[T_{e_3}, \wideparen{({e_3},0)}] $&$=$&$ \wideparen{(T_{e_3}^\varepsilon({e_3}),\  0)}$&$=$&$\wideparen{(0, \ 0)}
\ =\ 0;$\\

$[\varepsilon_7,\varepsilon_{11}]$&$=$&$[T_{e_3}, \wideparen{(0,{e_3})} ]$&$=$&$\wideparen{(0, \ T_{e_3}^{\varepsilon\delta} ({e_3}))} $&$=$&$\wideparen{(0,\ 0)} \ =\ 0.$
\end{longtable}

\item As the third step, we define  $[ \mathcal F_{1}+\mathcal F_{2}, \mathcal F_{1}+\mathcal F_{2}],$ i.e., 
$[(x,s), (y,r)] = (0, x\overline{y} - y\overline{x}).$

\begin{longtable}{lclcl}

$[\varepsilon_1,\varepsilon_2]$&$=$&$ [({e_1},0), ({e_2},0)] $&$=$&$(0,\ 0) \ =\ 0;$\\
$[\varepsilon_1,\varepsilon_3]$&$=$&$[({e_1},0), ({e_3},0)] $&$=$&$(0,\ -2e_3) \ =\ -2\varepsilon_4;$\\

$[\varepsilon_2,\varepsilon_3]$&$=$&$[({e_2},0), ({e_3},0)] $&$=$&$(0,\ 0) \ =\ 0;$\\

\end{longtable}

\item We define   $[ \mathcal F_{-2}+\mathcal F_{-1}, \mathcal F_{-2}+\mathcal F_{-1}],$ i.e., 
$[\wideparen{(x,s)}, \wideparen{(y,r)} ] = \wideparen{(0, x\overline{y} - y\overline{x})}.$

\begin{longtable}{lclcl}

$[\varepsilon_8,\varepsilon_9]$&$=$&$[\wideparen{({e_1},0)}, \wideparen{({e_2},0)}] $&$=$&$\wideparen{(0,\ 0)} \ =\ 0;$\\
$[\varepsilon_8,\varepsilon_{10}]$&$=$&$[\wideparen{({e_1},0)}, \wideparen{({e_3},0)}] $&$=$&$\wideparen{(0,\ -2e_3)} \ =\ -2\varepsilon_{11};$\\

$[\varepsilon_9,\varepsilon_{10}]$&$=$&$[\wideparen{({e_2},0)}, \wideparen{({e_3},0)}] $&$=$&$\wideparen{(0,\ 0)} \ =\ 0.$\\

\end{longtable}

\item We define  $[ \mathcal F_{1}+\mathcal F_{2}, \mathcal F_{-2}+\mathcal F_{-1}],$ i.e., 
       $[(x,s), \wideparen{(y,r)}] = -\wideparen{(r x, 0)}  + V_{x,y} + L_s L_r+ (s y, 0).$

\begin{longtable}{lclcl}
    
$[\varepsilon_1,\varepsilon_8]$&$=$&$  [({e_1},0),\wideparen{({e_1},0)}]$&$=$&$V_{{e_1},{e_1}}\ =\ T_{e_1}\ =\ \varepsilon_5;$\\
$[\varepsilon_1,\varepsilon_9]$&$=$&$  [({e_1},0),\wideparen{({e_2},0)}]$&$=$&$
V_{{e_1},{e_2}}\ =\ T_{e_2}\ =\ \varepsilon_6;$\\
$[\varepsilon_1,\varepsilon_{10}]$&$=$&$[({e_1},0),\wideparen{({e_3},0)}]$&$=$&$
V_{{e_1},{e_3}}\ =\ -T_{e_3}\ =\ -\varepsilon_7;$\\
$[\varepsilon_1,\varepsilon_{11}]$&$=$&$[({e_1},0),\wideparen{(0,{e_3})}]$&$=$&$ -\wideparen{({e_3},0)} \ =\ -\varepsilon_{10};$\\

$[\varepsilon_2,\varepsilon_8]$&$=$&$ [({e_2},0),\wideparen{({e_1},0)}]$&$=$&$V_{{e_2},{e_1}}\ =\ T_{e_2}\ =\ \varepsilon_6;$\\

$[\varepsilon_2,\varepsilon_9]$&$=$&$  [({e_2},0),\wideparen{({e_2},0)}]$&$=$&$
V_{{e_2},{e_2}}\ =\ 0;$\\
$[\varepsilon_2,\varepsilon_{10}]$&$=$&$[({e_2},0),\wideparen{({e_3},0)}]$&$=$&$V_{{e_2},{e_3}}\ =\ 0;$\\
$[\varepsilon_2,\varepsilon_{11}]$&$=$&$[({e_2},0),\wideparen{(0,{e_3})}]$&$=$&$-\wideparen{({e_3}{e_2},0)}\ =\ -\wideparen{(0,0)} \ =\ 0;$\\

$[\varepsilon_3,\varepsilon_8]$&$=$&$[({e_3},0),\wideparen{({e_1},0)}]$&$=$&$V_{{e_3},{e_1}}\ =\ T_{e_3} \ =\ \varepsilon_7;$\\
$[\varepsilon_3,\varepsilon_9]$&$=$&$[({e_3},0),\wideparen{({e_2},0)}]$&$=$&$V_{{e_3},{e_2}}\ =\ 0;$\\
$[\varepsilon_3,\varepsilon_{10}]$&$=$&$[({e_3},0),\wideparen{({e_3},0)}]$&$=$&$V_{{e_3},{e_3}}\ =\ 0;$\\

$[\varepsilon_3,\varepsilon_{11}]$&$=$&$[({e_3},0),\wideparen{(0,{e_3})}]$&$=$&$
-\wideparen{({e_3}{e_3},0)}\ =\ -\wideparen{(0,0)} \ =\ 0;$\\

$[\varepsilon_4,\varepsilon_8]$&$=$&$[(0,{e_3}),\wideparen{({e_1},0)}]$&$=$&$
({e_3},0) \ = \ \varepsilon_3;$\\
 
$[\varepsilon_4,\varepsilon_9]$&$=$&$[(0,{e_3}),\wideparen{({e_2},0)}]$&$=$&$({e_3}{e_2},0)=(0,0)\ =\ 0;$\\
$[\varepsilon_4,\varepsilon_{10}]$&$=$&$[(0,{e_3}),\wideparen{({e_3},0)}]$&$=$&$
({e_3}{e_3},0)\ =\ (0,0) \  =\ 0;$\\

$[\varepsilon_4,\varepsilon_{11}]$&$=$&$[(0,{e_3}),\wideparen{(0,{e_3})}]$&$=$&$L_{e_3}L_{e_3}\ =\ 0.$\\
\end{longtable}
\end{enumerate}
\end{proof}

\begin{remark}
$\mathcal{F}({\rm A}_1)$ is perfect and 
$\mathcal{F}({\rm A}_1) = S \ltimes R,$
where:
\begin{itemize}
\item $S=\big\langle\varepsilon_1,  \ 
\varepsilon_5, \ 
\varepsilon_8\big\rangle \cong \mathfrak{sl}_2;$
\item $R = \big\langle\varepsilon_2, \ 
\varepsilon_3,\  
\varepsilon_4,\ 
\varepsilon_6,\ 
\varepsilon_7,\ 
\varepsilon_9,\  
\varepsilon_{10},\ 
\varepsilon_{11}\big\rangle$ is the  abelian radical.
\end{itemize}

\end{remark}

\subsection{$\rm{AK}$ construction for ${\rm A}_2$}

\begin{theorem}
$\mathcal{F}({\rm A}_2)$ is an $11$-dimensional Lie algebra  with the  multiplication given by

\begin{longtable}{lcrlcrlcr}
$[\varepsilon_1,\varepsilon_3]$&$=$&$-2 \varepsilon_4,$&$
[\varepsilon_1,\varepsilon_5]$&$=$&$-\varepsilon_1, $&$ 
[\varepsilon_1,\varepsilon_6]$&$=$&$-\varepsilon_2,$ \\ 

$[\varepsilon_1,\varepsilon_7]$&$=$&$-3\varepsilon_3,$&$ 
[\varepsilon_1,\varepsilon_8]$&$=$&$\varepsilon_5,$&$ 
[\varepsilon_1,\varepsilon_9]$&$=$&$\varepsilon_6,$\\ 
$[\varepsilon_1,\varepsilon_{10}]$&$=$&$-\varepsilon_7,$&$
[\varepsilon_1,\varepsilon_{11}]$&$=$&$-\varepsilon_{10},$&$
[\varepsilon_2,\varepsilon_5]$&$=$&$-\varepsilon_2,$\\

$[\varepsilon_2,\varepsilon_8]$&$=$&$\varepsilon_6, $&$ 
[\varepsilon_3,\varepsilon_5]$&$=$&$ -\varepsilon_3,$&$  [\varepsilon_3,\varepsilon_7]$&$=$&$-3\varepsilon_2,$\\ 

$[\varepsilon_3,\varepsilon_8]$&$=$&$\varepsilon_7,$&$  
[\varepsilon_3,\varepsilon_{10}]$&$=$&$-\varepsilon_6,$&$ 
[\varepsilon_3,\varepsilon_{11}]$&$=$&$-\varepsilon_9,$\\

$[\varepsilon_4,\varepsilon_5]$&$=$&$-2\varepsilon_4,$&$
[\varepsilon_4,\varepsilon_8]$&$=$&$\varepsilon_3,$&$  [\varepsilon_4,\varepsilon_{10}]$&$=$&$\varepsilon_2,$\\  

$[\varepsilon_4,\varepsilon_{11}]$&$=$&$\varepsilon_6,$&$
[\varepsilon_5,\varepsilon_8]$&$=$&$-\varepsilon_8,$&$ 
[\varepsilon_5,\varepsilon_9]$&$=$&$-\varepsilon_9,$\\ 

$[\varepsilon_5,\varepsilon_{10}]$&$=$&$-\varepsilon_{10},$&$ [\varepsilon_5,\varepsilon_{11}]$&$=$&$-2\varepsilon_{11},$&$ 
[\varepsilon_6,\varepsilon_8]$&$=$&$-\varepsilon_9,$\\

$[\varepsilon_7,\varepsilon_8]$&$=$&$3\varepsilon_{10},$&$
[\varepsilon_7,\varepsilon_{10}]$&$=$&$3\varepsilon_9,$&$ 
[\varepsilon_8,\varepsilon_{10}]$&$=$&$-2\varepsilon_{11}.$
\end{longtable}

\end{theorem}
\begin{proof}
Following the $\rm{AK}$-algorithm given in subsection \ref{AKalg}, we define  
$$\mathcal{F}_0=\left\langle
T_{e_1}:=\begin{pmatrix} 1 & 0 & 0 \\ 0 & 1& 0 \\ 0 & 0 & 1 \end{pmatrix}, \ 
T_{e_2}:=\begin{pmatrix}  0& 0 & 0 \\ 1& 0 & 0 \\ 0 & 0 & 0 \end{pmatrix},\ 
T_{e_3}:=\begin{pmatrix} 0 & 0 &0\\ 0 & 0 & 3 \\ 3& 0 & 0 \end{pmatrix}  \right\rangle.$$

 It is easy to see that $[\mathcal F_0, \mathcal F_{0}]=0.$

 Let us note that $L_{e_3}L_{e_3}=T_{e_2}$ and   
 list elements $T_{e_i}^\varepsilon,$ 
 $T_{e_i}^\delta,$  
 $T_{e_i}^{\varepsilon \delta},$ and 
 $V_{e_i,e_j}$   below:

\begin{longtable}{|lcl|lclcl|}
\hline
   $T_{e_1}(e_{1})$&$=$&$e_1$&   $T_{e_1}^\varepsilon$&$=$&$T_{e_1}-T_{T_{e_1}(e_1)+\overline{T_{e_1}(e_1)}}$&$=$ &$-T_{e_1}$\\

$\overline{T_{e_1}({e_1})}$&$=$&$e_1$&$T_{e_1}^\delta$&$=$&$T_{e_1}+R_{\overline{T_{e_1}({e_1})}}$&$=$&$2T_{e_1}$\\

$R_{e_1}$&$=$&$T_{e_1}$&$T_{e_1}^{\varepsilon\delta}$&$=$&$T_{e_1}^\varepsilon+R_{\overline{T_{e_1}^\varepsilon({e_1})}}$&$=$&$-2T_{e_1}$\\
\hline 

 \hline
    
   $T_{e_2}({e_1})$&$=$&${e_2} $&$ T_{e_2}^\varepsilon$&$=$&$T_{e_2}-T_{T_{e_2}({e_1})+\overline{T_{e_2}({e_1})}}$&$=$&$-T_{e_2}$\\

$\overline{T_{e_2}({e_1})}$&$=$&${e_2}$&   
$T_{e_2}^\delta $&$=$&$T_{e_2}+R_{\overline{T_{e_2}({e_1})}}$&$=$&$2T_{e_2}$\\
   
$R_{e_2}$&$=$&$T_{e_2}$&$T_{e_2}^{\varepsilon\delta}$&$=$&$T_{e_2}^\varepsilon+R_{\overline{T_{e_2}^\varepsilon({e_1})}}$&$=$&$-2T_{e_2}$\\

   \hline
    
\hline
$T_{e_3}({e_1})$&$=$&$3{e_3}$&$    T_{e_3}^\varepsilon$&$=$&$T_{e_3}-T_{T_{e_3}({e_1})+\overline{T_{e_3}({e_1})}}$&$=$&$T_{e_3}$\\

$\overline{T_{e_3}({e_1}})$&$=$&$-3{e_3}$&$ T_{e_3}^\delta$&$=$&$T_{e_3}+R_{\overline{T_{e_3}({e_1})}}$&$=$&$0$\\
   
$R_{e_3}$&$=$&$\frac{1}{3}T_{e_3}$&$T_{e_3}^{\varepsilon\delta}$&$=$&$T_{e_3}^\varepsilon+R_{\overline{T_{e_3}^\varepsilon({e_1})}}$&$=$&$0$\\
\hline\end{longtable}

\begin{longtable}{lcllcllcl}

$V_{{e_1},{e_1}}$&$=$&$T_{e_1},$&$ V_{{e_1},{e_2}}$&$=$&$T_{e_2},$&$ V_{{e_1},{e_3}}$&$=$&$-T_{e_3},$\\ 
$V_{{e_2},{e_1}}$&$=$&$T_{e_2},$&$ V_{{e_2},{e_2}}$&$=$&$0, $&$ V_{{e_2},{e_3}}$&$=$&$0,$\\
$V_{{e_3},{e_1}}$&$=$&$T_{e_3},$&$ V_{{e_3},{e_2}}$&$=$&$0,$&$ V_{{e_3},{e_3}}$&$=$&$-T_{e_2}.$
\end{longtable}

Now we are ready to determine the multiplication table of $\mathcal F({\rm A}_2).$

\begin{enumerate}[(I)]
    \item First, we define   $[\mathcal F_0, \mathcal F_{1}+\mathcal F_{2}],$ i.e., $[E,(x,s)]=(E(x), E^\delta(s)).$

\begin{longtable}{lclclcl}
$[\varepsilon_5,\varepsilon_1]$&$=$&$  [T_{e_1}, ({e_1},0)]$&$=$&$ 
\big(T_{e_1}({e_1}), T_{e_1}^\delta(0) \big)$&$=$&$
({e_1}, 0) \ =\ \varepsilon_1;$\\

$[\varepsilon_5,\varepsilon_2]$&$=$&$[T_{e_1}, ({e_2},0)]$&$=$&$
(T_{e_1}({e_2}), T_{e_1}^\delta(0))$&$=$&$
({e_2},\ 0)\ =\ \varepsilon_2;$\\

$[\varepsilon_5,\varepsilon_3]$&$=$&$[T_{e_1}, ({e_3},0)]$&$=$&$
(T_{e_1}({e_3}), T_{e_1}^\delta (0))$&$=$&$
({e_3},\ 0)\ =\  \varepsilon_3;$\\
$[\varepsilon_5,\varepsilon_4]$&$=$&$[T_{e_1}, (0,{e_3})]$&$=$&$
(T_{e_1}(0), T_{e_1}^\delta (e_3))$&$=$&$(0,\  2e_3) \ =\ 2\varepsilon_4;$\\

$[\varepsilon_6,\varepsilon_1]$&$=$&$[T_{e_2}, ({e_1},0)]$&$=$&$
(T_{e_2}({e_1}), T_{e_2}^\delta (0))$&$=$&$
(e_2,\ 0) \ =\ \varepsilon_2;$\\

$[\varepsilon_6,\varepsilon_2]$&$=$&$[T_{e_2}, ({e_2},0)]$&$=$&$ 
(T_{e_2}({e_2}), T_{e_2}^\delta (0))$&$=$&$(0,\ 0)\ =\ 0;$\\

$[\varepsilon_6,\varepsilon_3]$&$=$&$[T_{e_2}, ({e_3},\ 0)]$&$=$&$
(T_{e_2}({e_3}), T_{e_2}^\delta (0))$&$=$&$(0,\ 0)\  =\ 0;$\\
$[\varepsilon_6,\varepsilon_4]$&$=$&$[T_{e_2}, (0,\ {e_3})]$&$=$&$
(T_{e_2}(0),T_{e_2}^\delta ({e_3}))$&$=$&$(0,\ 0) \ =\ 0;$\\

$[\varepsilon_7,\varepsilon_1]$&$=$&$[T_{e_3}, ({e_1},0)]$&$=$&$
(T_{e_3}({e_1}), T_{e_3}^\delta (0))$&$=$&$(3{e_3},\ 0) \ =\ 3\varepsilon_3;$\\

$[\varepsilon_7,\varepsilon_2]$&$=$&$[T_{e_3}, ({e_2},0)]$&$=$&$
(T_{e_3}({e_2}), T_{e_3}^\delta(0))$&$=$&$(0,\ 0)\ =\ 0;$\\

$[\varepsilon_7,\varepsilon_3]$&$=$&$[T_{e_3}, ({e_3},0)]$&$=$&$
(T_{e_3}({e_3}), T_{e_3}^\delta (0))$&$=$&$
(3{e_2},\  0) \ =\ 3\varepsilon_2;$\\

$[\varepsilon_7,\varepsilon_4]$&$=$&$[T_{e_3}, (0,{e_3})]$&$=$&$
(T_{e_3}(0), T_{e_3}^\delta ({e_3}))$&$=$&$
(0,\ 0)\ =\ 0.$
\end{longtable}

  \item Second, we define   $[\mathcal F_0, \mathcal F_{-2}+\mathcal F_{-1}],$ i.e., 
  $[{E}, \wideparen {(x,s)} ] =
\wideparen{({E}^\varepsilon(x), {E}^{\varepsilon\delta}(s))}.$

\begin{longtable}{lclclcl}
$[\varepsilon_5,\varepsilon_8]$&$=$&$[T_{e_1}, \wideparen{({e_1},0)}]$&$=$&$
\wideparen{(T_{e_1}^\varepsilon({e_1}),\  0)}$&$=$&$
\wideparen{(-T_{e_1}({e_1}), \ 0)}\ =\ \wideparen{(-{e_1}, \ 0)}
\ =\ -\varepsilon_8;$ \\

$[\varepsilon_5,\varepsilon_9]$&$=$&$[T_{e_1}, \wideparen{({e_2},0)}] $&$=$&$ \wideparen{(T_{e_1}^\varepsilon({e_2}),\  0)}$&$=$&$
\wideparen{(-T_{e_1}({e_2}), \ 0)}\ =\ \wideparen{(-{e_2}, \ 0)}
\ =\ -\varepsilon_9;$ \\

$[\varepsilon_5,\varepsilon_{10}]$&$=$&$[T_{e_1}, \wideparen{({e_3},0)}] $&$=$&$ \wideparen{(T_{e_1}^\varepsilon({e_3}),\  0)}$&$=$&$
\wideparen{(-T_{e_1}({e_3}), \ 0)}\ =\ \wideparen{(-{e_3}, \ 0)}
\ =\ -\varepsilon_{10};$ \\

$[\varepsilon_5,\varepsilon_{11}]$&$=$&$[T_{e_1}, \wideparen{(0,{e_3})} ]$&$=$&$
\wideparen{(0, T_{e_1}^{\varepsilon\delta} ({e_3}))} $&$=$&$
\wideparen{(0,\ -2{e_3})} \ =\ -2\varepsilon_{11};$\\

$[\varepsilon_6,\varepsilon_8]$&$=$&$[T_{e_2}, \wideparen{({e_1},0)}] $&$=$&$ \wideparen{(T_{e_2}^\varepsilon({e_1}),\  0)}$&$=$&$
\wideparen{(-{e_2}, \ 0)}
\ =\ -\varepsilon_9;$\\

$[\varepsilon_6,\varepsilon_9]$&$=$&$[T_{e_2}, \wideparen{({e_2},0)}] $&$=$&$ \wideparen{(T_{e_2}^\varepsilon({e_2}),\  0)}$&$=$&$\wideparen{(0, \ 0)}
\ =\ 0;$ \\

$[\varepsilon_6,\varepsilon_{10}]$&$=$&$[T_{e_2},\wideparen{({e_3},0)}] $&$=$&$ \wideparen{(T_{e_2}^\varepsilon({e_3}),\  0)}$&$=$&$
\wideparen{(0, \ 0)}
\ =\ 0;$\\

$[\varepsilon_6,\varepsilon_{11}]$&$=$&$[T_{e_2}, \wideparen{(0,{e_3})} ]$&$=$&$\wideparen{(0, T_{e_2}^{\varepsilon\delta} ({e_3}))} $&$=$&$\wideparen{(0,\ 0)} \ =\ 0;$\\

$[\varepsilon_7,\varepsilon_8]$&$=$&$[T_{e_3}, \wideparen{({e_1},0)}] $&$=$&$ \wideparen{(T_{e_3}^\varepsilon({e_1}),\  0)}$&$=$&$
\wideparen{(3{e_3}, \ 0)} \ =\ 3\varepsilon_{10};$\\

$[\varepsilon_7,\varepsilon_9]$&$=$&$[T_{e_3}, \wideparen{({e_2},0)}] $&$=$&$ \wideparen{(T_{e_3}^\varepsilon({e_2}),\  0)}$&$=$&$\wideparen{(0, \ 0)}
\ =\ 0;$\\

$[\varepsilon_7,\varepsilon_{10}]$&$=$&$[T_{e_3}, \wideparen{({e_3},0)}] $&$=$&$ \wideparen{(T_{e_3}^\varepsilon({e_3}),\  0)}$&$=$&$
\wideparen{(3{e_2}, \ 0)}\ =\ 3\varepsilon_9;$\\

$[\varepsilon_7,\varepsilon_{11}]$&$=$&$[T_{e_3}, \wideparen{(0,{e_3})} ]$&$=$&$
\wideparen{(0, T_{e_3}^{\varepsilon\delta} ({e_3}))} $&$=$&$ \wideparen{(0,\ 0)} \ =\ 0.$
\end{longtable}

\item Third, we define   $[ \mathcal F_{1}+\mathcal F_{2}, \mathcal F_{1}+\mathcal F_{2}],$ i.e., 
$[(x,s), (y,r)] = (0, x\overline{y} - y\overline{x}).$

\begin{longtable}{lclcl}

$[\varepsilon_1,\varepsilon_2]$&$=$&$ [({e_1},0), ({e_2},0)] $&$=$&$(0,\ 0) \ =\ 0;$\\
$[\varepsilon_1,\varepsilon_3]$&$=$&$[({e_1},0), ({e_3},0)] $&$=$&$(0,\ -2e_3) \ =\ -2\varepsilon_4;$\\

$[\varepsilon_2,\varepsilon_3]$&$=$&$[({e_2},0), ({e_3},0)] $&$=$&$(0,\ 0) \ =\ 0.$\\

\end{longtable}

\item Fourth, we define   $[ \mathcal F_{-2}+\mathcal F_{-1}, \mathcal F_{-2}+\mathcal F_{-1}],$ i.e., 
$[\wideparen{(x,s)}, \wideparen{(y,r)} ] = \wideparen{(0, x\overline{y} - y\overline{x})}.$

\begin{longtable}{lclcl}

$[\varepsilon_8,\varepsilon_9]$&$=$&$[\wideparen{({e_1},0)}, \wideparen{({e_2},0)}] $&$=$&$\wideparen{(0,\ 0)}\ =\ 0;$\\
$[\varepsilon_8,\varepsilon_{10}]$&$=$&$[\wideparen{({e_1},0)}, \wideparen{({e_3},0)}] $&$=$&$\wideparen{(0,\ -2e_3)} \ =\ -2\varepsilon_{11};$\\

$[\varepsilon_9,\varepsilon_{10}]$&$=$&$[\wideparen{({e_2},0)}, \wideparen{({e_3},0)}] $&$=$&$\wideparen{(0,\ 0)} \ =\ 0.$\\

\end{longtable}

\item End, we define   $[ \mathcal F_{1}+\mathcal F_{2}, \mathcal F_{-2}+\mathcal F_{-1}],$ i.e., 
    $[(x,s), \wideparen{(y,r)}] = -\wideparen{(r x, 0)}  + V_{x,y} + L_s L_r+ (s y, 0).$

\begin{longtable}{lclcl}
    
$[\varepsilon_1,\varepsilon_8]$&$=$&$  [({e_1},0),\wideparen{({e_1},0)}]$&$=$&$
V_{{e_1},{e_1}}\ =\ T_{e_1}\ =\ \varepsilon_5;$\\

$[\varepsilon_1,\varepsilon_9]$&$=$&$  [({e_1},0),\wideparen{({e_2},0)}]$&$=$&$
V_{{e_1},{e_2}}\ =\ T_{e_2\ }\ =\varepsilon_6;$\\

$[\varepsilon_1,\varepsilon_{10}]$&$=$&$[({e_1},0),\wideparen{({e_3},0)}]$&$=$&$
V_{{e_1},{e_3}}\ =\ -T_{e_3}\ =\ -\varepsilon_7;$\\

$[\varepsilon_1,\varepsilon_{11}]$&$=$&$[({e_1},0),\wideparen{(0,  {e_3})}]$&$=$&$
-\wideparen{({e_3},\ 0)}\ =\ -\varepsilon_{10};$\\

$[\varepsilon_2,\varepsilon_8]$&$=$&$ [({e_2},\ 0),\wideparen{({e_1},\ 0)}]$&$=$&$
V_{{e_2},{e_1}}\ =\ T_{e_2}\ =\ \varepsilon_6;$\\

$[\varepsilon_2,\varepsilon_9]$&$=$&$  [({e_2},0),\wideparen{({e_2},0)}]$&$=$&$
V_{{e_2},{e_2}}\ =\ 0;$\\
$[\varepsilon_2,\varepsilon_{10}]$&$=$&$[({e_2},0),\wideparen{({e_3},0)}]$&$=$&$
V_{{e_2},{e_3}}\ =\ 0;$\\
$[\varepsilon_2,\varepsilon_{11}]$&$=$&$[({e_2},0),\wideparen{(0,{e_3})}]$&$=$&$
-\wideparen{({e_3}{e_2},\ 0)}\ =\ -\wideparen{(0,\ 0)} \ =\ 0;$\\

$[\varepsilon_3,\varepsilon_8]$&$=$&$[({e_3},0),\wideparen{({e_1},0)}]$&$=$&$
V_{{e_3},{e_1}}\ =\ T_{e_3}\ =\ \varepsilon_7;$\\
$[\varepsilon_3,\varepsilon_9]$&$=$&$[({e_3},0),\wideparen{({e_2},0)}]$&$=$&$
V_{{e_3},{e_2}}\ =\ 0;$\\
$[\varepsilon_3,\varepsilon_{10}]$&$=$&$[({e_3},0),\wideparen{({e_3},0)}]$&$=$&$
V_{{e_3},{e_3}}\ =\ -T_{e_2}\ =\ -\varepsilon_6;$\\

$[\varepsilon_3,\varepsilon_{11}]$&$=$&$[({e_3},0),\wideparen{(0,{e_3})}]$&$=$&$
-\wideparen{({e_3}{e_3},\ 0)}\ =\ -\wideparen{({e_2},\ 0)} \ =\ 
-\varepsilon_9;$\\

$[\varepsilon_4,\varepsilon_8]$&$=$&$[(0,{e_3}),\wideparen{({e_1}, 0)}]$&$=$&$
({e_3},\ 0) \ =\ \varepsilon_3;$\\
 
$[\varepsilon_4,\varepsilon_9]$&$=$&$[(0,{e_3}),\wideparen{({e_2},0)}]$&$=$&$({e_3}{e_2},\ 0)\ =\ (0,\ 0)\ =\ 0;$\\
$[\varepsilon_4,\varepsilon_{10}]$&$=$&$[(0,{e_3}),\wideparen{({e_3},0)}]$&$=$&$
({e_3}{e_3},\ 0)\ =\ ({e_2},\ 0)  \ =\ \varepsilon_2;$\\

$[\varepsilon_4,\varepsilon_{11}]$&$=$&$[(0,{e_3}),\wideparen{(0,{e_3})}]$&$=$&$
L_{e_3}L_{e_3}\ =\ T_{e_2}\ =\ \varepsilon_6.$\\
\end{longtable}
\end{enumerate}
\end{proof}

\begin{remark} 
$\mathcal{F}({\rm A}_2)$ is non-perfect and 
$\mathcal{F}({\rm A}_2)= S \ltimes R,$ where

\begin{itemize}
\item
$S = \Big\langle\varepsilon_1, \varepsilon_5, \varepsilon_8\Big\rangle \cong \mathfrak{sl}_2;$

\item
$R = \Big\langle\varepsilon_2, \ 
\varepsilon_3, \ 
\varepsilon_4, \ 
\varepsilon_6, \  
\varepsilon_7, \ 
\varepsilon_9, \ 
\varepsilon_{10}, \ 
\varepsilon_{11}\Big\rangle$ is a nilpotent ideal nilindex equal to three.\end{itemize}

\end{remark}

  \subsection{$\rm{AK}$ construction for ${\rm A}_3$}

\begin{theorem}
$\mathcal{F}({\rm A}_3)$ is an $11$-dimensional Lie algebra  with the given product by

\begin{longtable}{lcrlcrlcr}
$[\varepsilon_1,\varepsilon_3]$&$=$&$-2 \varepsilon_4,$&$
[\varepsilon_1,\varepsilon_5]$&$=$&$-\varepsilon_1, $&$ 
[\varepsilon_1,\varepsilon_6]$&$=$&$-\varepsilon_2,$ \\ 

$[\varepsilon_1,\varepsilon_7]$&$=$&$-3\varepsilon_3,$&$ 
[\varepsilon_1,\varepsilon_8]$&$=$&$\varepsilon_5,$&$ 
[\varepsilon_1,\varepsilon_9]$&$=$&$\varepsilon_6,$\\ 
$[\varepsilon_1,\varepsilon_{10}]$&$=$&$-\varepsilon_7,$&$
[\varepsilon_1,\varepsilon_{11}]$&$=$&$-\varepsilon_{10},$&$
[\varepsilon_2,\varepsilon_5]$&$=$&$-\varepsilon_2,$\\
$[\varepsilon_2,\varepsilon_6]$&$=$&$-\varepsilon_2,$&$
[\varepsilon_2,\varepsilon_8]$&$=$&$\varepsilon_6, $&$ 
[\varepsilon_2,\varepsilon_9]$&$=$&$\varepsilon_6, $\\ 
$[\varepsilon_3,\varepsilon_5]$&$=$&$ -\varepsilon_3,$&$
[\varepsilon_3,\varepsilon_8]$&$=$&$\varepsilon_7,$&$
[\varepsilon_4,\varepsilon_5]$&$=$&$-2\varepsilon_4,$\\
$[\varepsilon_4,\varepsilon_8]$&$=$&$\varepsilon_3 $&$
[\varepsilon_5,\varepsilon_8]$&$=$&$-\varepsilon_8,$&$
[\varepsilon_5,\varepsilon_9]$&$=$&$-\varepsilon_9,$\\
$[\varepsilon_5,\varepsilon_{10}]$&$=$&$-\varepsilon_{10},$&$
[\varepsilon_5,\varepsilon_{11}]$&$=$&$-2\varepsilon_{11},$&$
[\varepsilon_6,\varepsilon_8]$&$=$&$-\varepsilon_9,$\\
$[\varepsilon_6,\varepsilon_9]$&$=$&$-\varepsilon_9,$&$
[\varepsilon_7,\varepsilon_8]$&$=$&$3\varepsilon_{10},$&$
[\varepsilon_8,\varepsilon_{10}]$&$=$&$-2\varepsilon_{11}.$\\
\end{longtable}

\end{theorem}
\begin{proof} In the accordance with the $\rm{AK}$-algorithm presented in subsection \ref{AKalg}, we define  
$$\mathcal{F}_0=\left\langle
T_{e_1}:=\begin{pmatrix} 1 & 0 & 0 \\ 0 & 1& 0 \\ 0 & 0 & 1 \end{pmatrix}, \ 
T_{e_2}:=\begin{pmatrix}  0& 0 & 0 \\ 1& 1 & 0 \\ 0 & 0 & 0 \end{pmatrix},\ 
T_{e_3}:=\begin{pmatrix} 0 & 0 &0\\ 0 & 0 & 0 \\ 3& 0 & 0 \end{pmatrix}  \right\rangle.$$

 It is easy to see that  $[\mathcal F_0, \mathcal F_{0}]=0.$

We give  the list of elements  
$T_{e_i}^\varepsilon,$ 
$T_{e_i}^\delta,$  
$T_{e_i}^{\varepsilon \delta},$
and $V_{e_i,e_j}$   below:

\begin{longtable}{|lcl|lclcl|}
\hline
   $T_{e_1}(e_{1})$&$=$&$e_1$&   $T_{e_1}^\varepsilon$&$=$&$T_{e_1}-T_{T_{e_1}(e_1)+\overline{T_{e_1}(e_1)}}$&$=$ &$-T_{e_1}$\\

$\overline{T_{e_1}({e_1})}$&$=$&$e_1$&$T_{e_1}^\delta$&$=$&$T_{e_1}+R_{\overline{T_{e_1}({e_1})}}$&$=$&$2T_{e_1}$\\

$R_{e_1}$&$=$&$T_{e_1}$&$T_{e_1}^{\varepsilon\delta}$&$=$&$T_{e_1}^\varepsilon+R_{\overline{T_{e_1}^\varepsilon({e_1})}}$&$=$&$-2T_{e_1}$\\
\hline 

 \hline
    
   $T_{e_2}({e_1})$&$=$&${e_2} $&$ T_{e_2}^\varepsilon$&$=$&$T_{e_2}-T_{T_{e_2}({e_1})+\overline{T_{e_2}({e_1})}}$&$=$&$-T_{e_2}$\\

$\overline{T_{e_2}({e_1})}$&$=$&${e_2}$&   
$T_{e_2}^\delta $&$=$&$T_{e_2}+R_{\overline{T_{e_2}({e_1})}}$&$=$&$2T_{e_2}$\\
   
$R_{e_2}$&$=$&$T_{e_2}$&$T_{e_2}^{\varepsilon\delta}$&$=$&$T_{e_2}^\varepsilon+R_{\overline{T_{e_2}^\varepsilon({e_1})}}$&$=$&$-2T_{e_2}$\\

   \hline
    
\hline
$T_{e_3}({e_1})$&$=$&$3{e_3}$&$    T_{e_3}^\varepsilon$&$=$&$T_{e_3}-T_{T_{e_3}({e_1})+\overline{T_{e_3}({e_1})}}$&$=$&$T_{e_3}$\\

$\overline{T_{e_3}({e_1}})$&$=$&$-3{e_3}$&$ T_{e_3}^\delta$&$=$&$T_{e_3}+R_{\overline{T_{e_3}({e_1})}}$&$=$&$0$\\
   
$R_{e_3}$&$=$&$\frac{1}{3}T_{e_3}$&$T_{e_3}^{\varepsilon\delta}$&$=$&$T_{e_3}^\varepsilon+R_{\overline{T_{e_3}^\varepsilon({e_1})}}$&$=$&$0$\\
\hline\end{longtable}

\begin{longtable}{lcllcllcl}

$V_{{e_1},{e_1}}$&$=$&$T_{e_1},$&$ V_{{e_1},{e_2}}$&$=$&$T_{e_2},$&$ V_{{e_1},{e_3}}$&$=$&$-T_{e_3},$\\ 
$V_{{e_2},{e_1}}$&$=$&$T_{e_2},$&$ V_{{e_2},{e_2}}$&$=$&$T_{e_2}, $&$ V_{{e_2},{e_3}}$&$=$&$0,$\\
$V_{{e_3},{e_1}}$&$=$&$T_{e_3},$&$ V_{{e_3},{e_2}}$&$=$&$0,$&$ V_{{e_3},{e_3}}$&$=$&$0.$
\end{longtable}

We are now in a position to determine the multiplication table of $\mathcal F({\rm A}_3).$

\begin{enumerate}[(I)]
    \item First, we define  $[\mathcal F_0, \mathcal F_{1}+\mathcal F_{2}],$ i.e., $[E,(x,s)]=(E(x), E^\delta(s)).$

\begin{longtable}{lclclcl}
$[\varepsilon_5,\varepsilon_1]$&$=$&$  [T_{e_1}, ({e_1},0)]$&$=$&$ 
\big(T_{e_1}({e_1}), T_{e_1}^\delta(0) \big)$&$=$&$({e_1},\ 0)\  =\ \varepsilon_1;$\\

$[\varepsilon_5,\varepsilon_2]$&$=$&$[T_{e_1}, ({e_2},0)]$&$=$&$
(T_{e_1}({e_2}), T_{e_1}^\delta(0))$&$=$&$
({e_2},\ 0)\ =\ \varepsilon_2;$\\

$[\varepsilon_5,\varepsilon_3]$&$=$&$[T_{e_1}, ({e_3},0)]$&$=$&$
(T_{e_1}({e_3}), T_{e_1}^\delta (0))$&$=$&$
({e_3},\ 0)\ =\  \varepsilon_3;$\\

$[\varepsilon_5,\varepsilon_4]$&$=$&$[T_{e_1}, (0,{e_3})]$&$=$&$(T_{e_1}(0), T_{e_1}^\delta (e_3))$&$=$&$(0, \ 2e_3) \ =\ 2\varepsilon_4;$\\

$[\varepsilon_6,\varepsilon_1]$&$=$&$[T_{e_2}, ({e_1},0)]$&$=$&$
(T_{e_2}({e_1}), T_{e_2}^\delta (0))$&$=$&$(e_2,\ 0)\ =\ \varepsilon_2;$\\

$[\varepsilon_6,\varepsilon_2]$&$=$&$[T_{e_2}, ({e_2},0)]$&$=$&$
(T_{e_2}({e_2}), T_{e_2}^\delta (0))$&$=$&$(e_2,\ 0) \ =\ \varepsilon_2;$\\

$[\varepsilon_6,\varepsilon_3]$&$=$&$[T_{e_2}, ({e_3},0)]$&$=$&$
(T_{e_2}({e_3}), T_{e_2}^\delta (0))$&$=$&$(0,\ 0)\ =\ 0;$\\
$[\varepsilon_6,\varepsilon_4]$&$=$&$[T_{e_2}, (0,{e_3})]$&$=$&$
(T_{e_2}(0),T_{e_2}^\delta ({e_3}))$&$=$&$(0,\ 0)\ =\ 0;$\\

$[\varepsilon_7,\varepsilon_1]$&$=$&$[T_{e_3}, ({e_1},0)]$&$=$&$
(T_{e_3}({e_1}), T_{e_3}^\delta (0))$&$=$&$(3{e_3},\ 0) \ =\ 3\varepsilon_3;$\\

$[\varepsilon_7,\varepsilon_2]$&$=$&$[T_{e_3}, ({e_2},0)]$&$=$&$
(T_{e_3}({e_2}), T_{e_3}^\delta(0))$&$=$&$(0,\ 0)\ =\ 0;$\\

$[\varepsilon_7,\varepsilon_3]$&$=$&$[T_{e_3}, ({e_3},0)]$&$=$&$
(T_{e_3}({e_3}), T_{e_3}^\delta (0))$&$=$&$(0,\ 0)\ =\ 0;$\\

$[\varepsilon_7,\varepsilon_4]$&$=$&$[T_{e_3}, (0,{e_3})]$&$=$&$
(T_{e_3}(0), T_{e_3}^\delta ({e_3}))$&$=$&$(0,\ 0)\ =\ 0.$
\end{longtable}

  \item Second, we define   $[\mathcal F_0, \mathcal F_{-2}+\mathcal F_{-1}],$ i.e., 
  $[{E}, \wideparen {(x,s)} ] =
\wideparen{({E}^\varepsilon(x), {E}^{\varepsilon\delta}(s))}.$

\begin{longtable}{lclclcl}
$[\varepsilon_5,\varepsilon_8]$&$=$&$[T_{e_1}, \wideparen{({e_1},0)}] $&$=$&$ \wideparen{(T_{e_1}^\varepsilon({e_1}),\  0)}$&$=$&$
\wideparen{(-T_{e_1}({e_1}), \ 0)}\ =\ 
\wideparen{(-{e_1}, \ 0)}
\ =\ -\varepsilon_8;$ \\

$[\varepsilon_5,\varepsilon_9]$&$=$&$[T_{e_1}, \wideparen{({e_2},0)}] $&$=$&$ \wideparen{(T_{e_1}^\varepsilon({e_2}),\  0)}$&$=$&$
\wideparen{(-T_{e_1}({e_2}), \ 0)}\ =\ 
\wideparen{(-{e_2}, \ 0)}
\ =\ -\varepsilon_9;$ \\

$[\varepsilon_5,\varepsilon_{10}]$&$=$&$[T_{e_1}, \wideparen{({e_3},0)}] $&$=$&$ \wideparen{(T_{e_1}^\varepsilon({e_3}),\  0)}$&$=$&$
\wideparen{(-T_{e_1}({e_3}), \ 0)}\ =\ \wideparen{(-{e_3}, \ 0)}
\ =\ -\varepsilon_{10};$ \\

$[\varepsilon_5,\varepsilon_{11}]$&$=$&$[T_{e_1}, \wideparen{(0,{e_3})} ]$&$=$&$
\wideparen{(0, T_{e_1}^{\varepsilon\delta} ({e_3}))} $&$=$&$\wideparen{(0,\ -2{e_3})} \ =\ -2\varepsilon_{11};$\\

$[\varepsilon_6,\varepsilon_8]$&$=$&$[T_{e_2}, \wideparen{({e_1},0)}] $&$=$&$ \wideparen{(T_{e_2}^\varepsilon({e_1}),\  0)}$&$=$&$\wideparen{(-{e_2}, \ 0)}
\ =\ -\varepsilon_9;$\\

$[\varepsilon_6,\varepsilon_9]$&$=$&$[T_{e_2}, \wideparen{({e_2},0)}] $&$=$&$ \wideparen{(T_{e_2}^\varepsilon({e_2}),\  0)}$&$=$&$\wideparen{(-e_2, \ 0)}
\ =\ -\varepsilon_9;$ \\

$[\varepsilon_6,\varepsilon_{10}]$&$=$&$[T_{e_2},\wideparen{({e_3},0)}] $&$=$&$ \wideparen{(T_{e_2}^\varepsilon({e_3}),\  0)}$&$=$&$\wideparen{(0, \ 0)}
\ =\ 0;$\\

$[\varepsilon_6,\varepsilon_{11}]$&$=$&$[T_{e_2}, \wideparen{(0,{e_3})} ]$&$=$&$\wideparen{(0, T_{e_2}^{\varepsilon\delta} ({e_3}))} $&$=$&$\wideparen{(0,\ 0)} \ =\ 0;$\\
$[\varepsilon_7,\varepsilon_8]$&$=$&$[T_{e_3}, \wideparen{({e_1},0)}] $&$=$&$ \wideparen{(T_{e_3}^\varepsilon({e_1}),\  0)}$&$=$&$\wideparen{(3{e_3}, \ 0)}
\ =\ 3\varepsilon_{10};$\\

$[\varepsilon_7,\varepsilon_9]$&$=$&$[T_{e_3}, \wideparen{({e_2},0)}] $&$=$&$ \wideparen{(T_{e_3}^\varepsilon({e_2}),\  0)}$&$=$&$\wideparen{(0, \ 0)}
\ =\ 0;$\\

$[\varepsilon_7,\varepsilon_{10}]$&$=$&$[T_{e_3}, \wideparen{({e_3},0)}] $&$=$&$ \wideparen{(T_{e_3}^\varepsilon({e_3}),\  0)}$&$=$&$\wideparen{(0, \ 0)}
\ =\ 0;$\\

$[\varepsilon_7,\varepsilon_{11}]$&$=$&$[T_{e_3}, \wideparen{(0,{e_3})} ]$&$=$&$
\wideparen{(0, T_{e_3}^{\varepsilon\delta} ({e_3}))} $&$=$&$
\wideparen{(0,\ 0)}\ =\ 0.$
\end{longtable}

\item Third, we define   $[ \mathcal F_{1}+\mathcal F_{2}, \mathcal F_{1}+\mathcal F_{2}],$ i.e., 
$[(x,s), (y,r)] = (0, x\overline{y} - y\overline{x}).$

\begin{longtable}{lclcl}

$[\varepsilon_1,\varepsilon_2]$&$=$&$ [({e_1},0), ({e_2},0)] $&$=$&$(0,\ 0) \ =\ 0;$\\
$[\varepsilon_1,\varepsilon_3]$&$=$&$[({e_1},0), ({e_3},0)] $&$=$&$(0,\ -2e_3) \ =\ -2\varepsilon_4;$\\

$[\varepsilon_2,\varepsilon_3]$&$=$&$[({e_2},0), ({e_3},0)] $&$=$&$(0,\ 0) \ =\ 0.$\\

\end{longtable}

\item Fourth, we define   $[ \mathcal F_{-2}+\mathcal F_{-1}, \mathcal F_{-2}+\mathcal F_{-1}],$ i.e., 
$[\wideparen{(x,s)}, \wideparen{(y,r)} ] = \wideparen{(0, x\overline{y} - y\overline{x})}.$

\begin{longtable}{lclcl}

$[\varepsilon_8,\varepsilon_9]$&$=$&$[\wideparen{({e_1},0)}, \wideparen{({e_2},0)}] $&$=$&$\wideparen{(0,\ 0)} \ =\ 0;$\\
$[\varepsilon_8,\varepsilon_{10}]$&$=$&$[\wideparen{({e_1},0)}, \wideparen{({e_3},0)}] $&$=$&$\wideparen{(0,\ -2e_3)} \ =\ -2\varepsilon_{11};$\\

$[\varepsilon_9,\varepsilon_{10}]$&$=$&$[\wideparen{({e_2},0)}, \wideparen{({e_3},0)}] $&$=$&$\wideparen{(0,\ 0)}\ =\ 0.$\\

\end{longtable}

\item Last, we define   $[ \mathcal F_{1}+\mathcal F_{2}, \mathcal F_{-2}+\mathcal F_{-1}],$ i.e., 
    $[(x,s), \wideparen{(y,r)}] = -\wideparen{(r x, 0)}  + V_{x,y} + L_s L_r+ (s y, 0).$

\begin{longtable}{lclcl}
    
$[\varepsilon_1,\varepsilon_8]$&$=$&$  [({e_1},0),\wideparen{({e_1},0)}]$&$=$&$
V_{{e_1},{e_1}}\ =\ T_{e_1}\ =\ \varepsilon_5;$\\
$[\varepsilon_1,\varepsilon_9]$&$=$&$  [({e_1},0),\wideparen{({e_2},0)}]$&$=$&$
V_{{e_1},{e_2}}\ =\ T_{e_2}\ =\ \varepsilon_6;$\\
$[\varepsilon_1,\varepsilon_{10}]$&$=$&$[({e_1},0),\wideparen{({e_3},0)}]$&$=$&$
V_{{e_1},{e_3}}\ =\ -T_{e_3}\ =\ -\varepsilon_7;$\\
$[\varepsilon_1,\varepsilon_{11}]$&$=$&$[({e_1},0),\wideparen{(0,{e_3})}]$&$=$&$
-\wideparen{({e_3},\ 0)} \ =\ -\varepsilon_{10};$\\
$[\varepsilon_2,\varepsilon_8]$&$=$&$ [({e_2},0),\wideparen{({e_1},0)}]$&$=$&$V_{{e_2},{e_1}}\ =\ T_{e_2}\ =\ \varepsilon_6;$\\

$[\varepsilon_2,\varepsilon_9]$&$=$&$  [({e_2},0),\wideparen{({e_2},0)}]$&$=$&$
V_{{e_2},{e_2}}\ =\ T_{e_2}\ =\ \varepsilon_6;$\\
$[\varepsilon_2,\varepsilon_{10}]$&$=$&$[({e_2},0),\wideparen{({e_3},0)}]$&$=$&$
V_{{e_2},{e_3}}\ =\ 0;$\\
$[\varepsilon_2,\varepsilon_{11}]$&$=$&$[({e_2},0),\wideparen{(0,{e_3})}]$&$=$&$
-\wideparen{({e_3}{e_2},\ 0)}\ =\ -\wideparen{(0,\ 0)} \ =\ 0;$\\
$[\varepsilon_3,\varepsilon_8]$&$=$&$[({e_3},0),\wideparen{({e_1},0)}]$&$=$&$V_{{e_3},{e_1}}\ =\ T_{e_3}\ =\ \varepsilon_7;$\\
$[\varepsilon_3,\varepsilon_9]$&$=$&$[({e_3},0),\wideparen{({e_2},0)}]$&$=$&$V_{{e_3},{e_2}}\ =\ 0;$\\
$[\varepsilon_3,\varepsilon_{10}]$&$=$&$[({e_3},0),\wideparen{({e_3},0)}]$&$=$&$V_{{e_3},{e_3}}\ =\ 0;$\\
$[\varepsilon_3,\varepsilon_{11}]$&$=$&$[({e_3},0),\wideparen{(0,{e_3})}]$&$=$&$-\wideparen{({e_3}{e_3},\ 0)}\ =\ -\wideparen{(0,\ 0)} \ =\ 0;$\\
$[\varepsilon_4,\varepsilon_8]$&$=$&$[(0,{e_3}),\wideparen{({e_1},0)}]$&$=$&$({e_3},\ 0)  \ =\ \varepsilon_3;$\\
 $[\varepsilon_4,\varepsilon_9]$&$=$&$[(0,{e_3}),\wideparen{({e_2},0)}]$&$=$&$
({e_3}{e_2}, 0)\ =\ (0,\ 0)\ =\ 0;$\\
$[\varepsilon_4,\varepsilon_{10}]$&$=$&$[(0,{e_3}),\wideparen{({e_3},0)}]$&$=$&$
({e_3}{e_3},\ 0)\ =\ (0,\ 0) \  =\ 0;$\\
$[\varepsilon_4,\varepsilon_{11}]$&$=$&$[(0,{e_3}),\wideparen{(0,{e_3})}]$&$=$&$L_{e_3}L_{e_3}\ =\ 0.$\\
\end{longtable}

\end{enumerate}
\end{proof}

\begin{remark}   $\mathcal{F}({\rm A}_3)$ is perfect
and  $\mathcal{F}({\rm A}_3)=S \ltimes R,$ 
where

\begin{itemize}
\item $S = \Big\langle\varepsilon_1,\ \varepsilon_2, \ 
\varepsilon_5, \ \varepsilon_6, \ 
\varepsilon_8, \ \varepsilon_9\Big\rangle
\cong  \mathfrak{sl}_2\oplus \mathfrak{sl}_2;$

    \item $R = \Big\langle 
    \varepsilon_3, \ 
    \varepsilon_4, \ 
    \varepsilon_7, \ 
    \varepsilon_{10}, \ 
    \varepsilon_{11}\Big\rangle$
is  a abelian radical.

\end{itemize}

\end{remark}

\subsection{$\rm{AK}$ construction for ${\rm A}_4$}

\begin{theorem}
$\mathcal{F}({\rm A}_4)$ is an $11$-dimensional Lie algebra  with    the multiplication defined by

\begin{longtable}{lcrlcrlcr}
$[\varepsilon_1,\varepsilon_3]$&$=$&$-2 \varepsilon_4,$&$
[\varepsilon_1,\varepsilon_5]$&$=$&$-\varepsilon_1, $&$ 
[\varepsilon_1,\varepsilon_6]$&$=$&$-\varepsilon_2,$ \\ 

$[\varepsilon_1,\varepsilon_7]$&$=$&$-3\varepsilon_3,$&$ 
[\varepsilon_1,\varepsilon_8]$&$=$&$\varepsilon_5,$&$ 
[\varepsilon_1,\varepsilon_9]$&$=$&$\varepsilon_6,$\\ 
$[\varepsilon_1,\varepsilon_{10}]$&$=$&$-\varepsilon_7,$&$
[\varepsilon_1,\varepsilon_{11}]$&$=$&$-\varepsilon_{10},$&$
[\varepsilon_2,\varepsilon_5]$&$=$&$-\varepsilon_2,$\\

$[\varepsilon_2,\varepsilon_6]$&$=$&$-\varepsilon_2, $&$
[\varepsilon_2,\varepsilon_8]$&$=$&$\varepsilon_6, $&$ 
[\varepsilon_2,\varepsilon_9]$&$=$&$\varepsilon_6,$\\
$[\varepsilon_3,\varepsilon_5]$&$=$&$ -\varepsilon_3,$&$
[\varepsilon_3,\varepsilon_7]$&$=$&$3\varepsilon_1-3\varepsilon_2,$&$ 
[\varepsilon_3,\varepsilon_8]$&$=$&$\varepsilon_7,$\\
$[\varepsilon_3,\varepsilon_{10}]$&$=$&$\varepsilon_5-\varepsilon_6,$&$ 
[\varepsilon_3,\varepsilon_{11}]$&$=$&$\varepsilon_8-\varepsilon_9,$&$
[\varepsilon_4,\varepsilon_5]$&$=$&$-2\varepsilon_4,$\\
$[\varepsilon_4,\varepsilon_8]$&$=$&$\varepsilon_3,$&$  [\varepsilon_4,\varepsilon_{10}]$&$=$&$-\varepsilon_1+\varepsilon_2,$&$
[\varepsilon_4,\varepsilon_{11}]$&$=$&$-\varepsilon_5+\varepsilon_6,$\\
$[\varepsilon_5,\varepsilon_8]$&$=$&$-\varepsilon_8,$&$ 
[\varepsilon_5,\varepsilon_9]$&$=$&$-\varepsilon_9,$&$
[\varepsilon_5,\varepsilon_{10}]$&$=$&$-\varepsilon_{10},$\\ $[\varepsilon_5,\varepsilon_{11}]$&$=$&$-2\varepsilon_{11},$&$ 
[\varepsilon_6,\varepsilon_8]$&$=$&$-\varepsilon_9,$&$
[\varepsilon_6,\varepsilon_9]$&$=$&$-\varepsilon_9,$\\
$[\varepsilon_7,\varepsilon_8]$&$=$&$3\varepsilon_{10},$&$
[\varepsilon_7,\varepsilon_{10}]$&$=$&$-3\varepsilon_8+3\varepsilon_9,$&$
[\varepsilon_8,\varepsilon_{10}]$&$=$&$-2\varepsilon_{11}.$\\
\end{longtable}

\end{theorem}
\begin{proof}
By means of the $\rm{AK}$-algorithm introduced in subsection \ref{AKalg}, we define 
$$\mathcal{F}_0=\left\langle
T_{e_1}:=\begin{pmatrix} 1 & 0 & 0 \\ 0 & 1& 0 \\ 0 & 0 & 1 \end{pmatrix}, \ 
T_{e_2}:=\begin{pmatrix}  0& 0 & 0 \\ 1& 1 & 0 \\ 0 & 0 & 0 \end{pmatrix},\ 
T_{e_3}:=\begin{pmatrix} 0 & 0 &-3\\ 0 & 0 & 3 \\ 3& 0 & 0 \end{pmatrix}  \right\rangle.$$

It is easy to see that   $[\mathcal F_0, \mathcal F_{0}]=0.$ 

It should be noted that $L_{e_3}L_{e_3}=T_{e_2}-T_{e_1};$ and  we list the elements 
$T_{e_i}^\varepsilon,$ 
$T_{e_i}^\delta,$ 
$T_{e_i}^{\varepsilon \delta},$ and 
$V_{e_i,e_j}$   below:

\begin{longtable}{|lcl|lclcl|}
\hline
   $T_{e_1}(e_{1})$&$=$&$e_1$&   $T_{e_1}^\varepsilon$&$=$&$T_{e_1}-T_{T_{e_1}(e_1)+\overline{T_{e_1}(e_1)}}$&$=$ &$-T_{e_1}$\\

$\overline{T_{e_1}({e_1})}$&$=$&$e_1$&$T_{e_1}^\delta$&$=$&$T_{e_1}+R_{\overline{T_{e_1}({e_1})}}$&$=$&$2T_{e_1}$\\

$R_{e_1}$&$=$&$T_{e_1}$&$T_{e_1}^{\varepsilon\delta}$&$=$&$T_{e_1}^\varepsilon+R_{\overline{T_{e_1}^\varepsilon({e_1})}}$&$=$&$-2T_{e_1}$\\
\hline 

 \hline
    
   $T_{e_2}({e_1})$&$=$&${e_2} $&$ T_{e_2}^\varepsilon$&$=$&$T_{e_2}-T_{T_{e_2}({e_1})+\overline{T_{e_2}({e_1})}}$&$=$&$-T_{e_2}$\\

$\overline{T_{e_2}({e_1})}$&$=$&${e_2}$&   
$T_{e_2}^\delta $&$=$&$T_{e_2}+R_{\overline{T_{e_2}({e_1})}}$&$=$&$2T_{e_2}$\\
   
$R_{e_2}$&$=$&$T_{e_2}$&$T_{e_2}^{\varepsilon\delta}$&$=$&$T_{e_2}^\varepsilon+R_{\overline{T_{e_2}^\varepsilon({e_1})}}$&$=$&$-2T_{e_2}$\\

   \hline
    
\hline
$T_{e_3}({e_1})$&$=$&$3{e_3}$&$    T_{e_3}^\varepsilon$&$=$&$T_{e_3}-T_{T_{e_3}({e_1})+\overline{T_{e_3}({e_1})}}$&$=$&$T_{e_3}$\\

$\overline{T_{e_3}({e_1}})$&$=$&$-3{e_3}$&$ T_{e_3}^\delta$&$=$&$T_{e_3}+R_{\overline{T_{e_3}({e_1})}}$&$=$&$0$\\
   
$R_{e_3}$&$=$&$\frac{1}{3}T_{e_3}$&$T_{e_3}^{\varepsilon\delta}$&$=$&$T_{e_3}^\varepsilon+R_{\overline{T_{e_3}^\varepsilon({e_1})}}$&$=$&$0$\\
\hline\end{longtable}

\begin{longtable}{lcllcllcl}

$V_{{e_1},{e_1}}$&$=$&$T_{e_1},$&$ V_{{e_1},{e_2}}$&$=$&$T_{e_2},$&$ V_{{e_1},{e_3}}$&$=$&$-T_{e_3},$\\ 
$V_{{e_2},{e_1}}$&$=$&$T_{e_2},$&$ V_{{e_2},{e_2}}$&$=$&$T_{e_2}, $&$ V_{{e_2},{e_3}}$&$=$&$0,$\\
$V_{{e_3},{e_1}}$&$=$&$T_{e_3},$&$ V_{{e_3},{e_2}}$&$=$&$0,$&$ V_{{e_3},{e_3}}$&$=$&$T_{e_1}-T_{e_2}.$
\end{longtable}

We are prepared to determine the multiplication table of $\mathcal F({\rm A}_4).$

\begin{enumerate}[(I)]
    \item First, we define   $[\mathcal F_0, \mathcal F_{1}+\mathcal F_{2}],$ i.e., $[E,(x,s)]=(E(x), E^\delta(s)).$

\begin{longtable}{lclclcl}
$[\varepsilon_5,\varepsilon_1]$&$=$&$  [T_{e_1}, ({e_1},0)]$&$=$&$ 
\big(T_{e_1}({e_1}), T_{e_1}^\delta(0) \big)$&$=$&$({e_1},\  0) \ =\ \varepsilon_1;$\\

$[\varepsilon_5,\varepsilon_2]$&$=$&$[T_{e_1}, ({e_2},0)]$&$=$&$
(T_{e_1}({e_2}), T_{e_1}^\delta(0))$&$=$&$({e_2},\ 0)\ =\ \varepsilon_2;$\\

$[\varepsilon_5,\varepsilon_3]$&$=$&$[T_{e_1}, ({e_3},0)]$&$=$&$
(T_{e_1}({e_3}), T_{e_1}^\delta (0))$&$=$&$({e_3},\ 0)\ =\  \varepsilon_3;$\\

$[\varepsilon_5,\varepsilon_4]$&$=$&$[T_{e_1}, (0,{e_3})]$&$=$&$
(T_{e_1}(0), T_{e_1}^\delta (e_3))$&$=$&$(0,\ 2e_3) \ =\ 2\varepsilon_4;$\\

$[\varepsilon_6,\varepsilon_1]$&$=$&$[T_{e_2}, ({e_1},0)]$&$=$&$(T_{e_2}({e_1}), T_{e_2}^\delta (0))$&$=$&$(e_2,\ 0) \ =\ \varepsilon_2;$\\

$[\varepsilon_6,\varepsilon_2]$&$=$&$[T_{e_2}, ({e_2},0)]$&$=$&$
(T_{e_2}({e_2}), T_{e_2}^\delta (0))$&$=$&$(e_2,\ 0)\  =\ \varepsilon_2;$\\

$[\varepsilon_6,\varepsilon_3]$&$=$&$[T_{e_2}, ({e_3},0)]$&$=$&$
(T_{e_2}({e_3}), T_{e_2}^\delta (0))$&$=$&$(0,\ 0) \ =\ 0;$\\
$[\varepsilon_6,\varepsilon_4]$&$=$&$[T_{e_2}, (0,{e_3})]$&$=$&$
(T_{e_2}(0),T_{e_2}^\delta ({e_3}))$&$=$&$(0,\ 0) \ =\ 0;$\\

$[\varepsilon_7,\varepsilon_1]$&$=$&$[T_{e_3}, ({e_1},0)]$&$=$&$
(T_{e_3}({e_1}), T_{e_3}^\delta (0))$&$=$&$(3{e_3},\ 0) \ =\ 3\varepsilon_3;$\\

$[\varepsilon_7,\varepsilon_2]$&$=$&$[T_{e_3}, ({e_2},0)]$&$=$&$
(T_{e_3}({e_2}), T_{e_3}^\delta(0))$&$=$&$(0,\ 0)\ =\ 0;$\\

$[\varepsilon_7,\varepsilon_3]$&$=$&$[T_{e_3}, ({e_3},0)]$&$=$&$
(T_{e_3}({e_3}), T_{e_3}^\delta (0))$&$=$&$
(-3e_1+3e_2, \ 0) \ =\ -3\varepsilon_1+3\varepsilon_2;$\\

$[\varepsilon_7,\varepsilon_4]$&$=$&$[T_{e_3}, (0,{e_3})]$&$=$&$
(T_{e_3}(0), T_{e_3}^\delta ({e_3}))$&$=$&$(0, \ 0)\ = \ 0.$
\end{longtable}

  \item Second, we define   $[\mathcal F_0, \mathcal F_{-2}+\mathcal F_{-1}],$ i.e., 
  $[{E}, \wideparen {(x,s)} ] =
\wideparen{({E}^\varepsilon(x), {E}^{\varepsilon\delta}(s))}.$

\begin{longtable}{lclclcl}
$[\varepsilon_5,\varepsilon_8]$&$=$&$[T_{e_1}, \wideparen{({e_1},0)}] $&$=$&$ \wideparen{(T_{e_1}^\varepsilon({e_1}),\  0)}$&$=$&$
\wideparen{(-T_{e_1}({e_1}), \ 0)}\ =\ \wideparen{(-{e_1}, \ 0)}
\ =\ -\varepsilon_8;$ \\

$[\varepsilon_5,\varepsilon_9]$&$=$&$[T_{e_1}, \wideparen{({e_2},0)}] $&$=$&$ \wideparen{(T_{e_1}^\varepsilon({e_2}),\  0)}$&$=$&$
\wideparen{(-T_{e_1}({e_2}), \ 0)}\  
=\ \wideparen{(-{e_2}, \ 0)}
\ =\ -\varepsilon_9;$ \\

$[\varepsilon_5,\varepsilon_{10}]$&$=$&$[T_{e_1}, \wideparen{({e_3},0)}] $&$=$&$ \wideparen{(T_{e_1}^\varepsilon({e_3}),\  0)}$&$=$&$
\wideparen{(-T_{e_1}({e_3}), \ 0)}\ =\ \wideparen{(-{e_3}, \ 0)}
\ =\ -\varepsilon_{10};$ \\

$[\varepsilon_5,\varepsilon_{11}]$&$=$&$[T_{e_1}, \wideparen{(0,{e_3})} ]$&$=$&$ \wideparen{(0, T_{e_1}^{\varepsilon\delta} ({e_3}))} $&$=$&$ 
\wideparen{(0,-2{e_3})} \ =\ -2\varepsilon_{11};$\\

$[\varepsilon_6,\varepsilon_8]$&$=$&$[T_{e_2}, \wideparen{({e_1},0)}] $&$=$&$  \wideparen{(T_{e_2}^\varepsilon({e_1}),\  0)}$&$=$&$ 
\wideparen{(-{e_2}, \ 0)}
\ =\ -\varepsilon_9;$\\

$[\varepsilon_6,\varepsilon_9]$&$=$&$[T_{e_2}, \wideparen{({e_2},0)}] $&$=$&$  \wideparen{(T_{e_2}^\varepsilon({e_2}),\  0)}$&$=$&$ 
\wideparen{(-e_2, \ 0)}
\ =\ -\varepsilon_9;$ \\

$[\varepsilon_6,\varepsilon_{10}]$&$=$&$[T_{e_2},\wideparen{({e_3},0)}] $&$=$&$  \wideparen{(T_{e_2}^\varepsilon({e_3}),\  0)}$&$=$&$ 
\wideparen{(0, \ 0)}
\ =\ 0;$\\

$[\varepsilon_6,\varepsilon_{11}]$&$=$&$[T_{e_2}, \wideparen{(0,{e_3})} ]$&$=$&$ \wideparen{(0, T_{e_2}^{\varepsilon\delta} ({e_3}))} $&$=$&$ 
\wideparen{(0,\ 0)}\  =\ 0;$\\
$[\varepsilon_7,\varepsilon_8]$&$=$&$[T_{e_3}, \wideparen{({e_1},0)}] $&$=$&$  \wideparen{(T_{e_3}^\varepsilon({e_1}),\  0)}$&$=$&$ 
\wideparen{(3{e_3}, \ 0)}
\ =\ 3\varepsilon_{10};$\\

$[\varepsilon_7,\varepsilon_9]$&$=$&$[T_{e_3}, \wideparen{({e_2},0)}] $&$=$&$  \wideparen{(T_{e_3}^\varepsilon({e_2}),\  0)}$&$=$&$ 
\wideparen{(0, \ 0)}
\ =\ 0;$\\

$[\varepsilon_7,\varepsilon_{10}]$&$=$&$[T_{e_3}, \wideparen{({e_3},0)}] $&$=$&$  \wideparen{(T_{e_3}^\varepsilon({e_3}),\  0)} $&$=$&$ 
\wideparen{(-3e_1+3{e_2}, \ 0)}
\ =\ -3\varepsilon_8+3\varepsilon_9;$\\

$[\varepsilon_7,\varepsilon_{11}]$&$=$&$[T_{e_3}, \wideparen{(0,{e_3})} ]$&$=$&$ \wideparen{(0, T_{e_3}^{\varepsilon\delta} ({e_3}))} $&$=$&$ \wideparen{(0,\ 0)}\ =\ 0.$
\end{longtable}

\item Third, we define  $[ \mathcal F_{1}+\mathcal F_{2}, \mathcal F_{1}+\mathcal F_{2}],$ i.e., 
$[(x,s), (y,r)] = (0, x\overline{y} - y\overline{x}).$

\begin{longtable}{lclcl}

$[\varepsilon_1,\varepsilon_2]$&$=$&$ [({e_1},0), ({e_2},0)]$&$=$&$ (0,\ 0) \ =\ 0;$\\
$[\varepsilon_1,\varepsilon_3]$&$=$&$[({e_1},0), ({e_3},0)] $&$=$&$ (0,\ -2e_3) \ =\ -2\varepsilon_4;$\\

$[\varepsilon_2,\varepsilon_3]$&$=$&$[({e_2},0), ({e_3},0)] $&$=$&$ (0,\ 0) \ =\ 0.$\\

\end{longtable}

\item Fourth, we define   $[ \mathcal F_{-2}+\mathcal F_{-1}, \mathcal F_{-2}+\mathcal F_{-1}],$ i.e., 
$[\wideparen{(x,s)}, \wideparen{(y,r)} ] = \wideparen{(0, x\overline{y} - y\overline{x})}.$

\begin{longtable}{lclcl}

$[\varepsilon_8,\varepsilon_9]$&$=$&$[\wideparen{({e_1},0)}, \wideparen{({e_2},0)}] $&$=$&$ \wideparen{(0,\ 0)} \ =\ 0;$\\
$[\varepsilon_8,\varepsilon_{10}]$&$=$&$[\wideparen{({e_1},0)}, \wideparen{({e_3},0)}] $&$=$&$ \wideparen{(0,\ -2e_3)} \ =\ -2\varepsilon_{11};$\\

$[\varepsilon_9,\varepsilon_{10}]$&$=$&$[\wideparen{({e_2},0)}, \wideparen{({e_3},0)}] $&$=$&$ \wideparen{(0,\ 0)} \ =\ 0.$\\

\end{longtable}

\item End, we define   $[ \mathcal F_{1}+\mathcal F_{2}, \mathcal F_{-2}+\mathcal F_{-1}],$ i.e., 
    $[(x,s), \wideparen{(y,r)}] = -\wideparen{(r x, 0)}  + V_{x,y} + L_s L_r+ (s y, 0).$

\begin{longtable}{lclcl}
    
$[\varepsilon_1,\varepsilon_8]$&$=$&$  [({e_1},0),\wideparen{({e_1},0)}]$&$=$&$ 
V_{{e_1},{e_1}}\ =\ T_{e_1}\ =\ \varepsilon_5;$\\
$[\varepsilon_1,\varepsilon_9]$&$=$&$  [({e_1},0),\wideparen{({e_2},0)}]$&$=$&$ 
V_{{e_1},{e_2}}\ =\ T_{e_2}\ =\ \varepsilon_6;$\\
$[\varepsilon_1,\varepsilon_{10}]$&$=$&$[({e_1},0),\wideparen{({e_3},0)}]$&$=$&$ 
V_{{e_1},{e_3}}\ =\ -T_{e_3}\ =\ -\varepsilon_7;$\\
$[\varepsilon_1,\varepsilon_{11}]$&$=$&$[({e_1},0),\wideparen{(0,{e_3})}]$&$=$&$ 
-\wideparen{({e_3},\ 0)} \ =\ -\varepsilon_{10};$\\

$[\varepsilon_2,\varepsilon_8]$&$=$&$ [({e_2},0),\wideparen{({e_1},0)}]$&$=$&$ 
V_{{e_2},{e_1}}\ =\ T_{e_2}\ =\ \varepsilon_6;$\\
$[\varepsilon_2,\varepsilon_9]$&$=$&$  [({e_2},0),\wideparen{({e_2},0)}]$&$=$&$ 
V_{{e_2},{e_2}}\ =\ T_{e_2}\ =\ \varepsilon_6;$\\
$[\varepsilon_2,\varepsilon_{10}]$&$=$&$[({e_2},0),\wideparen{({e_3},0)}]$&$=$&$ 
V_{{e_2},{e_3}}\ =\ 0;$\\
$[\varepsilon_2,\varepsilon_{11}]$&$=$&$[({e_2},0),\wideparen{(0,{e_3})}]$&$=$&$ 
-\wideparen{({e_3}{e_2},0)}\ =\ -\wideparen{(0,\ 0)} \ =\ 0;$\\

$[\varepsilon_3,\varepsilon_8]$&$=$&$[({e_3},0),\wideparen{({e_1},0)}]$&$=$&$ 
V_{{e_3},{e_1}}\ =\ T_{e_3}\ =\ \varepsilon_7;$\\
$[\varepsilon_3,\varepsilon_9]$&$=$&$[({e_3},0),\wideparen{({e_2},0)}]$&$=$&$ 
V_{{e_3},{e_2}}\ =\ 0;$\\
$[\varepsilon_3,\varepsilon_{10}]$&$=$&$[({e_3},0),\wideparen{({e_3},0)}]$&$=$&$ 
V_{{e_3},{e_3}}\ =\ T_{e_1}-T_{e_2}\ =\ \varepsilon_5-\varepsilon_6;$\\

$[\varepsilon_3,\varepsilon_{11}]$&$=$&$[({e_3},0),\wideparen{(0,{e_3})}]$&$=$&$ 
-\wideparen{({e_3}{e_3},0)}\ =\ -\wideparen{({-e_1+e_2},\ 0)} \ =\ \varepsilon_8-\varepsilon_9;$\\

$[\varepsilon_4,\varepsilon_8]$&$=$&$[(0,{e_3}),\wideparen{({e_1},0)}]$&$=$&$ 
({e_3},\ 0) \ =\ \varepsilon_3;$\\
 
$[\varepsilon_4,\varepsilon_9]$&$=$&$[(0,{e_3}),\wideparen{({e_2},0)}]$&$=$&$ 
({e_3}{e_2},\ 0)\ =\ (0,\ 0)\  =\ 0;$\\
$[\varepsilon_4,\varepsilon_{10}]$&$=$&$[(0,{e_3}),\wideparen{({e_3},0)}]$&$=$&$ 
({e_3}{e_3}, \ 0)\ =\ (-e_1+{e_2},\ 0) \  =\ -\varepsilon_1+\varepsilon_2;$\\

$[\varepsilon_4,\varepsilon_{11}]$&$=$&$[(0,{e_3}),\wideparen{(0,{e_3})}]$&$=$&$ 
L_{e_3}L_{e_3}\ =\ -T_{e_1}+T_{e_2}\ =\ -\varepsilon_5+\varepsilon_6.$\\
\end{longtable}

\end{enumerate}
\end{proof}

\begin{remark}
Consider the following change of basis in $\mathcal{F}({\rm A}_4)$:
\begin{longtable}{lcllcllcllcl}
$\xi_1 $&$=$ &$\varepsilon_1 - \varepsilon_2,$ &
$\xi_2 $&$=$ &$ \varepsilon_2,$ &
$\xi_3 $&$=$ &$ \varepsilon_3,$ &
$\xi_4 $&$=$ &$ \varepsilon_4,$ \\
$\xi_5 $&$=$ &$ \varepsilon_5 - \varepsilon_6,$ &
$\xi_6 $&$=$ &$ \varepsilon_6,$ &
$\xi_7 $&$=$ &$ \varepsilon_7,$ &
$\xi_8 $&$=$ &$ \varepsilon_8 - \varepsilon_9,$ \\
$\xi_9 $&$=$ &$ \varepsilon_9,$ &
$\xi_{10} $&$=$ &$ \varepsilon_{10},$ &
$\xi_{11} $&$=$ &$ \varepsilon_{11}.$
\end{longtable}\noindent
In this new basis, the nonzero multiplications become:
\begin{longtable}{lcrlcrlcr}
$[\xi_1,\xi_3]$&$=$&$-2\xi_4,$&$
[\xi_1,\xi_5]$&$=$&$-\xi_1,$&$
[\xi_1,\xi_7]$&$=$&$-3\xi_3,$\\
$[\xi_1,\xi_8]$&$=$&$\xi_5,$&$
[\xi_1,\xi_{10}]$&$=$&$-\xi_7,$&$
[\xi_1,\xi_{11}]$&$=$&$-\xi_{10},$\\
$[\xi_2,\xi_6]$&$=$&$-\xi_2,$&$
[\xi_2,\xi_9]$&$=$&$\xi_6,$&$
[\xi_3,\xi_5]$&$=$&$-\xi_3,$\\
$[\xi_3,\xi_7]$&$=$&$3\xi_1,$&$
[\xi_3,\xi_8]$&$=$&$\xi_7,$&$
[\xi_3,\xi_{10}]$&$=$&$\xi_5,$\\
$[\xi_3,\xi_{11}]$&$=$&$\xi_8,$&$
[\xi_4,\xi_5]$&$=$&$-2\xi_4,$&$
[\xi_4,\xi_8]$&$=$&$\xi_3,$\\
$[\xi_4,\xi_{10}]$&$=$&$-\xi_1,$&$
[\xi_4,\xi_{11}]$&$=$&$-\xi_5,$&$
[\xi_5,\xi_8]$&$=$&$-\xi_8,$\\
$[\xi_5,\xi_{10}]$&$=$&$-\xi_{10},$&$
[\xi_5,\xi_{11}]$&$=$&$-2\xi_{11},$&$
[\xi_6,\xi_9]$&$=$&$-\xi_9,$\\
$[\xi_7,\xi_8]$&$=$&$3\xi_{10},$&$
[\xi_7,\xi_{10}]$&$=$&$-3\xi_8,$&$
[\xi_8,\xi_{10}]$&$=$&$-2\xi_{11}.$
\end{longtable}\noindent 
This gives  $\mathcal{F}({\rm A}_4) \cong \mathfrak{sl}_2 \oplus\mathfrak{sl}_3,$ where
\begin{itemize}
\item $\mathfrak{sl}_2= \big\langle\xi_2,\  \xi_6, \ \xi_9\big\rangle;$
\item $\mathfrak{sl}_3=\big\langle\xi_1, \ \xi_3,\  \xi_4,\  \xi_5,\  \xi_7,\  \xi_8,\  \xi_{10}, \ \xi_{11}\big\rangle$.
\end{itemize}
\end{remark}

\subsection{$\rm{AK}$ construction for ${\rm A}_5$}

\begin{theorem}
$\mathcal{F}({\rm A}_5)$ is an $11$-dimensional Lie algebra  with the given product by

\begin{longtable}{lcrlcrlcr}
$[\varepsilon_1,\varepsilon_3]$&$=$&$-2\varepsilon_4,$&
$[\varepsilon_1,\varepsilon_5]$&$=$&$-\varepsilon_1,$&
$[\varepsilon_1,\varepsilon_6]$&$=$&$-\varepsilon_2,$\\

$[\varepsilon_1,\varepsilon_7]$&$=$&$-3\varepsilon_3,$&
$[\varepsilon_1,\varepsilon_8]$&$=$&$\varepsilon_5,$&
$[\varepsilon_1,\varepsilon_9]$&$=$&$\varepsilon_6,$\\

$[\varepsilon_1,\varepsilon_{10}]$&$=$&$-\varepsilon_7,$&
$[\varepsilon_1,\varepsilon_{11}]$&$=$&$-\varepsilon_{10},$&
$[\varepsilon_2,\varepsilon_5]$&$=$&$-\varepsilon_2,$\\

$[\varepsilon_2,\varepsilon_7]$&$=$&$-\varepsilon_2,$&
$[\varepsilon_2,\varepsilon_8]$&$=$&$\varepsilon_6,$&
$[\varepsilon_2,\varepsilon_{10}]$&$=$&$-\varepsilon_6,$\\

$[\varepsilon_2,\varepsilon_{11}]$&$=$&$\varepsilon_9,$&
$[\varepsilon_3,\varepsilon_5]$&$=$&$-\varepsilon_3,$&
$[\varepsilon_3,\varepsilon_6]$&$=$&$-\varepsilon_2,$\\

$[\varepsilon_3,\varepsilon_7]$&$=$&$-3\varepsilon_1,$&
$[\varepsilon_3,\varepsilon_8]$&$=$&$\varepsilon_7,$&
$[\varepsilon_3,\varepsilon_9]$&$=$&$-\varepsilon_6,$\\

$[\varepsilon_3,\varepsilon_{10}]$&$=$&$-\varepsilon_5,$&
$[\varepsilon_3,\varepsilon_{11}]$&$=$&$-\varepsilon_8,$&
$[\varepsilon_4,\varepsilon_5]$&$=$&$-2\varepsilon_4,$\\

$[\varepsilon_4,\varepsilon_8]$&$=$&$\varepsilon_3,$&
$[\varepsilon_4,\varepsilon_9]$&$=$&$-\varepsilon_2,$&
$[\varepsilon_4,\varepsilon_{10}]$&$=$&$\varepsilon_1,$\\

$[\varepsilon_4,\varepsilon_{11}]$&$=$&$\varepsilon_5,$&
$[\varepsilon_5,\varepsilon_8]$&$=$&$-\varepsilon_8,$&
$[\varepsilon_5,\varepsilon_9]$&$=$&$-\varepsilon_9,$\\

$[\varepsilon_5,\varepsilon_{10}]$&$=$&$-\varepsilon_{10},$&
$[\varepsilon_5,\varepsilon_{11}]$&$=$&$-2\varepsilon_{11},$&
$[\varepsilon_6,\varepsilon_7]$&$=$&$2\varepsilon_6,$\\

$[\varepsilon_6,\varepsilon_8]$&$=$&$-\varepsilon_9,$&
$[\varepsilon_6,\varepsilon_{10}]$&$=$&$-\varepsilon_9,$&
$[\varepsilon_7,\varepsilon_8]$&$=$&$3\varepsilon_{10},$\\

$[\varepsilon_7,\varepsilon_9]$&$=$&$\varepsilon_9,$&
$[\varepsilon_7,\varepsilon_{10}]$&$=$&$3\varepsilon_8,$&
$[\varepsilon_8,\varepsilon_{10}]$&$=$&$-2\varepsilon_{11}.$
\end{longtable}
\end{theorem}
\begin{proof}
Following the $\rm{AK}$-algorithm given in subsection \ref{AKalg}, we define  
$$\mathcal{F}_0=\left\langle
T_{e_1}:=\begin{pmatrix} 1 & 0 & 0 \\ 0 & 1& 0 \\ 0 & 0 & 1 \end{pmatrix}, \ 
T_{e_2}:=\begin{pmatrix}  0& 0 & 0 \\ 1& 0 & 1 \\ 0 & 0 & 0 \end{pmatrix},\ 
T_{e_3}:=\begin{pmatrix} 0 & 0 &3\\ 0 &1 & 0 \\ 3& 0 & 0 \end{pmatrix}  \right\rangle.$$

 Let us note that $L_{e_3}L_{e_3}=T_{e_1}$ and   list the elements  
$T_{e_i}^\varepsilon,$ 
$T_{e_i}^\delta,$  
$T_{e_i}^{\varepsilon \delta},$
and 
$V_{e_i,e_j}$   below:

\begin{longtable}{|lcl|lclcl|}
\hline
   $T_{e_1}(e_{1})$&$=$&$e_1$&   $T_{e_1}^\varepsilon$&$=$&$T_{e_1}-T_{T_{e_1}(e_1)+\overline{T_{e_1}(e_1)}}$&$=$ &$-T_{e_1}$\\

$\overline{T_{e_1}({e_1})}$&$=$&$e_1$&$T_{e_1}^\delta$&$=$&$T_{e_1}+R_{\overline{T_{e_1}({e_1})}}$&$=$&$2T_{e_1}$\\

$R_{e_1}$&$=$&$T_{e_1}$&$T_{e_1}^{\varepsilon\delta}$&$=$&$T_{e_1}^\varepsilon+R_{\overline{T_{e_1}^\varepsilon({e_1})}}$&$=$&$-2T_{e_1}$\\
\hline 

 \hline
    
   $T_{e_2}({e_1})$&$=$&${e_2} $&$ T_{e_2}^\varepsilon$&$=$&$T_{e_2}-T_{T_{e_2}({e_1})+\overline{T_{e_2}({e_1})}}$&$=$&$-T_{e_2}$\\

$\overline{T_{e_2}({e_1})}$&$=$&${e_2}$&   
$T_{e_2}^\delta $&$=$&$T_{e_2}+R_{\overline{T_{e_2}({e_1})}}$&$=$&$T_{e_2}+R_{e_2}$\\
   
& & & $T_{e_2}^{\varepsilon\delta}$ &$=$&$T_{e_2}^\varepsilon+R_{\overline{T_{e_2}^\varepsilon({e_1})}}$&$=$&$-T_{e_2}-R_{e_2}$\\

   \hline
    
\hline
$T_{e_3}({e_1})$&$=$&$3{e_3}$&$    T_{e_3}^\varepsilon$&$=$&$T_{e_3}-T_{T_{e_3}({e_1})+\overline{T_{e_3}({e_1})}}$&$=$&$T_{e_3}$\\

$\overline{T_{e_3}({e_1}})$&$=$&$-3{e_3}$&$ T_{e_3}^\delta$&$=$&$T_{e_3}+R_{\overline{T_{e_3}({e_1})}}$&$=$&$T_{e_3}-3R_{e_3}$\\
   
&&&$T_{e_3}^{\varepsilon\delta}$&$=$&$T_{e_3}^\varepsilon+R_{\overline{T_{e_3}^\varepsilon({e_1})}}$&$=$&$T_{e_3}-3R_{e_3}$\\
\hline\end{longtable}

\begin{longtable}{lcllcllcl}

$V_{{e_1},{e_1}}$&$=$&$T_{e_1},$&$ V_{{e_1},{e_2}}$&$=$&$T_{e_2},$&$ V_{{e_1},{e_3}}$&$=$&$-T_{e_3},$\\ 
$V_{{e_2},{e_1}}$&$=$&$T_{e_2},$&$ V_{{e_2},{e_2}}$&$=$&$0, $&$ V_{{e_2},{e_3}}$&$=$&$-T_{e_2},$\\
$V_{{e_3},{e_1}}$&$=$&$T_{e_3},$&$ V_{{e_3},{e_2}}$&$=$&$-T_{e_2},$&$ V_{{e_3},{e_3}}$&$=$&$-T_{e_1}.$
\end{longtable}

Now we are ready to determine the multiplication table of $\mathcal F({\rm A}_5).$

\begin{enumerate}[(I)]
    \item First, we define  $[\mathcal F_0, \mathcal F_{1}+\mathcal F_{2}],$ i.e., $[E,(x,s)]=(E(x), E^\delta(s)).$

\begin{longtable}{lclclcl}
$[\varepsilon_5,\varepsilon_1]$&$=$&$  [T_{e_1}, ({e_1},0)]$&$=$&$ \big(T_{e_1}({e_1}), T_{e_1}^\delta(0) \big)$&$=$&$({e_1}, \ 0) \  =\ \varepsilon_1;$\\

$[\varepsilon_5,\varepsilon_2]$&$=$&$[T_{e_1}, ({e_2},0)]$&$=$&$
(T_{e_1}({e_2}), T_{e_1}^\delta(0))$&$=$&$
({e_2},\ 0)\ =\ \varepsilon_2;$\\

$[\varepsilon_5,\varepsilon_3]$&$=$&$[T_{e_1}, ({e_3},0)]$&$=$&$
(T_{e_1}({e_3}), T_{e_1}^\delta (0))$&$=$&$({e_3},\ 0)\ =\  \varepsilon_3;$\\
$[\varepsilon_5,\varepsilon_4]$&$=$&$[T_{e_1}, (0,{e_3})]$&$=$&$
(T_{e_1}(0), T_{e_1}^\delta (e_3))$&$=$&$(0, \ 2e_3)\  =\ 2\varepsilon_4;$\\

$[\varepsilon_6,\varepsilon_1]$&$=$&$[T_{e_2}, ({e_1},0)]$&$=$&$
(T_{e_2}({e_1}), T_{e_2}^\delta (0))$&$=$&$(e_2,\ 0)\  =\ \varepsilon_2;$\\

$[\varepsilon_6,\varepsilon_2]$&$=$&$[T_{e_2}, ({e_2},0)]$&$=$&$
(T_{e_2}({e_2}), T_{e_2}^\delta (0))$&$=$&$(0,\ 0) \ =\ 0;$\\

$[\varepsilon_6,\varepsilon_3]$&$=$&$[T_{e_2}, ({e_3},0)]$&$=$&$
(T_{e_2}({e_3}), T_{e_2}^\delta (0))$&$=$&$(e_2,\ 0) \ =\ \varepsilon_2;$\\
$[\varepsilon_6,\varepsilon_4]$&$=$&$[T_{e_2}, (0,{e_3})]$&$=$&$
(T_{e_2}(0),T_{e_2}^\delta ({e_3}))$&$=$&$(0,\ 0)\ =\ 0;$\\

$[\varepsilon_7,\varepsilon_1]$&$=$&$[T_{e_3}, ({e_1},0)]$&$=$&$
(T_{e_3}({e_1}), T_{e_3}^\delta (0))$&$=$&$
(3{e_3},\ 0) \ =\ 3\varepsilon_3;$\\

$[\varepsilon_7,\varepsilon_2]$&$=$&$[T_{e_3}, ({e_2},0)]$&$=$&$
(T_{e_3}({e_2}), T_{e_3}^\delta(0))$&$=$&$(e_2,\ 0)\ =\ \varepsilon_2;$\\

$[\varepsilon_7,\varepsilon_3]$&$=$&$[T_{e_3}, ({e_3},0)]$&$=$&$
(T_{e_3}({e_3}), T_{e_3}^\delta (0))$&$=$&$(3{e_1}, \ 0) \ =\ 3\varepsilon_1;$\\

$[\varepsilon_7,\varepsilon_4]$&$=$&$[T_{e_3}, (0,{e_3})]$&$=$&$
(T_{e_3}(0), T_{e_3}^\delta ({e_3}))$&$=$&$ (0,\  0)\ =\ 0.$
\end{longtable}

  \item Second, 
  we define   $[\mathcal F_0, \mathcal F_{-2}+\mathcal F_{-1}],$ i.e., 
  $[{E}, \wideparen {(x,s)} ] =
\wideparen{({E}^\varepsilon(x), {E}^{\varepsilon\delta}(s))}.$

\begin{longtable}{lclclcl}
$[\varepsilon_5,\varepsilon_8]$&$=$&$[T_{e_1}, \wideparen{({e_1},0)}] $&$=$&$ \wideparen{(T_{e_1}^\varepsilon({e_1}),\  0)}$&$=$&$
\wideparen{(-T_{e_1}({e_1}), \ 0)}\ =\ \wideparen{(-{e_1}, \ 0)}
\ =\ -\varepsilon_8;$ \\

$[\varepsilon_5,\varepsilon_9]$&$=$&$[T_{e_1}, \wideparen{({e_2},0)}] $&$=$&$ \wideparen{(T_{e_1}^\varepsilon({e_2}),\  0)}$&$=$&$
\wideparen{(-T_{e_1}({e_2}), \ 0)}\ =\ \wideparen{(-{e_2}, \ 0)}
\ =\ -\varepsilon_9;$ \\

$[\varepsilon_5,\varepsilon_{10}]$&$=$&$[T_{e_1}, \wideparen{({e_3},0)}] $&$=$&$ \wideparen{(T_{e_1}^\varepsilon({e_3}),\  0)}$&$=$&$
\wideparen{(-T_{e_1}({e_3}), \ 0)}\ =\ \wideparen{(-{e_3}, \ 0)}
\ =\ -\varepsilon_{10};$ \\

$[\varepsilon_5,\varepsilon_{11}]$&$=$&$[T_{e_1}, \wideparen{(0,{e_3})} ]$&$=$&$
\wideparen{(0, T_{e_1}^{\varepsilon\delta} ({e_3}))} $&$=$&$
\wideparen{(0,\ -2{e_3})} \ =\ -2\varepsilon_{11};$\\

$[\varepsilon_6,\varepsilon_8]$&$=$&$[T_{e_2}, \wideparen{({e_1},0)}] $&$=$&$ \wideparen{(T_{e_2}^\varepsilon({e_1}),\  0)}$&$=$&$
\wideparen{(-{e_2}, \ 0)}
\ =\ -\varepsilon_9;$\\

$[\varepsilon_6,\varepsilon_9]$&$=$&$[T_{e_2}, \wideparen{({e_2},0)}] $&$=$&$ \wideparen{(T_{e_2}^\varepsilon({e_2}),\  0)}$&$=$&$
\wideparen{(0, \ 0)}\ =\ 0;$ \\

$[\varepsilon_6,\varepsilon_{10}]$&$=$&$[T_{e_2},\wideparen{({e_3},0)}] $&$=$&$ \wideparen{(T_{e_2}^\varepsilon({e_3}),\  0)}$&$=$&$
\wideparen{(-e_2, \ 0)}
\ =\ -\varepsilon_9;$\\

$[\varepsilon_6,\varepsilon_{11}]$&$=$&$[T_{e_2}, \wideparen{(0,{e_3})} ]$&$=$&$
\wideparen{(0, T_{e_2}^{\varepsilon\delta} ({e_3}))} $&$=$&$\wideparen{(0,\ 0)}\  =\ 0;$\\

$[\varepsilon_7,\varepsilon_8]$&$=$&$[T_{e_3}, \wideparen{({e_1},0)}] $&$=$&$ \wideparen{(T_{e_3}^\varepsilon({e_1}),\  0)}$&$=$&$ 
\wideparen{(3{e_3}, \ 0)}
\ =\ 3\varepsilon_{10};$\\

$[\varepsilon_7,\varepsilon_9]$&$=$&$[T_{e_3}, \wideparen{({e_2},0)}] $&$=$&$ \wideparen{(T_{e_3}^\varepsilon({e_2}),\  0)}$&$=$&$
\wideparen{(e_2, \ 0)}
\ =\ \varepsilon_9;$\\

$[\varepsilon_7,\varepsilon_{10}]$&$=$&$[T_{e_3}, \wideparen{({e_3},0)}] $&$=$&$ \wideparen{(T_{e_3}^\varepsilon({e_3}),\  0)}$&$=$&$
\wideparen{(3{e_1}, \ 0)}
\ =\ 3\varepsilon_8;$\\

$[\varepsilon_7,\varepsilon_{11}]$&$=$&$[T_{e_3}, \wideparen{(0,{e_3})} ]$&$=$&$
\wideparen{(0, T_{e_3}^{\varepsilon\delta} ({e_3}))} $&$=$&$ \wideparen{(0,\ 0)} \ =\ 0.$
\end{longtable}

\item Third, we define  $[ \mathcal F_{1}+\mathcal F_{2}, \mathcal F_{1}+\mathcal F_{2}],$ i.e., 
$[(x,s), (y,r)] = (0, x\overline{y} - y\overline{x}).$

\begin{longtable}{lclcl}

$[\varepsilon_1,\varepsilon_2]$&$=$&$ [({e_1},0), ({e_2},0)] $&$=$&$(0,\ 0) \ =\ 0;$\\
$[\varepsilon_1,\varepsilon_3]$&$=$&$[({e_1},0), ({e_3},0)] $&$=$&$(0,\ -2e_3) \ =\ -2\varepsilon_4;$\\

$[\varepsilon_2,\varepsilon_3]$&$=$&$[({e_2},0), ({e_3},0)] $&$=$&$(0,\ 0) \ =\ 0.$\\

\end{longtable}

\item Fourth, we define   $[ \mathcal F_{-2}+\mathcal F_{-1}, \mathcal F_{-2}+\mathcal F_{-1}],$ i.e., 
$[\wideparen{(x,s)}, \wideparen{(y,r)} ] = \wideparen{(0, x\overline{y} - y\overline{x})}.$

\begin{longtable}{lclcl}

$[\varepsilon_8,\varepsilon_9]$&$=$&$[\wideparen{({e_1},0)}, \wideparen{({e_2},0)}] $&$=$&$
\wideparen{(0,\ 0)} \ =\ 0;$\\
$[\varepsilon_8,\varepsilon_{10}]$&$=$&$[\wideparen{({e_1},0)}, \wideparen{({e_3},0)}] $&$=$&$
\wideparen{(0,\ -2e_3)} \ =\ -2\varepsilon_{11};$\\

$[\varepsilon_9,\varepsilon_{10}]$&$=$&$[\wideparen{({e_2},0)}, \wideparen{({e_3},0)}] $&$=$&$\wideparen{(0,\ 0)} \ =\ 0.$\\

\end{longtable}

\item Fifth, we define   $[ \mathcal F_{1}+\mathcal F_{2}, \mathcal F_{-2}+\mathcal F_{-1}],$ i.e., 
    $[(x,s), \wideparen{(y,r)}] = -\wideparen{(r x, 0)}  + V_{x,y} + L_s L_r+ (s y, 0).$

\begin{longtable}{lclcl}
    
$[\varepsilon_1,\varepsilon_8]$&$=$&$  [({e_1},0),\wideparen{({e_1},0)}]$&$=$&$
V_{{e_1},{e_1}}\ =\ T_{e_1}\ =\ \varepsilon_5;$\\
$[\varepsilon_1,\varepsilon_9]$&$=$&$  [({e_1},0),\wideparen{({e_2},0)}]$&$=$&$
V_{{e_1},{e_2}}\ =\ T_{e_2}\ =\ \varepsilon_6;$\\
$[\varepsilon_1,\varepsilon_{10}]$&$=$&$[({e_1},0),\wideparen{({e_3},0)}]$&$=$&$
V_{{e_1},{e_3}}\ =\ -T_{e_3}\ =\ -\varepsilon_7;$\\
$[\varepsilon_1,\varepsilon_{11}]$&$=$&$[({e_1},0),\wideparen{(0,{e_3})}]$&$=$&$
-\wideparen{({e_3},\ 0)} \ =\ -\varepsilon_{10};$\\

$[\varepsilon_2,\varepsilon_8]$&$=$&$ [({e_2},0),\wideparen{({e_1},0)}]$&$=$&$
V_{{e_2},{e_1}}\ =\ T_{e_2}\ =\ \varepsilon_6;$\\

$[\varepsilon_2,\varepsilon_9]$&$=$&$  [({e_2},0),\wideparen{({e_2},0)}]$&$=$&$
V_{{e_2},{e_2}}\ =\ 0;$\\
$[\varepsilon_2,\varepsilon_{10}]$&$=$&$[({e_2},0),\wideparen{({e_3},0)}]$&$=$&$
V_{{e_2},{e_3}}\ =\ -T_{e_2}\ =\ -\varepsilon_6;$\\

$[\varepsilon_2,\varepsilon_{11}]$&$=$&$[({e_2},0),\wideparen{(0,{e_3})}]$&$=$&$-\wideparen{({e_3}{e_2},\ 0)}\ =\ \wideparen{(e_2,\ 0)} \ =\ \varepsilon_9;$\\

$[\varepsilon_3,\varepsilon_8]$&$=$&$[({e_3},0),\wideparen{({e_1},0)}]$&$=$&$
V_{{e_3},{e_1}}\ =\ T_{e_3}\ =\ \varepsilon_7;$\\
$[\varepsilon_3,\varepsilon_9]$&$=$&$[({e_3},0),\wideparen{({e_2},0)}]$&$=$&$
V_{{e_3},{e_2}}\ =\ -T_{e_2}\ =\ -\varepsilon_6;$\\
$[\varepsilon_3,\varepsilon_{10}]$&$=$&$[({e_3},0),\wideparen{({e_3},0)}]$&$=$&$
V_{{e_3},{e_3}}\ =\ -T_{e_1}\ =\ -\varepsilon_5;$\\

$[\varepsilon_3,\varepsilon_{11}]$&$=$&$[({e_3},0),\wideparen{(0,{e_3})}]$&$=$&$
-\wideparen{({e_3}{e_3},\ 0)}\ =\ -\wideparen{({e_1},0)}\ =\ -\varepsilon_8;$\\

$[\varepsilon_4,\varepsilon_8]$&$=$&$[(0,{e_3}),\wideparen{({e_1},0)}]$&$=$&$
({e_3},\ 0) \ =\ \varepsilon_3;$\\
 
$[\varepsilon_4,\varepsilon_9]$&$=$&$[(0,{e_3}),\wideparen{({e_2},0)}]$&$=$&$
({e_3}{e_2},\ 0)\ =\ (-e_2,\ 0) =-\varepsilon_2;$\\
$[\varepsilon_4,\varepsilon_{10}]$&$=$&$[(0,{e_3}),\wideparen{({e_3},0)}]$&$=$&$
({e_3}{e_3},\ 0)\ =\ ({e_1},\ 0) \ =\ \varepsilon_1;$\\

$[\varepsilon_4,\varepsilon_{11}]$&$=$&$[(0,{e_3}),\wideparen{(0,{e_3})}]$&$=$&$L_{e_3}L_{e_3}\ =\ T_{e_1}\ = \ \varepsilon_5.$\\

\end{longtable}

(VI)\ End, we define   $[ \mathcal F_{0}, \mathcal F_{0}].$
\begin{longtable}{lclcl}
    $[\varepsilon_5,\varepsilon_{6}]$&$=$&$[T_{e_1}, T_{e_2}]$&$=$&$0 .$ \\
    
$[\varepsilon_5,\varepsilon_{7}]$&$=$&$[T_{e_1}, T_{e_3}]$&$=$&$0 .$ \\
$[\varepsilon_6,\varepsilon_{7}]$&$=$&$[T_{e_2}, T_{e_3}]$&$=$&$2T_{e_2}\ =\ 2\varepsilon_6.$
\end{longtable}
\end{enumerate}
\end{proof}

\begin{remark}
$\mathcal{F}({\rm A}_5)$ is perfect and 
$\mathcal{F}({\rm A}_5)= S \ltimes R,$ 
where:
\begin{itemize}

\item $S = \big\langle\varepsilon_1, \ 
\varepsilon_3, \ 
\varepsilon_4, \ 
\varepsilon_5, \ 
\varepsilon_7, \ 
\varepsilon_8, \ 
\varepsilon_{10},\ 
\varepsilon_{11}\big\rangle \cong \mathfrak{sl}_3;$

\item $R=\big\langle\varepsilon_2, \ \varepsilon_6, \ \varepsilon_9\big\rangle$ is the abelian radical$.$

\end{itemize}
\end{remark}

\subsection{$\rm{AK}$ construction for ${\rm S}_1$}
$\mathcal{F}(\rm S_1)$ has dimension $13.$
Let us now fix the following notations :

\begin{longtable}{lclcl}
$\mathcal{F}_2$ & $=$ & $(0,\mathcal{S})$ & $=$ & $\big\langle \varepsilon_4:=(0,e_2), \ \varepsilon_5:=(0,e_3)\big\rangle$;\\

$\mathcal{F}_1$ & $=$ & $(\mathcal{A},0)$ & $=$ & $\big\langle \varepsilon_1:=(e_1,0), \ \varepsilon_2:=(e_2,0), \ \varepsilon_3:=(e_3,0)\big\rangle$;\\

$\mathcal{F}_0$ & $=$ & $\rm{Instr}(\mathcal{A})$ & $=$ & $\big\langle
\varepsilon_6:=T_{e_1}, \
\varepsilon_7:=T_{e_2}, \
\varepsilon_8:=T_{e_3} \big\rangle$;\\

$\mathcal{F}_{-1}$ & $=$ & $\wideparen{(\mathcal{A},0)}$ & $=$ & $\big\langle
\varepsilon_9:=\wideparen{(e_1,0)}, \
\varepsilon_{10}:=\wideparen{(e_2,0)}, \
\varepsilon_{11}:=\wideparen{(e_3,0)} \big\rangle$;\\

$\mathcal{F}_{-2}$ & $=$ & $\wideparen{(0,\mathcal{S})}$ & $=$ & $\big\langle
\varepsilon_{12}:=\wideparen{(0,e_2)}, \
\varepsilon_{13}:=\wideparen{(0,e_3)} \big\rangle$.
\end{longtable}

\begin{theorem}
$\mathcal{F}({\rm S}_1)$ is a  $13$-dimensional Lie algebra  with the given product by

\begin{longtable}{lcrlcrlcr}
$[\varepsilon_1,\varepsilon_2]$ & $=$ & $-2\varepsilon_4,$& $[\varepsilon_1,\varepsilon_3]$ & $=$ & $-2\varepsilon_5,$ & $[\varepsilon_1,\varepsilon_6]$ & $=$ & $-\varepsilon_1,$ \\
$[\varepsilon_1,\varepsilon_7]$ & $=$ & $-3\varepsilon_2,$ &

$[\varepsilon_1,\varepsilon_8]$ & $=$ & $-3\varepsilon_3,$ & $[\varepsilon_1,\varepsilon_9]$ & $=$ & $\varepsilon_6,$ \\
$[\varepsilon_1,\varepsilon_{10}]$ & $=$ & $-\varepsilon_7,$ &

$[\varepsilon_1,\varepsilon_{11}]$ & $=$ & $-\varepsilon_8,$ & $[\varepsilon_1,\varepsilon_{12}]$ & $=$ & $-\varepsilon_{10},$ \\
$[\varepsilon_1,\varepsilon_{13}]$ & $=$ & $-\varepsilon_{11},$ 
&
$[\varepsilon_2,\varepsilon_6]$ & $=$ & $-\varepsilon_2,$ &
$[\varepsilon_2,\varepsilon_9]$ & $=$ & $\varepsilon_7,$ \\

$[\varepsilon_3,\varepsilon_6]$ & $=$ & $-\varepsilon_3,$ &
$[\varepsilon_3,\varepsilon_9]$ & $=$ & $\varepsilon_8,$ &
$[\varepsilon_4,\varepsilon_6]$ & $=$ & $-2\varepsilon_4,$ \\

$[\varepsilon_4,\varepsilon_9]$ & $=$ & $\varepsilon_2,$ &
$[\varepsilon_5,\varepsilon_6]$ & $=$ & $-2\varepsilon_5,$ &
$[\varepsilon_5,\varepsilon_9]$ & $=$ & $\varepsilon_3,$ \\

$[\varepsilon_6,\varepsilon_9]$ & $=$ & $-\varepsilon_9,$ &
$[\varepsilon_6,\varepsilon_{10}]$ & $=$ & $-\varepsilon_{10},$ &
$[\varepsilon_6,\varepsilon_{11}]$ & $=$ & $-\varepsilon_{11},$ \\

$[\varepsilon_6,\varepsilon_{12}]$ & $=$ & $-2\varepsilon_{12},$ &
$[\varepsilon_6,\varepsilon_{13}]$ & $=$ & $-2\varepsilon_{13},$ &
$[\varepsilon_7,\varepsilon_9]$ & $=$ & $3\varepsilon_{10},$ \\

$[\varepsilon_8,\varepsilon_9]$ & $=$ & $3\varepsilon_{11},$ &
$[\varepsilon_9,\varepsilon_{10}]$ & $=$ & $-2\varepsilon_{12},$ &
$[\varepsilon_9,\varepsilon_{11}]$ & $=$ & $-2\varepsilon_{13}.$
\end{longtable}

\end{theorem}

\begin{proof}
Following the $\rm{AK}$-algorithm given in subsection \ref{AKalg}, we define  
$$\mathcal{F}_0=\left\langle
T_{e_1}:=\begin{pmatrix} 1 & 0 & 0 \\ 0 & 1& 0 \\ 0 & 0 & 1 \end{pmatrix}, \ 
T_{e_2}:=\begin{pmatrix}  0& 0 & 0 \\ 3& 0 & 0 \\ 0 & 0 & 0 \end{pmatrix},\ 
T_{e_3}:=\begin{pmatrix} 0 & 0 &0\\ 0 &0 & 0 \\ 3& 0 & 0 \end{pmatrix}  \right\rangle.$$

It is easy to see that  
$[\mathcal{F}_0,\mathcal{F}_0]=0.$

We give the list of elements
$T_{e_i}^\varepsilon,$ 
$T_{e_i}^\delta,$ 
$T_{e_i}^{\varepsilon \delta},$ and 
$V_{e_i,e_j}$   below:

\begin{longtable}{|lcl|lclcl|}
\hline
   $T_{e_1}(e_{1})$&$=$&$e_1$&   $T_{e_1}^\varepsilon$&$=$&$T_{e_1}-T_{T_{e_1}(e_1)+\overline{T_{e_1}(e_1)}}$&$=$ &$-T_{e_1}$\\

$\overline{T_{e_1}({e_1})}$&$=$&$e_1$&$T_{e_1}^\delta$&$=$&$T_{e_1}+R_{\overline{T_{e_1}({e_1})}}$&$=$&$2T_{e_1}$\\

$R_{e_1}$&$=$&$T_{e_1}$&$T_{e_1}^{\varepsilon\delta}$&$=$&$T_{e_1}^\varepsilon+R_{\overline{T_{e_1}^\varepsilon({e_1})}}$&$=$&$-2T_{e_1}$\\
\hline 

 \hline
    
   $T_{e_2}({e_1})$&$=$&$3{e_2} $&$ T_{e_2}^\varepsilon$&$=$&$T_{e_2}-T_{T_{e_2}({e_1})+\overline{T_{e_2}({e_1})}}$&$=$&$T_{e_2}$\\

$\overline{T_{e_2}({e_1})}$&$=$&${-3e_2}$&   
$T_{e_2}^\delta $&$=$&$T_{e_2}+R_{\overline{T_{e_2}({e_1})}}$&$=$&$0$\\
   
$R_{e_2}$&$=$&$\frac{1}{3}T_{e_2}$&$T_{e_2}^{\varepsilon\delta}$&$=$&$T_{e_2}^\varepsilon+R_{\overline{T_{e_2}^\varepsilon({e_1})}}$&$=$&$0$\\

   \hline
    
\hline
$T_{e_3}({e_1})$&$=$&$3{e_3}$&$    T_{e_3}^\varepsilon$&$=$&$T_{e_3}-T_{T_{e_3}({e_1})+\overline{T_{e_3}({e_1})}}$&$=$&$T_{e_3}$\\

$\overline{T_{e_3}({e_1}})$&$=$&$-3{e_3}$&$ T_{e_3}^\delta$&$=$&$T_{e_3}+R_{\overline{T_{e_3}({e_1})}}$&$=$&$0$\\
   
$R_{e_3}$&$=$&$\frac{1}{3}T_{e_3}$&$T_{e_3}^{\varepsilon\delta}$&$=$&$T_{e_3}^\varepsilon+R_{\overline{T_{e_3}^\varepsilon({e_1})}}$&$=$&$0$\\
\hline\end{longtable}

\begin{longtable}{lcllcllcl}

$V_{{e_1},{e_1}}$&$=$&$T_{e_1},$&$ V_{{e_1},{e_2}}$&$=$&$-T_{e_2},$&$ V_{{e_1},{e_3}}$&$=$&$-T_{e_3},$\\ 
$V_{{e_2},{e_1}}$&$=$&$T_{e_2},$&$ V_{{e_2},{e_2}}$&$=$&$0, $&$ V_{{e_2},{e_3}}$&$=$&$0,$\\
$V_{{e_3},{e_1}}$&$=$&$T_{e_3},$&$ V_{{e_3},{e_2}}$&$=$&$0,$&$ V_{{e_3},{e_3}}$&$=$&$0.$
\end{longtable}

Now we are ready to determine the multiplication table of $\mathcal F({\rm S}_1).$

\begin{enumerate}[(I)]
    \item First, we define   $[\mathcal F_0, \mathcal F_{1}+\mathcal F_{2}],$ i.e., $[E,(x,s)]=(E(x), E^\delta(s)).$

\begin{longtable}{lclclcl}
$[\varepsilon_6,\varepsilon_1]$&$=$&$  [T_{e_1}, ({e_1},0)]$&$=$&$ 
\big(T_{e_1}({e_1}), T_{e_1}^\delta(0) \big)$&$=$&$({e_1},\  0)\   =\ \varepsilon_1;$\\

$[\varepsilon_6,\varepsilon_2]$&$=$&$[T_{e_1}, ({e_2},0)]$&$=$&$
(T_{e_1}({e_2}), T_{e_1}^\delta(0))$&$=$&$({e_2},\ 0)\ =\ \varepsilon_2;$\\

$[\varepsilon_6,\varepsilon_3]$&$=$&$[T_{e_1}, ({e_3},0)]$&$=$&$
(T_{e_1}({e_3}), T_{e_1}^\delta (0))$&$=$&$({e_3},\ 0)\ =\  \varepsilon_3;$\\
$[\varepsilon_6,\varepsilon_4]$&$=$&$[T_{e_1}, (0,{e_2})]$&$=$&$
(T_{e_1}(0), T_{e_1}^\delta (e_2))$&$=$&$(0,\ 2e_2)\ =\ 2\varepsilon_4;$\\
$[\varepsilon_6,\varepsilon_5]$&$=$&$[T_{e_1}, (0,{e_3})]$&$=$&$
(T_{e_1}(0), T_{e_1}^\delta (e_3))$&$=$&$(0, \ 2e_3) \ =\ 2\varepsilon_5;$\\

$[\varepsilon_7,\varepsilon_1]$&$=$&$[T_{e_2}, ({e_1},0)]$&$=$&$
(T_{e_2}({e_1}), T_{e_2}^\delta (0))$&$=$&$(3e_2,\ 0)\  =\ 3\varepsilon_2;$\\

$[\varepsilon_7,\varepsilon_2]$&$=$&$[T_{e_2}, ({e_2},0)]$&$=$&$
(T_{e_2}({e_2}),\ T_{e_2}^\delta (0))$&$=$&$(0,\ 0) \ =\ 0;$\\

$[\varepsilon_7,\varepsilon_3]$&$=$&$[T_{e_2}, ({e_3},0)]$&$=$&$
(T_{e_2}({e_3}), T_{e_2}^\delta (0))$&$=$&$(0,\ 0)\ =\ 0;$\\
$[\varepsilon_7,\varepsilon_4]$&$=$&$[T_{e_2}, (0,{e_2})]$&$=$&$
(T_{e_2}(0),T_{e_2}^\delta ({e_2}))$&$=$&$(0,\ 0)\ =\ 0;$\\
$[\varepsilon_7,\varepsilon_5]$&$=$&$[T_{e_2}, (0,{e_3})]$&$=$&$
(T_{e_2}(0),T_{e_2}^\delta ({e_3}))$&$=$&$(0, \ 0)\ =\ 0;$\\
$[\varepsilon_8,\varepsilon_1]$&$=$&$[T_{e_3}, ({e_1},0)]$&$=$&$
(T_{e_3}({e_1}), T_{e_3}^\delta (0))$&$=$&$(3{e_3},\ 0) \ =\ 3\varepsilon_3;$\\

$[\varepsilon_8,\varepsilon_2]$&$=$&$[T_{e_3}, ({e_2},0)]$&$=$&$
(T_{e_3}({e_2}), T_{e_3}^\delta(0))$&$=$&$(0,\ 0)\ =\ 0;$\\

$[\varepsilon_8,\varepsilon_3]$&$=$&$[T_{e_3}, ({e_3},0)]$&$=$&$
(T_{e_3}({e_3}), T_{e_3}^\delta (0))$&$=$&$(0, \  0) \ =\ 0;$\\
$[\varepsilon_8,\varepsilon_4]$&$=$&$[T_{e_3}, (0,{e_2})]$&$=$&$
(T_{e_3}(0), T_{e_3}^\delta ({e_2}))$&$=$&$(0, \ 0)\ =\ 0.$\\
$[\varepsilon_8,\varepsilon_5]$&$=$&$[T_{e_3}, (0,{e_3})]$&$=$&$
(T_{e_3}(0), T_{e_3}^\delta ({e_3}))$&$=$&$(0,\  0)\ =\ 0.$
\end{longtable}

  \item Second, we define  $[\mathcal F_0, \mathcal F_{-2}+\mathcal F_{-1}],$ i.e., 
  $[{E}, \wideparen {(x,s)} ] =
\wideparen{({E}^\varepsilon(x), {E}^{\varepsilon\delta}(s))}.$

\begin{longtable}{lclclcl}
$[\varepsilon_6,\varepsilon_9]$&$=$&$[T_{e_1}, \wideparen{({e_1},0)}]$&$=$&$ \wideparen{(T_{e_1}^\varepsilon({e_1}),\  0)}$&$=$&$\wideparen{(-T_{e_1}({e_1}), \ 0)}\ =\ \wideparen{(-{e_1}, \ 0)}
\ =\ -\varepsilon_9;$ \\

$[\varepsilon_6,\varepsilon_{10}]$&$=$&$[T_{e_1}, \wideparen{({e_2},0)}] $&$=$&$ \wideparen{(T_{e_1}^\varepsilon({e_2}),\  0)}$&$=$&$
\wideparen{(-T_{e_1}({e_2}), \ 0)}\ =\ \wideparen{(-{e_2}, \ 0)}
\ =\ -\varepsilon_{10};$ \\

$[\varepsilon_6,\varepsilon_{11}]$&$=$&$[T_{e_1}, \wideparen{({e_3},0)}] $&$=$&$ \wideparen{(T_{e_1}^\varepsilon({e_3}),\  0)}$&$=$&$
\wideparen{(-T_{e_1}({e_3}), \ 0)}\ =\ 
\wideparen{(-{e_3}, \ 0)}
\  =\ -\varepsilon_{11};$ \\

$[\varepsilon_6,\varepsilon_{12}]$&$=$&$[T_{e_1}, \wideparen{(0,{e_2})} ]$&$=$&$\wideparen{(0, T_{e_1}^{\varepsilon\delta} ({e_2}))} $&$=$&$\wideparen{(0,\ -2{e_2})} \ = \ -2\varepsilon_{12};$\\
$[\varepsilon_6,\varepsilon_{13}]$&$=$&$[T_{e_1}, \wideparen{(0,{e_3})} ]$&$=$&$
\wideparen{(0, T_{e_1}^{\varepsilon\delta} ({e_3}))} $&$=$&$
\wideparen{(0,\ -2{e_3})} \ =\ -2\varepsilon_{13};$\\

$[\varepsilon_7,\varepsilon_9]$&$=$&$[T_{e_2}, \wideparen{({e_1},0)}] $&$=$&$ \wideparen{(T_{e_2}^\varepsilon({e_1}),\  0)}$&$=$&$
\wideparen{(3{e_2}, \ 0)}\ =\ 3\varepsilon_{10};$\\

$[\varepsilon_7,\varepsilon_{10}]$&$=$&$[T_{e_2}, \wideparen{({e_2},0)}]$&$=$&$ \wideparen{(T_{e_2}^\varepsilon({e_2}),\  0)}$&$=$&$
\wideparen{(0, \ 0)} \ =\ 0;$ \\

$[\varepsilon_7,\varepsilon_{11}]$&$=$&$[T_{e_2},\wideparen{({e_3},0)}] $&$=$&$ \wideparen{(T_{e_2}^\varepsilon({e_3}),\  0)}$&$=$&$\wideparen{(0, \ 0)} \ =\ 0;$\\

$[\varepsilon_7,\varepsilon_{12}]$&$=$&$[T_{e_2},\wideparen{({0},e_{2})}] $&$=$&$ \wideparen{(T_{e_2}^\varepsilon({e_2}),\  0)}$&$=$&$\wideparen{(0, \ 0)} \ =\ 0;$\\

$[\varepsilon_7,\varepsilon_{13}]$&$=$&$[T_{e_2}, \wideparen{(0,{e_3})} ]$&$=$&$
\wideparen{(0, T_{e_2}^{\varepsilon\delta} ({e_3}))} $&$=$&$\wideparen{(0,\ 0)}\ =\ 0;$\\
$[\varepsilon_8,\varepsilon_9]$&$=$&$[T_{e_3}, \wideparen{({e_1},0)}] $&$=$&$ \wideparen{(T_{e_3}^\varepsilon({e_1}),\  0)}$&$=$&$
\wideparen{(3{e_3}, \ 0)}\ =\ 3\varepsilon_{11};$\\

$[\varepsilon_8,\varepsilon_{10}]$&$=$&$[T_{e_3}, \wideparen{({e_2},0)}] $&$=$&$ \wideparen{(T_{e_3}^\varepsilon({e_2}),\  0)}$&$=$&$
\wideparen{(0, \ 0)}\ =\ 0;$\\

$[\varepsilon_8,\varepsilon_{11}]$&$=$&$[T_{e_3}, \wideparen{({e_3},0)}] $&$=$&$ \wideparen{(T_{e_3}^\varepsilon({e_3}),\  0)} $&$=$&$ 
\wideparen{(0, \ 0)} \ =\ 0;$\\

$[\varepsilon_8,\varepsilon_{12}]$&$=$&$[T_{e_3}, \wideparen{(0,{e_2})} ]$&$=$&$
\wideparen{(0, T_{e_3}^{\varepsilon\delta} ({e_2}))} $&$=$&$
\wideparen{(0,\ 0)} \ =\ 0;$\\

$[\varepsilon_8,\varepsilon_{13}]$&$=$&$[T_{e_3},\wideparen{(0,{e_3})} ]$&$=$&$
\wideparen{(0, T_{e_3}^{\varepsilon\delta} ({e_3}))}$&$=$&$\wideparen{(0,\ 0)}\ =\ 0.$
\end{longtable}

\item Third, we define   $[ \mathcal F_{1}+\mathcal F_{2}, \mathcal F_{1}+\mathcal F_{2}],$ i.e., 
$[(x,s), (y,r)] = (0, x\overline{y} - y\overline{x}).$

\begin{longtable}{lclcl}

$[\varepsilon_1,\varepsilon_2]$&$=$&$[({e_1},0), ({e_2},0)] $&$=$&$(0,\ -2e_2) \ =\ -2\varepsilon_4;$\\
$[\varepsilon_1,\varepsilon_3]$&$=$&$[({e_1},0), ({e_3},0)] $&$=$&$(0,\ -2e_3) \ =\ -2\varepsilon_5;$\\

$[\varepsilon_2,\varepsilon_3]$&$=$&$[({e_2},0), ({e_3},0)] $&$=$&$(0,\ 0) \ =\ 0.$\\

\end{longtable}

\item Fourth, we define   $[ \mathcal F_{-2}+\mathcal F_{-1}, \mathcal F_{-2}+\mathcal F_{-1}],$ i.e., 
$[\wideparen{(x,s)}, \wideparen{(y,r)} ] = \wideparen{(0, x\overline{y} - y\overline{x})}.$

\begin{longtable}{lclcl}

$[\varepsilon_{9},\varepsilon_{10}]$&$=$&$[\wideparen{({e_1},0)}, \wideparen{({e_2},0)}] $&$=$&$\wideparen{(0,\ -2e_2)}\ =\ -2e_{12};$\\
$[\varepsilon_{9},\varepsilon_{11}]$&$=$&$[\wideparen{({e_1},0)}, \wideparen{({e_3},0)}] $&$=$&$\wideparen{(0,\ -2e_3)} \ =\ -2\varepsilon_{13};$\\

$[\varepsilon_{10},\varepsilon_{11}]$&$=$&$[\wideparen{({e_2},0)}, \wideparen{({e_3},0)}] $&$=$&$\wideparen{(0,\ 0)}\ =\ 0.$\\

\end{longtable}

\item End, we define   $[ \mathcal F_{1}+\mathcal F_{2}, \mathcal F_{-2}+\mathcal F_{-1}],$ i.e., 
    $[(x,s), \wideparen{(y,r)}] = -\wideparen{(r x, 0)}  + V_{x,y} + L_s L_r+ (s y, 0).$

\begin{longtable}{lcl cl}
    
$[\varepsilon_1,\varepsilon_9]$&$=$&$  [({e_1},0),\wideparen{({e_1},0)}]$&$=$&$V_{{e_1},{e_1}}\ = \ T_{e_1}\ =\ \varepsilon_6;$\\
$[\varepsilon_1,\varepsilon_{10}]$&$=$&$  [({e_1},0),\wideparen{({e_2},0)}]$&$=$&$
V_{{e_1},{e_2}} \ =\ -T_{e_2}\ =\ -\varepsilon_7;$\\
$[\varepsilon_1,\varepsilon_{11}]$&$=$&$[({e_1},0),\wideparen{({e_3},0)}]$&$=$&$V_{{e_1},{e_3}}\ =\ -T_{e_3}\ =\ -\varepsilon_8;$\\
$[\varepsilon_1,\varepsilon_{12}]$&$=$&$[({e_1},0),\wideparen{(0,{e_2})}]$&$=$&$-\wideparen{({e_2},0)}\ =\ -\varepsilon_{10};$\\
$[\varepsilon_1,\varepsilon_{13}]$&$=$&$[({e_1},0),\wideparen{(0,{e_3})}]$&$=$&$-\wideparen{({e_3},0)} \ =\ -\varepsilon_{11};$\\

$[\varepsilon_2,\varepsilon_9]$&$=$&$ [({e_2},0),\wideparen{({e_1},0)}]$&$=$&$
V_{{e_2},{e_1}}\ =\ T_{e_2}\ =\ \varepsilon_7;$\\

$[\varepsilon_2,\varepsilon_{12}]$&$=$&$[({e_2},0),\wideparen{(0,{e_3})}]$&$=$&$
-\wideparen{({e_3}{e_2},\ 0)}\ =\ \wideparen{(0,\ 0)}\ =\ 0;$\\

$[\varepsilon_2,\varepsilon_{10}]$&$=$&$  [({e_2},0),\wideparen{({e_2},0)}]$&$=$&$V_{{e_2},{e_2}}\ =\ 0;$\\
$[\varepsilon_2,\varepsilon_{11}]$&$=$&$[({e_2},0),\wideparen{({e_3},0)}]$&$=$&$V_{{e_2},{e_3}}\ =\ 0;$\\
$[\varepsilon_2,\varepsilon_{12}]$&$=$&$[({e_2},0),\wideparen{(0,e_2)}]$&$=$&$0;$\\
$[\varepsilon_2,\varepsilon_{13}]$&$=$&$[({e_2},0),\wideparen{(0,e_3)}]$&$=$&$0;$\\
$[\varepsilon_3,\varepsilon_9]$&$=$&$[({e_3},0),\wideparen{({e_1},0)}]$&$=$&$V_{{e_3},{e_1}}\ =\ T_{e_3}\ =\ \varepsilon_8;$\\
$[\varepsilon_3,\varepsilon_{10}]$&$=$&$[({e_3},0),\wideparen{({e_2},0)}]$&$=$&$V_{{e_3},{e_2}}\ =\ 0;$\\
$[\varepsilon_3,\varepsilon_{11}]$&$=$&$[({e_3},0),\wideparen{({e_3},0)}]$&$=$&$V_{{e_3},{e_3}}\ =\ 0;$\\

$[\varepsilon_3,\varepsilon_{12}]$&$=$&$[({e_3},0),\wideparen{(0,{e_2})}]$&$=$&$-\wideparen{({e_3}{e_2},0)}\ =\ -\wideparen{(0,\ 0)} =0;$\\
$[\varepsilon_3,\varepsilon_{13}]$&$=$&$[({e_3},0),\wideparen{(0,{e_3})}]$&$=$&$
-\wideparen{({e_3}{e_3},\ 0)}\ =\ -\wideparen{(0,\ 0)} =0;$\\

$[\varepsilon_4,\varepsilon_9]$&$=$&$[(0,{e_3}),\wideparen{({e_1},0)}]$&$=$&$({e_2},\ 0)  \ =\  \varepsilon_2;$\\
 
$[\varepsilon_4,\varepsilon_{10}]$&$=$&$[(0,{e_3}),\wideparen{({e_2},0)}]$&$=$&$
({e_3}{e_2},\ 0)\ =\ (0,\ 0)\ =\ 0;$\\
$[\varepsilon_4,\varepsilon_{11}]$&$=$&$[(0,{e_3}),\wideparen{({e_3},0)}]$&$=$&$
({e_3}{e_3},\ 0)\ =\ (0,\ 0) \ =\ 0;$\\

$[\varepsilon_4,\varepsilon_{12}]$&$=$&$[(0,{e_2}),\wideparen{(0,{e_2})}]$&$=$&$L_{e_2}L_{e_2}\ =\ 0;$\\
$[\varepsilon_4,\varepsilon_{13}]$&$=$&$[(0,{e_3}),\wideparen{(0,{e_3})}]$&$=$&$L_{e_2}L_{e_3}\ =\ 0;$\\

$[\varepsilon_5,\varepsilon_9]$&$=$&$[(0,{e_3}),\wideparen{({e_1},0)}]$&$=$&$({e_3},\ 0)  \ =\ \varepsilon_3;$\\
 
$[\varepsilon_5,\varepsilon_{10}]$&$=$&$[(0,{e_3}),\wideparen{({e_2},0)}]$&$=$&$
({e_3}{e_2},\ 0)\ =\ (0,\ 0)\  =\ 0;$\\
$[\varepsilon_5,\varepsilon_{11}]$&$=$&$[(0,{e_3}),\wideparen{({e_3},0)}]$&$=$&$
({e_3}{e_3},\ 0)\ =\ (0,\ 0) \  =\ 0;$\\

$[\varepsilon_5,\varepsilon_{12}]$&$=$&$[(0,{e_3}),\wideparen{(0,{e_2})}]$&$=$&$L_{e_3}L_{e_2}\ =\ 0;$\\
$[\varepsilon_5,\varepsilon_{13}]$&$=$&$[(0,{e_3}),\wideparen{(0,{e_3})}]$&$=$&$L_{e_3}L_{e_3}\ =\ 0.$\\

\end{longtable}

\end{enumerate}
\end{proof}

\begin{remark} 
$\mathcal{F}({\rm S}_1)$ is perfect and  
$\mathcal{F}({\rm S}_1) = S \ltimes R,$ where
\begin{itemize}
\item $S = \Big\langle\varepsilon_1, \ 
\varepsilon_6,\ 
\varepsilon_9\Big\rangle \cong \mathfrak{sl}_2;$
    \item $R = \Big\langle\varepsilon_2,\  
\varepsilon_3, \ 
\varepsilon_4, \ 
\varepsilon_5, \ 
\varepsilon_7, \ 
\varepsilon_8, \ 
\varepsilon_{10}, \ 
\varepsilon_{11}, \ 
\varepsilon_{12},\ 
\varepsilon_{13}\Big\rangle$ is the abelian radical.

\end{itemize}
\end{remark}

\subsection{$\rm{AK}$ construction for ${\rm S}_2$}

$\mathcal{F}(\rm{S}_2)$ has dimension $14.$
Let us now fix the following notations:

\begin{longtable}{lclcl}
$\mathcal{F}_2$ & $=$ & $(0,\mathcal{S})$ & $=$ & $\big\langle \varepsilon_4:=(0,e_2), \ \varepsilon_5:=(0,e_3)\big\rangle$;\\

$\mathcal{F}_1$ & $=$ & $(\mathcal{A},0)$ & $=$ & $\big\langle \varepsilon_1:=(e_1,0), \ \varepsilon_2:=(e_2,0), \ \varepsilon_3:=(e_3,0)\big\rangle$;\\

$\mathcal{F}_0$ & $=$ & $\rm{Instr}(\mathcal{A})$ & $=$ & $\big\langle
\varepsilon_6:=T_{e_1}, \
\varepsilon_7:=T_{e_2}, \
\varepsilon_8:=T_{e_3}, \
\varepsilon_9:=D_{e_2,e_3} \big\rangle$;\\

$\mathcal{F}_{-1}$ & $=$ & $\wideparen{(\mathcal{A},0)}$ & $=$ & $\big\langle
\varepsilon_{10}:=\wideparen{(e_1,0)}, \
\varepsilon_{11}:=\wideparen{(e_2,0)}, \
\varepsilon_{12}:=\wideparen{(e_3,0)} \big\rangle$;\\

$\mathcal{F}_{-2}$ & $=$ & $\wideparen{(0,\mathcal{S})}$ & $=$ & $\big\langle
\varepsilon_{13}:=\wideparen{(0,e_2)}, \
\varepsilon_{14}:=\wideparen{(0,e_3)} \big\rangle$.
\end{longtable}
\begin{theorem}
$\mathcal{F}({\rm S}_2)$ is a  $14$-dimensional Lie algebra  with    the multiplication defined by
 
\begin{longtable}{lcrlcrlcr}

$[\varepsilon_1,\varepsilon_2]$&$=$&$-2\varepsilon_4,$&
$[\varepsilon_1,\varepsilon_3]$&$=$&$-2\varepsilon_5,$&
$[\varepsilon_1,\varepsilon_6]$&$=$&$-\varepsilon_1,$\\

$[\varepsilon_1,\varepsilon_7]$&$=$&$-3\varepsilon_2,$&
$[\varepsilon_1,\varepsilon_8]$&$=$&$-3\varepsilon_3,$&
$[\varepsilon_1,\varepsilon_{10}]$&$=$&$\varepsilon_6,$\\

$[\varepsilon_1,\varepsilon_{11}]$&$=$&$-\varepsilon_7,$&
$[\varepsilon_1,\varepsilon_{12}]$&$=$&$-\varepsilon_8,$&
$[\varepsilon_1,\varepsilon_{13}]$&$=$&$-\varepsilon_{11},$\\

$[\varepsilon_1,\varepsilon_{14}]$&$=$&$-\varepsilon_{12},$&
$[\varepsilon_2,\varepsilon_3]$&$=$&$-2\varepsilon_4,$&
$[\varepsilon_2,\varepsilon_6]$&$=$&$-\varepsilon_2,$\\

$[\varepsilon_2,\varepsilon_8]$&$=$&$-\varepsilon_2,$&
$[\varepsilon_2,\varepsilon_{10}]$&$=$&$\varepsilon_7,$&
$[\varepsilon_2,\varepsilon_{12}]$&$=$&$\frac13\varepsilon_7-\frac83\varepsilon_9,$\\

$[\varepsilon_2,\varepsilon_{14}]$&$=$&$\varepsilon_{11},$&
$[\varepsilon_3,\varepsilon_6]$&$=$&$-\varepsilon_3,$&
$[\varepsilon_3,\varepsilon_7]$&$=$&$\varepsilon_2,$\\

$[\varepsilon_3,\varepsilon_8]$&$=$&$-3\varepsilon_1,$&
$[\varepsilon_3,\varepsilon_9]$&$=$&$-\varepsilon_2,$&
$[\varepsilon_3,\varepsilon_{10}]$&$=$&$\varepsilon_8,$\\

$[\varepsilon_3,\varepsilon_{11}]$&$=$&$-\frac13\varepsilon_7+\frac83\varepsilon_9,$&
$[\varepsilon_3,\varepsilon_{12}]$&$=$&$-\varepsilon_6,$&
$[\varepsilon_3,\varepsilon_{13}]$&$=$&$-\varepsilon_{11},$\\

$[\varepsilon_3,\varepsilon_{14}]$&$=$&$-\varepsilon_{10},$&
$[\varepsilon_4,\varepsilon_6]$&$=$&$-2\varepsilon_4,$&
$[\varepsilon_4,\varepsilon_8]$&$=$&$2\varepsilon_4,$\\

$[\varepsilon_4,\varepsilon_{10}]$&$=$&$\varepsilon_2,$&
$[\varepsilon_4,\varepsilon_{12}]$&$=$&$\varepsilon_2,$&
$[\varepsilon_4,\varepsilon_{14}]$&$=$&$\frac13\varepsilon_7+\frac43\varepsilon_9,$\\

$[\varepsilon_5,\varepsilon_6]$&$=$&$-2\varepsilon_5,$&
$[\varepsilon_5,\varepsilon_7]$&$=$&$-2\varepsilon_4,$&
$[\varepsilon_5,\varepsilon_{9}]$&$=$&$-\varepsilon_4,$\\
$[\varepsilon_5,\varepsilon_{10}]$&$=$&$\varepsilon_3,$&
$[\varepsilon_5,\varepsilon_{11}]$&$=$&$-\varepsilon_2,$&
$[\varepsilon_5,\varepsilon_{12}]$&$=$&$\varepsilon_1,$\\
$[\varepsilon_5,\varepsilon_{13}]$&$=$&$-\frac13\varepsilon_7-\frac43\varepsilon_9,$&
$[\varepsilon_5,\varepsilon_{14}]$&$=$&$\varepsilon_6,$&
$[\varepsilon_6,\varepsilon_{10}]$&$=$&$-\varepsilon_{10},$\\
$[\varepsilon_6,\varepsilon_{11}]$&$=$&$-\varepsilon_{11},$&
$[\varepsilon_6,\varepsilon_{12}]$&$=$&$-\varepsilon_{12},$&
$[\varepsilon_6,\varepsilon_{13}]$&$=$&$-2\varepsilon_{13},$\\
$[\varepsilon_6,\varepsilon_{14}]$&$=$&$-2\varepsilon_{14},$&
$[\varepsilon_7,\varepsilon_8]$&$=$&$-2\varepsilon_7+8\varepsilon_9,$&
$[\varepsilon_7,\varepsilon_{10}]$&$=$&$3\varepsilon_{11},$\\
$[\varepsilon_7,\varepsilon_{12}]$&$=$&$-\varepsilon_{11},$&
$[\varepsilon_7,\varepsilon_{14}]$&$=$&$2\varepsilon_{13},$&
$[\varepsilon_8,\varepsilon_9]$&$=$&$-\varepsilon_7,$\\
$[\varepsilon_8,\varepsilon_{10}]$&$=$&$3\varepsilon_{12},$&
$[\varepsilon_8,\varepsilon_{11}]$&$=$&$\varepsilon_{11},$&
$[\varepsilon_8,\varepsilon_{12}]$&$=$&$3\varepsilon_{10},$\\
$[\varepsilon_8,\varepsilon_{13}]$&$=$&$-2\varepsilon_{13},$&
$[\varepsilon_9,\varepsilon_{12}]$&$=$&$\varepsilon_{11},$&
$[\varepsilon_9,\varepsilon_{14}]$&$=$&$\varepsilon_{13},$\\
$[\varepsilon_{10},\varepsilon_{11}]$&$=$&$-2\varepsilon_{13},$&
$[\varepsilon_{10},\varepsilon_{12}]$&$=$&$-2\varepsilon_{14},$&
$[\varepsilon_{11},\varepsilon_{12}]$&$=$&$-2\varepsilon_{13}.$\\
\end{longtable}
\end{theorem}

\begin{proof}
Following the $\rm{AK}$-algorithm given in subsection \ref{AKalg}, we define  
$$\mathcal{F}_0=\left\langle
T_{e_1}:=\begin{pmatrix} 1 & 0 & 0 \\ 0 & 1& 0 \\ 0 & 0 & 1 \end{pmatrix}, \ 
T_{e_2}:=\begin{pmatrix}  0& 0 & 0 \\3& 0 &-1 \\ 0 & 0 & 0 \end{pmatrix},\ 
T_{e_3}:=\begin{pmatrix} 0 & 0 &3\\ 0 &1 & 0 \\ 3& 0 & 0 \end{pmatrix} ,\
D_{e_2,e_3}:=\begin{pmatrix} 0 & 0 &0\\ 0 &0 & 1 \\ 0& 0 & 0 \end{pmatrix} \right\rangle.$$

$$\left\langle
R_{e_1}:=\begin{pmatrix} 1 & 0 & 0 \\ 0 & 1& 0 \\ 0 & 0 & 1 \end{pmatrix}, \ 
R_{e_2}:=\begin{pmatrix}  0& 0 & 0 \\1& 0 &-1 \\ 0 & 0 & 0 \end{pmatrix},\ 
R_{e_3}:=\begin{pmatrix} 0 & 0 &1\\ 0 &1 & 0 \\ 1& 0 & 0 \end{pmatrix} 
 \right\rangle.$$
\noindent Let us note that $L_{e_3}L_{e_3}=T_{e_1}$ and  give a list    
$T_{e_i}^\varepsilon,$  
$T_{e_i}^\delta,$ 
$T_{e_i}^{\varepsilon \delta},$ 
$D_{e_2,e_3}^\varepsilon,$  
$D_{e_2,e_3}^\delta,$ 
$D_{e_2,e_3}^{\varepsilon \delta},$
and 
$V_{e_i,e_j}$   below:

\begin{longtable}{|lcl|lclcl|}
\hline
   $T_{e_1}(e_{1})$&$=$&$e_1$&   $T_{e_1}^\varepsilon$&$=$&$T_{e_1}-T_{T_{e_1}(e_1)+\overline{T_{e_1}(e_1)}}$&$=$ &$-T_{e_1}$\\

$\overline{T_{e_1}({e_1})}$&$=$&$e_1$&$T_{e_1}^\delta$&$=$&$T_{e_1}+R_{\overline{T_{e_1}({e_1})}}$&$=$&$2T_{e_1}$\\

$R_{e_1}$&$=$&$T_{e_1}$&$T_{e_1}^{\varepsilon\delta}$&$=$&$T_{e_1}^\varepsilon+R_{\overline{T_{e_1}^\varepsilon({e_1})}}$&$=$&$-2T_{e_1}$\\
\hline 

 \hline
    
   $T_{e_2}({e_1})$&$=$&$3{e_2} $&$ T_{e_2}^\varepsilon$&$=$&$T_{e_2}-T_{T_{e_2}({e_1})+\overline{T_{e_2}({e_1})}}$&$=$&$T_{e_2}$\\

$\overline{T_{e_2}({e_1})}$&$=$&$-3{e_2}$&   
$T_{e_2}^\delta $&$=$&$T_{e_2}+R_{\overline{T_{e_2}({e_1})}}$&$=$&$2D_{e_2,e_3}$\\
   
&&&$T_{e_2}^{\varepsilon\delta}$&$=$&$T_{e_2}^\varepsilon+R_{\overline{T_{e_2}^\varepsilon({e_1})}}$&$=$&$2D_{e_2,e_3}$\\

   \hline
    
\hline
$T_{e_3}({e_1})$&$=$&$3{e_3}$&$    T_{e_3}^\varepsilon$&$=$&$T_{e_3}-T_{T_{e_3}({e_1})+\overline{T_{e_3}({e_1})}}$&$=$&$T_{e_3}$\\

$\overline{T_{e_3}({e_1}})$&$=$&$-3{e_3}$&$ T_{e_3}^\delta$&$=$&$T_{e_3}+R_{\overline{T_{e_3}({e_1})}}$&$=$&$T_{e_3}-3R_{e_3}$\\
   
&&&$T_{e_3}^{\varepsilon\delta}$&$=$&$T_{e_3}^\varepsilon+R_{\overline{T_{e_3}^\varepsilon({e_1})}}$&$=$&$T_{e_3}-3R_{e_3}$\\
\hline

   $D_{e_2,e_3}(e_{1})$&$=$&$0$&   $D_{e_2,e_3}^\varepsilon$&$=$&$D_{e_2,e_3}^{\varepsilon\delta}=\overline{D_{e_2,e_3}}$&$=$&$D_{e_2,e_3}$\\
\hline 
\end{longtable}

\begin{longtable}{lcllcllcl}

$V_{{e_1},{e_1}}$&$=$&$T_{e_1},$&$ V_{{e_1},{e_2}}$&$=$&$-T_{e_2},$&$ V_{{e_1},{e_3}}$&$=$&$-T_{e_3},$\\ 
$V_{{e_2},{e_1}}$&$=$&$T_{e_2},$&$ V_{{e_2},{e_2}}$&$=$&$0, $&$ V_{{e_2},{e_3}}$&$=$&$\frac{1}{3}T_{e_2}-\frac{8}{3}D_{e_2,e_3},$\\
$V_{{e_3},{e_1}}$&$=$&$T_{e_3},$&$ V_{{e_3},{e_2}}$&$=$&$-\frac{1}{3}T_{e_2}+\frac{8}{3}D_{e_2,e_3},,$&$ V_{{e_3},{e_3}}$&$=$&$-T_{e_1}.$
\end{longtable}

Now we are ready to determine the multiplication table of $\mathcal F({\rm S}_2).$

\begin{enumerate}[(I)]
    \item First, we define   $[\mathcal F_0, \mathcal F_{1}+\mathcal F_{2}],$ i.e., $[E,(x,s)]=(E(x), E^\delta(s)).$

\begin{longtable}{lclclcl}
$[\varepsilon_6,\varepsilon_1]$&$=$&$  [T_{e_1}, ({e_1},0)]$&$=$&$ \big(T_{e_1}({e_1}), T_{e_1}^\delta(0) \big)$&$=$&$
({e_1},\  0) \ =\ \varepsilon_1;$\\

$[\varepsilon_6,\varepsilon_2]$&$=$&$[T_{e_1}, ({e_2},0)]$&$=$&$
(T_{e_1}({e_2}), T_{e_1}^\delta(0))$&$=$&$({e_2},\ 0)\ =\ \varepsilon_2;$\\

$[\varepsilon_6,\varepsilon_3]$&$=$&$[T_{e_1}, ({e_3},0)]$&$=$&$
(T_{e_1}({e_3}), T_{e_1}^\delta (0))$&$=$&$
({e_3},\ 0)\ =\  \varepsilon_3;$\\
$[\varepsilon_6,\varepsilon_4]$&$=$&$[T_{e_1}, (0,{e_2})]$&$=$&$
(T_{e_1}(0), T_{e_1}^\delta (e_2))$&$=$&$
(0,\ 2e_2) \ =\ 2\varepsilon_4;$\\
$[\varepsilon_6,\varepsilon_5]$&$=$&$[T_{e_1}, (0,{e_3})]$&$=$&$(T_{e_1}(0), T_{e_1}^\delta (e_3))$&$=$&$(0, \ 2e_3)\ =\ 2\varepsilon_5;$\\

$[\varepsilon_7,\varepsilon_1]$&$=$&$[T_{e_2}, ({e_1},0)]$&$=$&$
(T_{e_2}({e_1}), T_{e_2}^\delta (0))$&$=$&$(3e_2,\ 0) \ =\ 3\varepsilon_2;$\\

$[\varepsilon_7,\varepsilon_2]$&$=$&$[T_{e_2}, ({e_2},0)]$&$=$&$
(T_{e_2}({e_2}), T_{e_2}^\delta (0))$&$=$&$(0,\ 0)\ =\ 0;$\\

$[\varepsilon_7,\varepsilon_3]$&$=$&$[T_{e_2}, ({e_3},0)]$&$=$&$
(T_{e_2}({e_3}), T_{e_2}^\delta (0))$&$=$&$(-e_2,\ 0) \ =\ -\varepsilon_2;$\\
$[\varepsilon_7,\varepsilon_4]$&$=$&$[T_{e_2}, (0,{e_2})]$&$=$&$
(T_{e_2}(0),T_{e_2}^\delta ({e_2}))$&$=$&$(0,\  0) \ =\ 0;$\\
$[\varepsilon_7,\varepsilon_5]$&$=$&$[T_{e_2}, (0,{e_3})]$&$=$&$
(T_{e_2}(0),T_{e_2}^\delta ({e_3}))$&$=$&$(0, \ 2e_2) \ =\ 2\varepsilon_4;$\\
$[\varepsilon_8,\varepsilon_1]$&$=$&$[T_{e_3}, ({e_1},0)]$&$=$&$
(T_{e_3}({e_1}), T_{e_3}^\delta (0))$&$=$&$(3{e_3},\ 0)\  =\ 3\varepsilon_3;$\\

$[\varepsilon_8,\varepsilon_2]$&$=$&$[T_{e_3}, ({e_2},0)]$&$=$&$
(T_{e_3}({e_2}), T_{e_3}^\delta(0))$&$=$&$(e_2,\ 0)\ =\ \varepsilon_2;$\\

$[\varepsilon_8,\varepsilon_3]$&$=$&$[T_{e_3}, ({e_3},0)]$&$=$&$
(T_{e_3}({e_3}), T_{e_3}^\delta (0))$&$=$&$(3{e_1},\  0)\  =\ 3\varepsilon_1;$\\
$[\varepsilon_8,\varepsilon_4]$&$=$&$[T_{e_3}, (0,{e_2})]$&$=$&$
(T_{e_3}(0), T_{e_3}^\delta ({e_2}))$&$=$&$(0,\ -2e_2)\ =\ -2\varepsilon_4;$\\
$[\varepsilon_8,\varepsilon_5]$&$=$&$[T_{e_3}, (0,{e_3})]$&$=$&$
(T_{e_3}(0), T_{e_3}^\delta ({e_3}))$&$=$&$(0,\ 0)\ =\ 0;$\\

$[\varepsilon_9,\varepsilon_1]$&$=$&$[D_{e_2,e_3}, ({e_1},0)]$&$=$&$
(D_{e_2,e_3}({e_1}), D_{e_2,e_3}^\delta (0))$&$=$&$(0,\ 0)\ =\ 0;$\\

$[\varepsilon_9,\varepsilon_2]$&$=$&$[D_{e_2,e_3}, ({e_2},0)]$&$=$&$
(D_{e_2,e_3}({e_2}), D_{e_2,e_3}^\delta(0))$&$=$&$(0,\ 0)\ =\ 0;$\\

$[\varepsilon_9,\varepsilon_3]$&$=$&$[D_{e_2,e_3}, ({e_3},0)]$&$=$&$
(D_{e_2,e_3}({e_3}), D_{e_2,e_3}^\delta (0))$&$=$&$({e_2}, \ 0)\  =\ \varepsilon_2;$\\
$[\varepsilon_9,\varepsilon_4]$&$=$&$[D_{e_2,e_3}, (0,{e_2})]$&$=$&$
(D_{e_2,e_3}(0), D_{e_2,e_3}^\delta ({e_2}))$&$=$&$(0, \ 0)\ =\ 0;$\\
$[\varepsilon_9,\varepsilon_5]$&$=$&$[D_{e_2,e_3}, (0,{e_3})]$&$=$&$
(D_{e_2,e_3}(0), D_{e_2,e_3}^\delta ({e_3}))$&$=$&$(0,\  e_2) \ =\ \varepsilon_4.$
\end{longtable}

  \item Second, we define   $[\mathcal F_0, \mathcal F_{-2}+\mathcal F_{-1}],$ i.e., 
  $[{E}, \wideparen {(x,s)} ] =
\wideparen{({E}^\varepsilon(x), {E}^{\varepsilon\delta}(s))}.$

\begin{longtable}{lclclcl}
$[\varepsilon_6,\varepsilon_{10}]$&$=$&$[T_{e_1}, \wideparen{({e_1},0)}] $&$=$&$ \wideparen{(T_{e_1}^\varepsilon({e_1}),\  0)}$&$=$&$
\wideparen{(-T_{e_1}({e_1}), \ 0)}\ =\ \wideparen{(-{e_1}, \ 0)}
\ =\ -\varepsilon_{10};$ \\

$[\varepsilon_6,\varepsilon_{11}]$&$=$&$[T_{e_1}, \wideparen{({e_2},0)}] $&$=$&$ \wideparen{(T_{e_1}^\varepsilon({e_2}),\  0)}$&$=$&$
\wideparen{(-T_{e_1}({e_2}), \ 0)}\ =\ \wideparen{(-{e_2}, \ 0)}
\ =\ -\varepsilon_{11};$ \\

$[\varepsilon_6,\varepsilon_{12}]$&$=$&$[T_{e_1}, \wideparen{({e_3},0)}]$&$=$&$ \wideparen{(T_{e_1}^\varepsilon({e_3}),\  0)}$&$=$&$
\wideparen{(-T_{e_1}({e_3}), \ 0)}\ = \ \wideparen{(-{e_3}, \ 0)}
\ =\ -\varepsilon_{12};$ \\

$[\varepsilon_6,\varepsilon_{13}]$&$=$&$[T_{e_1}, \wideparen{(0,{e_2})} ]$&$=$&$
\wideparen{(0, T_{e_1}^{\varepsilon\delta} ({e_2}))} $&$=$&$\wideparen{(0,\ -2{e_2})} \ =\ -2\varepsilon_{13};$\\

$[\varepsilon_6,\varepsilon_{14}]$&$=$&$[T_{e_1}, \wideparen{(0,{e_3})} ]$&$=$&$
\wideparen{(0, T_{e_1}^{\varepsilon\delta} ({e_3}))} $&$=$&$
\wideparen{(0,\ -2{e_3})} \ =\ -2\varepsilon_{14};$\\

$[\varepsilon_7,\varepsilon_{10}]$&$=$&$[T_{e_2}, \wideparen{({e_1},0)}] $&$=$&$ \wideparen{(T_{e_2}^\varepsilon({e_1}),\  0)}$&$=$&$
\wideparen{(3{e_2}, \ 0)} \ =\ 3\varepsilon_{11};$\\

$[\varepsilon_7,\varepsilon_{11}]$&$=$&$[T_{e_2}, \wideparen{({e_2},0)}] $&$=$&$ \wideparen{(T_{e_2}^\varepsilon({e_2}),\  0)}$&$=$&$\wideparen{(0, \ 0)}
\ =\ 0;$ \\

$[\varepsilon_7,\varepsilon_{12}]$&$=$&$[T_{e_2},\wideparen{({e_3},0)}] $&$=$&$ \wideparen{(T_{e_2}^\varepsilon({e_3}),\  0)}$&$=$&$
\wideparen{(-e_2, \ 0)}
\ =\ -\varepsilon_{11};$\\

$[\varepsilon_7,\varepsilon_{13}]$&$=$&$[T_{e_2},\wideparen{(0,e_2)}] $&$=$&$ \wideparen{(0,T_{e_2}^{\varepsilon\delta} ({e_2}))}$&$=$&$\wideparen{(0, \ 0)}
\ =\ 0;$\\ 

$[\varepsilon_7,\varepsilon_{14}]$&$=$&$[T_{e_2},\wideparen{(0,e_3)}] $&$=$&$ \wideparen{(0,T_{e_2}^{\varepsilon\delta} ({e_3}))}$&$=$&$
\wideparen{(0, 2e_2)}
\ =\ 2\varepsilon_{13};$\\

$[\varepsilon_8,\varepsilon_{10}]$&$=$&$[T_{e_3}, \wideparen{({e_1},0)}] $&$=$&$ \wideparen{(T_{e_3}^\varepsilon({e_1}),\  0)}$&$=$&$
\wideparen{(3{e_3}, \ 0)}
\ =\ 3\varepsilon_{12};$\\

$[\varepsilon_8,\varepsilon_{11}]$&$=$&$[T_{e_3}, \wideparen{({e_2},0)}] $&$=$&$ \wideparen{(T_{e_3}^\varepsilon({e_2}),\  0)}$&$=$&$\wideparen{(e_2, \ 0)}
\ =\ \varepsilon_{11};$\\

$[\varepsilon_8,\varepsilon_{12}]$&$=$&$[T_{e_3}, \wideparen{({e_3},0)}] $&$=$&$ \wideparen{(T_{e_3}^\varepsilon({e_3}),\  0)}$&$=$&$\wideparen{(3{e_1}, \ 0)}
\ =\ 3\varepsilon_{10};$\\

$[\varepsilon_8,\varepsilon_{13}]$&$=$&$[T_{e_3}, \wideparen{(0,e_2)}] $&$=$&$ \wideparen{(0,T_{e_3}^{\varepsilon\delta} ({e_2}))}$&$=$&$
\wideparen{(0, \ -2e_2)}
\ =\ -2\varepsilon_{13};$\\

$[\varepsilon_8,\varepsilon_{14}]$&$=$&$[T_{e_3}, \wideparen{(0,e_3)}] $&$=$&$ \wideparen{(0,T_{e_3}^{\varepsilon\delta} ({e_3}))}$&$=$&$
\wideparen{(0, \ 0)} \ =\ 0;$\\

$[\varepsilon_9,\varepsilon_{10}]$&$=$&$[D_{e_2,e_3}, \wideparen{({e_1},0)}] $&$=$&$\wideparen{(D_{e_2,e_3}^\varepsilon({e_1}),\  0)}$&$=$&$
\wideparen{(0, \ 0)}\ =\ 0;$\\

$[\varepsilon_9,\varepsilon_{11}]$&$=$&$[D_{e_2,e_3}, \wideparen{({e_2},0)}] $&$=$&$ \wideparen{(D_{e_2,e_3}^\varepsilon({e_2}),\  0)}$&$=$&$
\wideparen{(0, \ 0)}\ =\ 0;$\\

$[\varepsilon_9,\varepsilon_{12}]$&$=$&$[D_{e_2,e_3}, \wideparen{({e_3},0)}] $&$=$&$ \wideparen{(D_{e_2,e_3}^\varepsilon({e_3}),\  0)}$&$=$&$
\wideparen{({e_2}, \ 0)}\ =\ \varepsilon_{11};$\\

$[\varepsilon_9,\varepsilon_{13}]$&$=$&$[D_{e_2,e_3}, \wideparen{(0,e_2)}] $&$=$&$ \wideparen{(0,D_{e_2,e_3}^{\varepsilon\delta} ({e_2}))}$&$=$&$
\wideparen{(0, \ 0)} \ =\ 0;$\\

$[\varepsilon_9,\varepsilon_{14}]$&$=$&$[D_{e_2,e_3}, \wideparen{(0,e_3)}] $&$=$&$ \wideparen{(0, D_{e_2,e_3}^{\varepsilon\delta} ({e_3}))}$&$=$&$
\wideparen{(0, \ e_2)}\ =\ \varepsilon_{13}.$\\

\end{longtable}

\item Third, we define   $[ \mathcal F_{1}+\mathcal F_{2}, \mathcal F_{1}+\mathcal F_{2}],$ i.e., 
$[(x,s), (y,r)] = (0, x\overline{y} - y\overline{x}).$

\begin{longtable}{lclcl}

$[\varepsilon_1,\varepsilon_2]$&$=$&$ [({e_1},0), ({e_2},0)] $&$=$&$(0,\ -2e_2) \ =\ -2\varepsilon_4;$\\
$[\varepsilon_1,\varepsilon_3]$&$=$&$[({e_1},0), ({e_3},0)] $&$=$&$(0,\ -2e_3) \ =\ -2\varepsilon_5;$\\

$[\varepsilon_2,\varepsilon_3]$&$=$&$[({e_2},0), ({e_3},0)] $&$=$&$(0,\ -2e_2) \ =\ -2\varepsilon_4;$\\

\end{longtable}

\item Fourth, we define  $[ \mathcal F_{-2}+\mathcal F_{-1}, \mathcal F_{-2}+\mathcal F_{-1}],$ i.e., 
$[\wideparen{(x,s)}, \wideparen{(y,r)} ] = \wideparen{(0, x\overline{y} - y\overline{x})}.$

\begin{longtable}{lclcl}

$[\varepsilon_{10},\varepsilon_{11}]$&$=$&$[\wideparen{({e_1},0)}, \wideparen{({e_2},0)}] $&$=$&$\wideparen{(0,\ -2e_2)} \ =\ -2e_{13};$\\
$[\varepsilon_{10},\varepsilon_{12}]$&$=$&$[\wideparen{({e_1},0)}, \wideparen{({e_3},0)}] $&$=$&$\wideparen{(0,\ -2e_3)} \ =\ -2\varepsilon_{14};$\\

$[\varepsilon_{11},\varepsilon_{12}]$&$=$&$[\wideparen{({e_2},0)}, \wideparen{({e_3},0)}] $&$=$&$\wideparen{(0,\ -2e_2)}\  =\ -2\varepsilon_{13}.$\\

\end{longtable}

\item Fifth, we define   $[ \mathcal F_{1}+\mathcal F_{2}, \mathcal F_{-2}+\mathcal F_{-1}],$ i.e., 
    $[(x,s), \wideparen{(y,r)}] = -\wideparen{(r x, 0)}  + V_{x,y} + L_s L_r+ (s y, 0).$

\begin{longtable}{lclcl}
    
$[\varepsilon_1,\varepsilon_{10}]$&$=$&$  [({e_1},0),\wideparen{({e_1},0)}]$&$=$&$
V_{{e_1},{e_1}}\ =\ T_{e_1}\ =\ \varepsilon_6;$\\
$[\varepsilon_1,\varepsilon_{11}]$&$=$&$  [({e_1},0),\wideparen{({e_2},0)}]$&$=$&$
V_{{e_1},{e_2}}\ =\ -T_{e_2}\ =\ -\varepsilon_7;$\\
$[\varepsilon_1,\varepsilon_{12}]$&$=$&$[({e_1},0),\wideparen{({e_3},0)}]$&$=$&$
V_{{e_1},{e_3}}\ =\ -T_{e_3}\ =\ -\varepsilon_8;$\\
$[\varepsilon_1,\varepsilon_{13}]$&$=$&$[({e_1},0),\wideparen{(0,{e_2})}]$&$=$&$-\wideparen{({e_2},0)} \ =\ -\varepsilon_{11};$\\
$[\varepsilon_1,\varepsilon_{14}]$&$=$&$[({e_1},0),\wideparen{(0,{e_3})}]$&$=$&$-\wideparen{({e_3},0)} \ =\ -\varepsilon_{12};$\\

$[\varepsilon_2,\varepsilon_{10}]$&$=$&$ [({e_2},0),\wideparen{({e_1},0)}]$&$=$&$V_{{e_2},{e_1}}\ =\ T_{e_2}\ =\ \varepsilon_7;$\\

$[\varepsilon_2,\varepsilon_{11}]$&$=$&$  [({e_2},0),\wideparen{({e_2},0)}]$&$=$&$
V_{{e_2},{e_2}}\ =\ 0;$\\
$[\varepsilon_2,\varepsilon_{12}]$&$=$&$[({e_2},0),\wideparen{({e_3},0)}]$&$=$&$V_{{e_2},{e_3}}\ =\ 
\frac{1}{3}T_{e_2}-\frac{8}{3}D_{e_2,e_3}\ =\ 
\frac{1}{3}\varepsilon_7-\frac{8}{3}\varepsilon_9;$\\

$[\varepsilon_2,\varepsilon_{13}]$&$=$&$[({e_2},0),\wideparen{(0,{e_2})}]$&$=$&$
-\wideparen{({e_2}{e_2},0)}\ =\ -\wideparen{(0,\ 0)} \ =\ 0;$\\
$[\varepsilon_2,\varepsilon_{14}]$&$=$&$[({e_2},0),\wideparen{(0,{e_3})}]$&$=$&$
-\wideparen{({e_3}{e_2},\ 0)}\ =\ \wideparen{(e_2,\ 0)}\ =\ \varepsilon_{11};$\\

$[\varepsilon_3,\varepsilon_{10}]$&$=$&$[({e_3},0),\wideparen{({e_1},0)}]$&$=$&$
V_{{e_3},{e_1}}\ =\ T_{e_3}\ =\ \varepsilon_8;$\\
$[\varepsilon_3,\varepsilon_{11}]$&$=$&$[({e_3},0),\wideparen{({e_2},0)}]$&$=$&$
V_{{e_3},{e_2}}\ =\ -\frac{1}{3}T_{e_2}+\frac{8}{3}D_{e_2,e_3}\ =\ 
-\frac{1}{3}\varepsilon_7+\frac{8}{3}\varepsilon_9;$\\
$[\varepsilon_3,\varepsilon_{12}]$&$=$&$[({e_3},0),\wideparen{({e_3},0)}]$&$=$&$
V_{{e_3},{e_3}}\ =\ -T_{e_1}\ =\ -\varepsilon_6;$\\

$[\varepsilon_3,\varepsilon_{13}]$&$=$&$[({e_3},0),\wideparen{(0,{e_2})}]$&$=$&$-\wideparen{({e_2}{e_3},0)}\ =\ -\wideparen{({e_2},\ 0)} \ =\ -\varepsilon_{11};$\\
$[\varepsilon_3,\varepsilon_{14}]$&$=$&$[({e_3},0),\wideparen{(0,{e_3})}]$&$=$&$
-\wideparen{({e_3}{e_3},\ 0)} \ =\ -\wideparen{({e_1},\ 0)} \ =\ -\varepsilon_{10};$\\

$[\varepsilon_4,\varepsilon_{10}]$&$=$&$[(0,{e_2}),\wideparen{({e_1},0)}]$&$=$&$
({e_2},\ 0) \ =\ \varepsilon_2;$\\
 
$[\varepsilon_4,\varepsilon_{11}]$&$=$&$[(0,{e_2}),\wideparen{({e_2},0)}]$&$=$&$
({e_2}{e_2},\ 0)\ =\ (0,\ 0) \ =\ 0;$\\
$[\varepsilon_4,\varepsilon_{12}]$&$=$&$[(0,{e_2}),\wideparen{({e_3},0)}]$&$=$&$
({e_2}{e_3},\ 0)\ =\ ({e_2},\ 0) \  =\ \varepsilon_2;$\\
$[\varepsilon_4,\varepsilon_{13}]$&$=$&$[(0,{e_2}),\wideparen{(0,{e_2})}]$&$=$&$
L_{e_2}L_{e_2}\ =\ 0;$\\
$[\varepsilon_4,\varepsilon_{14}]$&$=$&$[(0,{e_2}),\wideparen{(0,{e_3})}]$&$=$&$L_{e_2}L_{e_3}\ =\ \frac{1}{3}T_{e_2}+\frac{4}{3}D_{e_2,e_3}\ =\ \frac{1}{3}\varepsilon_7+\frac{4}{3}\varepsilon_9;$\\

$[\varepsilon_5,\varepsilon_{10}]$&$=$&$[(0,{e_3}),\wideparen{({e_1},0)}]$&$=$&$({e_3},\ 0)  \ =\ \varepsilon_3;$\\
 
$[\varepsilon_5,\varepsilon_{11}]$&$=$&$[(0,{e_3}),\wideparen{({e_2},0)}]$&$=$&$({e_3}{e_2},0)\ =\ (-e_2,\ 0) \ =\ -\varepsilon_2;$\\
$[\varepsilon_5,\varepsilon_{12}]$&$=$&$[(0,{e_3}),\wideparen{({e_3},0)}]$&$=$&$({e_3}{e_3},\ 0)\ =\ ({e_1},\ 0)\ =\ \varepsilon_1;$\\
$[\varepsilon_5,\varepsilon_{13}]$&$=$&$[(0,{e_3}),\wideparen{(0,{e_2})}]$&$=$&$
L_{e_3}L_{e_2}\ =\ -\frac{1}{3}T_{e_2}-\frac{4}{3}D_{e_2,e_3}\ =\ 
-\frac{1}{3}\varepsilon_7-\frac{4}{3}\varepsilon_9;$\\
$[\varepsilon_5,\varepsilon_{14}]$&$=$&$[(0,{e_3}),\wideparen{(0,{e_3})}]$&$=$&$L_{e_3}L_{e_3}\ =\ T_{e_1}\ =\ \varepsilon_6;$\\
\end{longtable}

\item End, we define   $[ \mathcal F_{0}, \mathcal F_{0}].$

\begin{longtable}{lclcl}
$[\varepsilon_6,\varepsilon_{7}]$&$=$&$[T_{e_1}, T_{e_2}]$&$=$&$0 ;$ \\
$[\varepsilon_6,\varepsilon_{8}]$&$=$&$[T_{e_1}, T_{e_3}]$&$=$&$0 ;$ \\
$[\varepsilon_6,\varepsilon_{9}]$&$=$&$[T_{e_1}, D]$&$=$&$0 ;$ \\
$[\varepsilon_7,\varepsilon_{8}]$&$=$&$[T_{e_2},T_{e_3}]$&$=$&$-2T_{e_2}+8D_{e_2,e_3}\ =\ -2\varepsilon_7+8\varepsilon_9;$\\
$[\varepsilon_7,\varepsilon_{9}]$&$=$&$[T_{e_2}, D_{e_2,e_3}]$&$=$&$0;$ \\

$[\varepsilon_8,\varepsilon_{9}]$&$=$&$[T_{e_3},D_{e_2,e_3}]$&$=$&$-T_{e_2}\ =\ -\varepsilon_7.$\\

\end{longtable}

\end{enumerate}

\end{proof}

\begin{remark}
$\mathcal{F}({\rm S}_2)$ is perfect and 
$\mathcal{F}({\rm S}_2)=S \ltimes R,$ where
\begin{itemize}

\item $S=\big\langle\varepsilon_1, \ 
\varepsilon_3, \ 
\varepsilon_5, \ 
\varepsilon_6, \ 
\varepsilon_8, \ 
\varepsilon_{10}, \ 
\varepsilon_{12}, \ 
\varepsilon_{14}\big\rangle \cong \mathfrak{sl}_3;$

\item $
R = \big\langle\varepsilon_2, \ 
\varepsilon_4, \ 
\varepsilon_7, \ 
\varepsilon_9, \ 
\varepsilon_{11}, \ 
\varepsilon_{13}\big\rangle$ is 
 the abelian radical.\end{itemize}

\end{remark}

\end{document}